\documentclass[a4paper, 12pt]{amsart}
\usepackage[cp1251]{inputenc}
\usepackage[ukrainian,russian]{babel}
\usepackage{amsmath,amssymb,a4wide,latexsym,amscd,amsthm}
\usepackage{graphics}

\theoremstyle{plain}
\newtheorem{theorem}{Теорема}[section]
\newtheorem{lemma}{Лемма}[section]
\newtheorem{statement}{Утверждение}[section]
\newtheorem{conclusion}{Следствие}[section]

\theoremstyle{definition}
\newtheorem{example}{Пример}[section]
\newtheorem{definition}{Определение}[section]
\newtheorem{remark}{Замечание}

\numberwithin{equation}{section}

\newcommand{\ve}{\varepsilon}
\newcommand{\ra}{\rightarrow}
\newcommand{\bi}{\bigoplus}
\newcommand{\wt}{\widetilde}
\newcommand{\ol}{\overline}

\newcommand{\ls}{\leqslant}
\newcommand{\gs}{\geqslant}
\newcommand{\tmd}{\textmd}
\newcommand{\mc}{\mathcal}
\newcommand{\tbf}{\textbf}
\newcommand{\upr}{\upharpoonright}

\begin{document}

\title{О замкнутости суммы $n$ подпространств гильбертова пространства}
\author{И.С.Фещенко}

\begin{abstract}
В работе изучаются необходимые и достаточные условия
для того, чтобы сумма подпространств $H_{1},\ldots,H_{n}\,(n \gs 2)$
гильбертова пространства $H$ была подпространством, а также различные
свойства $n$-ок подпространств с замкнутой суммой.
\end{abstract}
\maketitle

\section{Введение}
Изучение систем $L=(V;V_{1},...,V_{n})$ $n$ подпространств $V_{1},...,V_{n}$
линейного пространства $V$, в частности, описание неразложимых четвёрок
подпространств $V$
(с точностью до эквивалентности), описание неразложимых представлений в пространстве
$V$ конечных частично упорядоченных множеств и т.д. являются классическими задачами
алгебры
(см., например, библиографию в \cite{SamStr}).

Пусть $H$ "--- комплексное гильбертово пространство, а $H_{i},\,1\ls i\ls n$ "--- набор его подпространств.
Изучение
системы подпространств $S=(H;H_{1},\ldots,H_{n})$ гильбертова пространства $H$
(или, что тоже самое, наборов соответствующих ортопроекторов
$P_{1},\ldots,P_{n}$) является
важной задачей
функционального анализа,
которой посвящены многочисленные публикации (см. например
\cite{SamStr} и библиографию там).

Пусть $H$ "--- комплексное гильбертово пространство, $H_{1},\ldots,H_{n}$ "---
набор его подпространств. Если
$\dim H<\infty$, то сумма $H_{1}+\ldots+H_{n}$ замкнута в $H$. В бесконечномерном гильбертовом
пространстве это утверждение, вообще говоря, неверно (даже при $n=2$ сумма
$H_{1}+H_{2}$ может не быть замкнутой). Поэтому естественной есть задача про нахождение
необходимых и достаточных условий (или только достаточных), при которых $H_{1}+\ldots+H_{n}$ есть подпространством $H$
(см. например \cite{Kober,Davis,newcriterion,Spitkovsky94,
BjoMan,IBAP,Lang,SchSmTs,Friedrichs,Grobler,Svensson,HalmosProbl}).

Системы подпространств с замкнутой суммой имеют многочисленные приложения:
к построению статистических оценок (см. например \cite{probability1}),
к задачам квадратичного программирования (\cite{SchSmTs}), томографии (\cite{Lang}, \cite{Svensson}),
паралельным вычислениям (\cite{BjoMan}),
алгоритмам для решения выпуклых задач существования (см. \cite{BauBor-main} и библиографию там),
изучению сходимости произведений ортопроекторов (итерационного процесса) (см. \cite{Browder};\cite{Hildebrandt}
и библиографию там),
методам Шварца (численные методы, которые применяются, например, для численного
решения дифференциальных уравнений в частных производных (методы разбиения области))
(см.~\cite{Holst} и библиографию там),
в подразделе~\ref{SS:IBAP} показана связь задачи про замкнутость суммы подпространств
со свойством обратного наилучшего приближения системы подпространств (inverse best approximation property), которое имеет
приложения
к периодическим проекционным алгоритмам, гармоническому анализу, интегральным уравнениям,
вейвлетам (см.~\cite{IBAP}).

В разделах~\ref{S:pairsubspaces}-\ref{S:Imagesofoperators}
приведены необходимые и достаточные условия
для того, чтобы сумма подпространств $H_{1},\ldots,H_{n}$
гильбертова пространства $H$ была подпространством, а также различные
свойства $n$-ок подпространств с замкнутой суммой.
Особое внимание уделено линейно независимым системам подпространств
(см. раздел~\ref{S:linind}).

В разделе
\ref{S:pairsubspaces} изучается задача о замкнутости суммы пары подпространств.
Основным рабочим инструментом является спектральная теорема для пары ортопроекторов
(<<представление>> П. Халмоша для пары подпространств).

В разделе
\ref{S:nsubspacesto2} показано, как задача про замкнутость суммы $n$ подпространств
сводится к задаче про замкнутость суммы пары подпространств. Используя критерии замкнутости
суммы пары подпространств, мы получаем критерии замкнутости суммы
$n$ подпространств.

В разделе
\ref{S:linind} изучаются линейно независимые системы подпространств.
В частности, доказано, что произвольная система
$n$ подпространств может быть <<уменьшена>> до линейно независимой системы подпространств
с сохранением суммы подпространств (см. теоремы \ref{T:sumclotolinind},~\ref{T:umenshenie1}).

В разделе
\ref{S:Imagesofoperators} изучается более общий объект чем система подпространств "---
система образов линейных непрерывных операторов. Такой подход позволяет получить
новые критерии замкнутости суммы $n$ подпространств, изучить некоторые свойства сумм
$n$-ок подпространств $H$.

Автор искренне признателен своему научному руководителю Ю.С.Самойленко,
а также В.И.Рабановичу, А.В.Стрельцу за полезные советы и обсуждения
результатов, приведённых в работе. Автор выражает благодарность
С.Ф.Коляде за помощь при работе с литературой.

\textbf{Обозначения.}
В данной работе мы рассматриваем комплексные гильбертовы пространства,
которые, как правило, мы обозначаем буквами $H,M,K$.
Отметим, что мы не накладываем дополнительных условий на размерность гильбертова
пространства. Для гильбертова пространства
$H$ $I_{H}$ "--- единичный оператор в $H$ (или просто
$I$, если понятно, о каком гильбертовом пространстве идет речь); если
$A:H\to H$ "--- линейный непрерывный оператор, то
$\sigma(A)$ обозначает спектр оператора $A$.
Некоторые ортогональные
разложения гильбертовых пространств, рассмотреные в работе, могут содержать
нулевые компоненты. Спектр оператора, определённого на такой компоненте, считаем
равным пустому множеству.

\section{Замкнутость суммы пары подпространств}\label{S:pairsubspaces}

Задаче о замкнутости суммы пары подпространств посвящены многочисленные
публикации (см. например \cite{newcriterion,Davis,Friedrichs,Grobler,Kober,SchSmTs,Spitkovsky94}),
некоторые критерии замкнутости суммы пары подпространств
уже стали математическим фольклором.

Для получения критериев замкнутости суммы пары подпространств
мы будем использовать спектральную теорему для пары ортопроекторов (<<представление>> П.Халмоша
для пары подпространств).
Многочисленные применения спектральной теоремы для пары ортопроекторов содержатся
в \cite{Spitcovsky,Halmos}. Критерии замкнутости суммы пары подпространств, сформулированые
в терминах обобщённых обратных операторов Мура-Пенроуза, содержатся в \cite{newcriterion}.

Отметим, что задача про замкнутость суммы пары подпространств в некотором смысле является
базовой "--- в разделе \ref{S:nsubspacesto2} мы покажем, как задача про замкнутость суммы $n$ подпространств сводится к
задаче про замкнутость суммы пары подпространств.

\subsection{Спектральная теорема для пары ортопроекторов}\label{SS:spectheorem}

Пусть $H$ "--- комплексное гильбертово пространство, $H_{1}$ и $H_{2}$ "--- подпространства $H$.
Будем говорить, что они находятся в общем положении (по Халмошу) (см. \cite{Halmos}),
если
\begin{equation*}
H_{1}\cap H_{2}=H_{1}\cap H_{2}^{\bot}
=H_{1}^{\bot}\cap H_{2}=H_{1}^{\bot}\cap H_{2}^{\bot}=0.
\end{equation*}

\begin{theorem}\label{T:Halmos1}\cite{Halmos}
Пусть $H_{1}$ и $H_{2}$ "--- подпространства $H$, которые
находятся в общем положении, $P_{1},P_{2}$ "--- соответствующие ортопроекторы.
Тогда найдутся гильбертово пространство $K$ и ограниченный самосопряженный оператор
$a:K\to K$, $0\leq a=a^{*}\leq I_{K}$, $\ker(a)=\ker(I_{K}-a)=0$ такие, что
$H=K\oplus K$,
а блочные разложения $P_{1}$ и $P_{2}$ имеют вид
\begin{equation*}
P_{1}=\begin{pmatrix}
   I_{K}&0 \\
   0&0
   \end{pmatrix},
P_{2}=\begin{pmatrix}
   a & \sqrt{a(I_{K}-a)}\\
   \sqrt{a(I_{K}-a)}& I_{K}-a
   \end{pmatrix}.
\end{equation*}

Наоборот, каждому такому оператору $a$ соответствуют определенные выше ортопроекторы $P_{1}$ и $P_{2}$
на подпространства $H_{1}$ и $H_{2}$
гильбертова пространства $H=K\oplus K$, которые
находятся в общем положении.
\end{theorem}

Пусть теперь $H_{1}$ и $H_{2}$ "--- произвольные подпространства $H$. Тогда можем записать:
\begin{equation}\label{E:rozkladH1}
H=(H_{1}\cap H_{2})\oplus (H_{1}\cap H_{2}^{\bot})\oplus (H_{1}^{\bot}\cap H_{2})
\oplus (H_{1}^{\bot}\cap H_{2}^{\bot})\oplus\wt{H},
\end{equation}
\begin{equation}\label{E:rozkladH11}
H_{1}=(H_{1}\cap H_{2})\oplus (H_{1}\cap H_{2}^{\bot})\oplus 0\oplus 0\oplus\wt{H}_{1},
\end{equation}
\begin{equation}\label{E:rozkladH21}
H_{2}=(H_{1}\cap H_{2})\oplus 0\oplus (H_{1}^{\bot}\cap H_{2})\oplus 0\oplus\wt{H}_{2},
\end{equation}
где $\wt{H}_{1},\wt{H}_{2}$ "--- подпространства
$\wt{H}$, которые находятся в общем положении. Непосредственно из теоремы~\ref{T:Halmos1}
получаем следующую теорему.

\begin{theorem}\label{T:spectraltheorem}(спектральная теорема для пары ортопроекторов)
Пусть $H_{1}$ и $H_{2}$ "--- подпространства $H$, $P_{1},P_{2}$ "--- соответствующие
ортопроекторы.
Тогда найдется гильбертово пространство $K$
и ограниченный самосопряженный оператор $a:K\to K$, $0\leq a=a^{*}\leq I_{K}$, $\ker(a)=\ker(I_{K}-a)=0$
такие, что
\begin{equation}\label{E:rozkladH2}
H=(H_{1}\cap H_{2})\oplus (H_{1}\cap H_{2}^{\bot})
\oplus (H_{1}^{\bot}\cap H_{2})\oplus (H_{1}^{\bot}\cap H_{2}^{\bot})
\oplus \left(K \oplus K\right),
\end{equation}
а блочные разложения $P_{1}$ и $P_{2}$ имеют вид
\begin{equation}\label{E:rozkladP1}
P_{1}=I_{H_{1}\cap H_{2}}\oplus I_{H_{1}\cap H_{2}^{\bot}}\oplus 0_{H_{1}^{\bot}\cap H_{2}}\oplus 0_{H_{1}^{\bot}\cap H_{2}^{\bot}}\oplus
\begin{pmatrix}
   I_{K}&0 \\
   0&0
   \end{pmatrix},
\end{equation}
\begin{equation}\label{E:rozkladP2}
P_{2}=I_{H_{1}\cap H_{2}}\oplus 0_{H_{1}\cap H_{2}^{\bot}}\oplus I_{H_{1}^{\bot}\cap H_{2}}\oplus 0_{H_{1}^{\bot}\cap H_{2}^{\bot}}\oplus
\begin{pmatrix}
   a & \sqrt{a(I_{K}-a)}\\
   \sqrt{a(I_{K}-a)}& I_{K}-a
   \end{pmatrix}.
\end{equation}

Используя спектральное представление
самосопряженного оператора $a$
в виде спектрального интеграла по разложению
единицы
$E_{a}(\cdot)$ в $K$ на $(0,1)$, имеем:
\begin{align*}
&P_{1}=I_{H_{1}\cap H_{2}}\oplus I_{H_{1}\cap H_{2}^{\bot}}\oplus 0_{H_{1}^{\bot}\cap H_{2}}\oplus 0_{H_{1}^{\bot}\cap H_{2}^{\bot}}\oplus
\begin{pmatrix}
1&0\\
0&0
\end{pmatrix}
\otimes I_{K},\\
&P_{2}=I_{H_{1}\cap H_{2}}\oplus 0_{H_{1}\cap H_{2}^{\bot}}\oplus I_{H_{1}^{\bot}\cap H_{2}}\oplus 0_{H_{1}^{\bot}\cap H_{2}^{\bot}}\oplus
\int_{(0,1)}\begin{pmatrix}
x&\sqrt{x(1-x)}\\
\sqrt{x(1-x)}&1-x
\end{pmatrix}
\otimes dE_{a}(x),
\end{align*}
где интеграл сходится равномерно.

Это представление для пары подпространств
позволяет провести аналогию с двумерным случаем
$(\dim K=1).$ Действительно, существует и единственный
самосопряженный оператор
$\Theta:K\to K$, $0\ls \Theta\ls\frac{\pi}{2}I_{K}$,
$\ker(\Theta)=\ker(\frac{\pi}{2}I_{K}-\Theta)=0$
(угловой оператор пары подпространств) такой, что
$a=\cos^{2}\Theta$,
$I_{K}-a=\sin^{2}\Theta$.
Тогда ортопроекторы имеют вид:
\begin{align*}
&P_{1}=I_{H_{1}\cap H_{2}}\oplus I_{H_{1}\cap H_{2}^{\bot}}\oplus 0_{H_{1}^{\bot}\cap H_{2}}\oplus 0_{H_{1}^{\bot}\cap H_{2}^{\bot}}\oplus
\begin{pmatrix}
I_{K}&0\\
0&0
\end{pmatrix},\\
&P_{2}=I_{H_{1}\cap H_{2}}\oplus 0_{H_{1}\cap H_{2}^{\bot}}\oplus I_{H_{1}^{\bot}\cap H_{2}}\oplus 0_{H_{1}^{\bot}\cap H_{2}^{\bot}}\oplus
\begin{pmatrix}
\cos^{2}\Theta&\cos\Theta\sin\Theta\\
\cos\Theta\sin\Theta&\sin^{2}\Theta
\end{pmatrix}.
\end{align*}
\end{theorem}

Далее в разделе \ref{S:pairsubspaces} мы будем использовать обозначения
теоремы \ref{T:spectraltheorem}, а
также (\ref{E:rozkladH1}),(\ref{E:rozkladH11}),(\ref{E:rozkladH21}).

\subsection{Классические критерии замкнутости суммы пары подпространств как следствия спектральной
теоремы для пары ортопроекторов}\label{SS:classictwosubspaces}

В этом подразделе собраны критерии замкнутости суммы пары подпространств, которые
были получены многими авторами и относятся к математическому фольклору.
Наша цель "---
показать, что применяя спектральную теорему для пары ортопроекторов,
эти критерии доказываются несложно и единообразно.

\begin{statement}\label{ST:pair1}
Следующие условия равносильны:
\begin{enumerate}
\item
$H_{1}+H_{2}$ замкнуто,
\item
$1\notin\sigma(a)$,
\item
$\sigma(P_{1}P_{2})\cap(1-\ve,1)=\varnothing$ для некоторого $\varepsilon>0$,
\item
$\|P_{1}P_{2}-P_{H_{1}\cap H_{2}}\|<1$,
\item
$H_{1}^{\bot}+H_{2}^{\bot}$ замкнуто,
\item
$\tmd{Im}((I-P_{1})P_{2})$ замкнуто,
\item
$\tmd{Im}(I-P_{1}P_{2})$ замкнуто.
\end{enumerate}
\end{statement}
\begin{proof}
\textbf{1.} Сумма
$H_{1}+H_{2}$ замкнута тогда и только тогда, когда
$\wt{H}_{1}+\wt{H}_{2}$ замкнуто в
$\wt{H}=K\oplus K$. Поскольку
$\wt{H}_{1}=\{(x,0),\,x\in K\}$,
$\wt{H}_{2}=\{(\sqrt{a}x,\sqrt{I-a}x),\,x\in K\}$, то
$\wt{H}_{1}+\wt{H}_{2}=\{(x,\sqrt{I-a}y),\,x,y\in K\}$.
Так как
$\ker(I-a)=0$, то
$\wt{H}_{1}+\wt{H}_{2}$ плотно в
$K\oplus K$, и замкнуто тогда и только тогда, когда
$\tmd{Im}\sqrt{I-a}=K$. Это условие равносильно обратимости
$I-a$. Таким образом,
$H_{1}+H_{2}$ замкнуто тогда и только тогда, когда
$1\notin\sigma(a)$.

\textbf{2.}
Поскольку
$P_{1}P_{2}=I\oplus 0\oplus 0\oplus 0\oplus
\begin{pmatrix}
a&\sqrt{a(I-a)}\\
0&0
\end{pmatrix}$, то
$\sigma(P_{1}P_{2})$ совпадает с
$\sigma(a)$ с точностью до точек
$0,1$. Предположим, что
$1\notin\sigma(a)$. Тогда для некоторого
$\ve>0$
$\sigma(a)\subset[0,1-\ve]$, а поэтому
$\sigma(P_{1}P_{2})\cap(1-\ve,1)=\varnothing$. Наоборот, пусть
$\sigma(P_{1}P_{2})\cap(1-\ve,1)=\varnothing$ для некоторого $\varepsilon>0$. Тогда
$\sigma(a)\cap(1-\ve,1)=\varnothing$. Поскольку
$\ker(I-a)=0$ и изолированная точка спектра самосопряженного оператора является его собственным значением, то
$1\notin\sigma(a)$.

\textbf{3.}
Ясно, что
$P_{1}P_{2}P_{1}-P_{H_{1}\cap H_{2}}=0\oplus 0\oplus 0\oplus 0\oplus
\begin{pmatrix}
a&0\\
0&0
\end{pmatrix}$. Используя равенство $\|A\|^{2}=\|AA^{*}\|$, получим
\begin{equation*}
\|P_{1}P_{2}-P_{H_{1}\cap H_{2}}\|^{2}=\|P_{1}P_{2}P_{1}-P_{H_{1}\cap H_{2}}\|=\|a\|.
\end{equation*}
Поэтому
$1\notin\sigma(a)$ тогда и только тогда, когда
$\|P_{1}P_{2}-P_{H_{1}\cap H_{2}}\|<1$.

\textbf{4.}
Из формул
\eqref{E:rozkladP1}, \eqref{E:rozkladP2} следует, что
$\sigma((I-P_{1})(I-P_{2}))$ совпадает с
$\sigma(a)$ с точностью до точек $0,1$. Поэтому
$\sigma((I-P_{1})(I-P_{2}))$ совпадает с
$\sigma(P_{1}P_{2})$ с точностью до точек
$0,1$. Отсюда следует нужное утверждение.

\textbf{5.} Из формул \eqref{E:rozkladP1}, \eqref{E:rozkladP2} следует, что
\begin{equation*}
\tmd{Im}((I-P_{1})P_{2})=0\oplus 0\oplus H_{1}^{\bot}\cap H_{2}\oplus 0\oplus (0\oplus\tmd{Im}\sqrt{I-a}).
\end{equation*}
Поэтому $\tmd{Im}((I-P_{1})P_{2})$
замкнуто тогда и только тогда, когда
$\tmd{Im}\sqrt{I-a}$ замкнуто. Последнее равносильно
$1\notin\sigma(a)$.

\textbf{6.}
Обозначим
$\wt{P}_{1},\wt{P}_{2}$ ортопроекторы на
$\wt{H}_{1},\wt{H}_{2}$.
$\tmd{Im}(I-P_{1}P_{2})$ замкнут тогда и только тогда, когда
$\tmd{Im}(I-\wt{P}_{1}\wt{P}_{2})$ замкнут.
Поскольку
\begin{equation*}
\tmd{Im}(I-\wt{P}_{1}\wt{P}_{2})=\{((I-a)y-\sqrt{a(I-a)}z,z),\,y,z\in K\},
\end{equation*}
последнее условие равносильно замкнутости
$\tmd{Im}(I-a)$, т.е.
$1\notin\sigma(a)$.
\end{proof}

\begin{example}
Пусть $0=\tau_{0}<\tau_{1}<\ldots<\tau_{n}<\tau_{n+1}=1$,
а $P_{1},P_{2}$ удовлетворяют равенству
$\prod_{k=0}^{n+1}(P_{1}P_{2}P_{1}-\tau_{k}P_{1})=0$

Поскольку блочные разложения $P_{1},P_{2}$ имеют вид \eqref{E:rozkladP1}, \eqref{E:rozkladP2}, то
\begin{equation*}
P_{1}P_{2}P_{1}-\tau P_{1}=
(1-\tau)I\oplus -\tau I\oplus 0\oplus 0\oplus
\begin{pmatrix}
a-\tau I&0\\
0&0
\end{pmatrix}.
\end{equation*}
Поэтому
$\prod_{k=0}^{n+1}(a-\tau_{k}I)=0$,
т.е. $\sigma(a)\subset\{\tau_{1},\ldots,\tau_{n}\}$
(напомним, что $\ker(a)=\ker(I-a)=0$).
Поэтому $H_{1}+H_{2}$ "--- подпространство.
\end{example}

Дадим условие замкнутости суммы пары подпространств в
терминах угла (по Фридрихсу) между ними.

\begin{definition}(см.~\cite{Friedrichs})
Углом
$\gamma=\gamma(H_{1},H_{2})$, $0\ls\gamma\ls \pi/2$, между подпространствами
$H_{1},H_{2}$ назовём угол, определённый равенством
\begin{equation*}
\cos\gamma=\sup\{|(x,y)|,\,x\in H_{1}\ominus(H_{1}\cap H_{2}),\,\|x\|=1,\,
y\in H_{2}\ominus(H_{1}\cap H_{2}),\,
\|y\|=1\}.
\end{equation*}
Если
$H_{1}\subset H_{2}$ или
$H_{2}\subset H_{1}$, то полагаем
$\cos\gamma=0$, т.е. $\gamma=\pi/2$.
\end{definition}

Пусть
$x=(0,x_{2},0,0,z_{1},0)\in H_{1}\ominus(H_{1}\cap H_{2})$,
тогда
$\|x\|^{2}=\|x_{2}\|^{2}+\|z_{1}\|^{2}=1$. Далее, пусть
$y=(0,0,x_{3},0,\sqrt{a}z_{2},\sqrt{1-a}z_{2})\in H_{2}\ominus(H_{1}\cap H_{2})$,
тогда
$\|y\|^{2}=\|x_{3}\|^{2}+\|z_{2}\|^{2}=1$.
Тут векторы записаны покомпонентно относительно ортогонального разложения $H$
\eqref{E:rozkladH2}.
Тогда
$(x,y)=(z_{1},\sqrt{a}z_{2})$.
Отсюда легко получаем
$\cos\gamma=\sqrt{\|a\|}$. Выше мы доказали, что
$H_{1}+H_{2}$ замкнуто тогда и только тогда, когда
$\|a\|<1$. Таким образом, мы доказали следующее утверждение.

\begin{statement}\label{ST:pair2}
$H_{1}+H_{2}$ замкнуто тогда и только тогда, когда
$\gamma(H_{1},H_{2})>0$.
\end{statement}

\begin{remark}
В терминах углового оператора
$\Theta$ (см. теорему \ref{T:spectraltheorem}) угол
$\gamma=\min\{\lambda,\,\lambda\in \sigma(\Theta)\}$.
\end{remark}

Следующее утверждение касается линейно независимых пар подпространств, т.е. таких, что
$H_{1}\cap H_{2}=0$.

\begin{statement}\label{ST:pair3}
Следующие утверждения эквивалентны:
\begin{enumerate}
\item
$H_{1}\cap H_{2}=0$ и $H_{1}+H_{2}$ замкнуто,
\item
$\|P_{1}P_{2}\|<1$,
\item
существует
$\ve>0$, такое, что $|(x,y)|\leqslant 1-\ve$ для произвольных
$x\in H_{1},\,y\in H_{2},\,\|x\|=\|y\|=1$,
\item
существует
$\ve>0$, такое, что  $\|x+y\|^{2}\gs \ve(\|x\|^{2}+\|y\|^{2})$ для произвольных
$x\in H_{1},y\in H_{2}$,
\item
существует
$\ve>0$, такие, что $\|(I-P_{1})x\|\gs\ve\|x\|$ для произвольного
$x\in H_{2}$.
\end{enumerate}
\end{statement}
\begin{proof}
Равносильность
$(1),(2),(3)$ следует из утверждений
\ref{ST:pair1},~\ref{ST:pair2}.

$(3)\Leftrightarrow(4)$.

Пусть выполнено
$(3)$. Тогда для произвольных
$x\in H_{1},y\in H_{2}$ имеем:
$\|x+y\|^{2}=\|x\|^{2}+\|y\|^{2}+2\tmd{Re}(x,y)\gs\ve(\|x\|^{2}+\|y\|^{2})$.

Пусть выполнено
$(4)$. Рассмотрим произвольные
$x\in H_{1},y\in H_{2}$,
$\|x\|=\|y\|=1$. Для произвольных комплексных
$t_{1},t_{2}$ имеем:
$\|t_{1}x+t_{2}y\|^{2}\gs\ve(|t_{1}|^{2}+|t_{2}|^{2})$, т.е. эрмитова матрица
$\begin{pmatrix}
1-\ve &(x,y)\\
(y,x)&1-\ve
\end{pmatrix}$
неотрицательно определена. Поэтому
$|(x,y)|\ls 1-\ve$.

$(1)\Leftrightarrow(5)$.

Пусть выполнено
$(1)$. Тогда в ортогональном разложении
$H$ \eqref{E:rozkladH2} компонента
$H_{1}\cap H_{2}$ отсутствует. Пусть
$x=(0,y,0,\sqrt{a}z,\sqrt{I-a}z)\in H_{2}$, тогда
$\|x\|^{2}=\|y\|^{2}+\|z\|^{2}$. Поскольку
$(I-P_{1})x=(0,y,0,0,\sqrt{I-a}z)$, то
$\|(I-P_{1})x\|^{2}=\|y\|^{2}+((I-a)z,z)$. Так как
$1\notin\sigma(a)$, то для некоторого
$\ve>0$
$I-a\gs\ve^{2}I$, а тогда
$\|(I-P_{1})x\|\gs\ve\|x\|$.

Пусть выполнено
$(5)$. Очевидно,
$H_{1}\cap H_{2}=0$. Из предыдущих рассуждений
$I-a\gs\ve^{2}I$, откуда
$1\notin\sigma(a)$.
\end{proof}

Следующее утверждение доказано в \cite{Lang} для подпространств пространства Фреше.
Мы приведём простое доказательство с помощью спектральной теоремы для пары ортопроекторов.
Обозначим $S_{\infty}(H)$ множество компактных операторов в
$H$.

\begin{statement}\label{ST:compactproduct}
Если $P_{1}P_{2}$ компактный, то
$H_{1}+H_{2}$ замкнуто и
$P_{H_{1}+H_{2}}=P_{1}+P_{2}\pmod{S_{\infty}(H)}$.
\end{statement}
\begin{proof}
Поскольку
$P_{1}P_{2}$ компактный, то
$\dim (H_{1}\bigcap H_{2})<\infty$ и
$a\in S_{\infty}(H)$. Поскольку
$\ker(I-a)=0$, то
$1\notin\sigma(a)$, а поэтому сумма
$H_{1}+H_{2}$ замкнута. Имеем:
\begin{align*}
&P_{1}+P_{2}=2I_{H_{1}\cap H_{2}}\oplus I_{H_{1}\cap H_{2}^{\bot}}\oplus I_{H_{1}^{\bot}\cap H_{2}}\oplus 0_{H_{1}^{\bot}\cap H_{2}^{\bot}}\oplus
\begin{pmatrix}
   I_{K}+a&\sqrt{a(I_{K}-a)} \\
   \sqrt{a(I_{K}-a)}&I_{K}-a
   \end{pmatrix},\\
&P_{H_{1}+H_{2}}=I_{H_{1}\cap H_{2}}\oplus I_{H_{1}\cap H_{2}^{\bot}}\oplus I_{H_{1}^{\bot}\cap H_{2}}\oplus 0_{H_{1}^{\bot}\cap H_{2}^{\bot}}\oplus
\begin{pmatrix}
   I_{K}&0 \\
   0&I_{K}
   \end{pmatrix}.
\end{align*}
Поэтому
$P_{H_{1}+H_{2}}=P_{1}+P_{2}\pmod{S_{\infty}(H)}$.
\end{proof}

\subsection{Условие замкнутости $H_{1}+H_{2}$ в терминах свойств функций от $P_{1},P_{2}$}
\label{SS:functionofprojection}

Пусть
$f_{1}(x),f_{2}(x),f_{3}(x),f_{4}(x)$ "--- непрерывные на
$[0,1]$ комплекснозначные функции. Определим оператор
\begin{equation}\label{E:operatorb}
b=P_{1}f_{1}(P_{1}P_{2}P_{1})+P_{2}f_{2}(P_{2}P_{1}P_{2})+P_{1}P_{2}f_{3}(P_{2}P_{1}P_{2})+
P_{2}P_{1}f_{4}(P_{1}P_{2}P_{1}).
\end{equation}

\begin{remark}
Алгебра, порожденная ортопроекторами
$P_{1},P_{2}$ имеет вид
$\mc{A}(P_{1},P_{2})=\{b\}$,
где
$f_{1},f_{2},f_{3},f_{4}$
пробегают множество полиномов.
\end{remark}

Сужения оператора $b$ на компоненты ортогонального разложения $H$
\eqref{E:rozkladH2} обозначим
$b_{1,1},b_{1,0},b_{0,1},b_{0,0},\wt{b}$.
Спектр $\sigma(b)$ является объединением спектров операторов $b_{1,1},b_{1,0},b_{0,1},b_{0,0},\wt{b}$.

На компоненте $H_{1}\cap H_{2}$
$P_{1}=P_{2}=I$, а поэтому $b_{1,1}=(f_{1}(1)+f_{2}(1)+f_{3}(1)+f_{4}(1))I$.
На компоненте $H_{1}\cap H_{2}^{\bot}$
$P_{1}=I,P_{2}=0$, а поэтому $b_{1,0}=f_{1}(0)I$.
На компоненте $H_{1}^{\bot}\bigcap H_{2}$
$P_{1}=0,P_{2}=I$, а поэтому $b_{0,1}=f_{2}(0)I$.
На компоненте $H_{1}^{\bot}\bigcap H_{2}^{\bot}$
$P_{1}=P_{2}=0$, а поэтому $b_{0,0}=0$.
Рассмотрим компоненту
$K\oplus K$. Сужения ортопроекторов $P_{1},P_{2}$ на
$K\oplus K$ равны
\begin{equation*}
\wt{P}_{1}=\begin{pmatrix}
I&0\\
0&0
\end{pmatrix},
\wt{P}_{2}=\begin{pmatrix}
a&\sqrt{a(I-a)}\\
\sqrt{a(I-a)}&I-a
\end{pmatrix},
\end{equation*}
а поэтому
\begin{equation*}
\wt{P}_{1}\wt{P}_{2}=\begin{pmatrix}
a&\sqrt{a(I-a)}\\
0&0
\end{pmatrix},
\wt{P}_{2}\wt{P}_{1}=\begin{pmatrix}
a&0\\
\sqrt{a(I-a)}&0
\end{pmatrix}
\end{equation*}
и
\begin{equation*}
\wt{P}_{1}\wt{P}_{2}\wt{P}_{1}=\begin{pmatrix}
a&0\\
0&0
\end{pmatrix}=a\wt{P}_{1},\quad
\wt{P}_{2}\wt{P}_{1}\wt{P}_{2}=\begin{pmatrix}
a^{2}&a\sqrt{a(I-a)}\\
a\sqrt{a(I-a)}&a(I-a)
\end{pmatrix}=a\wt{P}_{2}.
\end{equation*}
Тут умножение на
$a$ означает умножение на оператор
$\begin{pmatrix}
a&0\\
0&a
\end{pmatrix}$.
Для непрерывной на
$[0,1]$ комплекснозначной функции $f(x)$ имеем:
\begin{equation*}
\wt{P}_{1}f(\wt{P}_{1}\wt{P}_{2}\wt{P}_{1})=f(a)\wt{P}_{1},\quad
\wt{P}_{2}f(\wt{P}_{2}\wt{P}_{1}\wt{P}_{2})=f(a)\wt{P}_{2}.
\end{equation*}
Поэтому
\begin{align*}
\wt{b}&=f_{1}(a)\wt{P}_{1}+f_{2}(a)\wt{P}_{2}+f_{3}(a)\wt{P}_{1}\wt{P}_{2}+f_{4}(a)\wt{P}_{2}\wt{P}_{1}=\\
&=\begin{pmatrix}
f_{1}(a)+af_{2}(a)+af_{3}(a)+af_{4}(a)&\sqrt{a(I-a)}(f_{2}(a)+f_{3}(a))\\
\sqrt{a(I-a)}(f_{2}(a)+f_{4}(a))&(I-a)f_{2}(a).
\end{pmatrix}
\end{align*}
Определим функции
$T(x)=f_{1}(x)+f_{2}(x)+x(f_{3}(x)+f_{4}(x))$ и
$D(x)=(1-x)(f_{1}(x)f_{2}(x)-xf_{3}(x)f_{4}(x))$, $x\in[0,1]$.
У операторной матрицы $2\times 2$, задающей
$\wt{b}$, компоненты коммутируют, её операторный след (сумма диагональных элементов)
равен $T(a)$, а операторный определитель $D(a)$.
Оператор
$\lambda I-\wt{b}$ не обратим тогда и только тогда, когда операторный определитель
операторной матрицы $2\times 2$, задающей
$\lambda I-\wt{b}$, не обратим (см.
\cite{HalmosProbl}, задача 55). Этот операторный определитель равен
$\lambda^{2}I-\lambda T(a)+D(a)$.
Поскольку функции
$T(x),D(x)$ непрерывны, по теореме об отображение спектра
$\sigma(\lambda^{2}I-\lambda T(a)+D(a))=\{\lambda^{2}-T(x)\lambda+D(x),\,x\in\sigma(a)\}.$
Таким образом,
$\sigma(\wt{b})$ "--- множество решений уравнения
$\lambda^{2}-T(x)\lambda+D(x)=0$, когда $x$ пробегает $\sigma(a)$.

Сумма
$H_{1}+H_{2}$ замкнута тогда и только тогда, когда
$1\notin\sigma(a)$. Теперь мы готовы сформулировать
критерий замкнутости $H_{1}+H_{2}$
в терминах спектра
$b$ (напомним, что $b$ задан формулой (\ref{E:operatorb})).
Определим функцию $F(x)=f_{1}(x)f_{2}(x)-xf_{3}(x)f_{4}(x),\,x\in[0,1]$.

\begin{statement}\label{ST:functionpair1}
Пусть функция
$F(x)$ не обращается в 0 на
$[0,1)$. Тогда:
\begin{enumerate}
\item
сумма
$H_{1}+H_{2}$ замкнута тогда и только тогда, когда существует
$\ve>0$, такое, что
$\sigma(b)\cap(\{z\in\mathbb{C},\,|z|<\ve\}\setminus\{0\})=\varnothing$;
\item
пусть дополнительно
$f_{1}(1)+f_{2}(1)+f_{3}(1)+f_{4}(1)\neq 0$; сумма
$H_{1}+H_{2}=H$ тогда и только тогда, когда оператор
$b$ обратим.
\end{enumerate}
\end{statement}
\begin{proof}
Докажем (1).

Предположим,
$1\notin\sigma(a)$. Тогда существует
$m_{1}>0$, для которого $|D(x)|\gs m_{1},\,x\in\sigma(a)$.
Существует $m_{2}$, такое, что для произвольного $x\in\sigma(a)$ $|T(x)|\ls m_{2}$.
Отсюда легко следует существование искомого
$\ve>0$.

Предположим,
$1\in\sigma(a)$. Поскольку изолированная точка спектра самосопряженного
оператора есть его собственным значением и
$\ker(I-a)=0$, то существует последовательность
$x_{j}\in\sigma(a),\,j\gs 1$, сходящаяся к 1, причём для всех
$j\gs 1$ $x_{j}<1$.
Теперь из теоремы о непрерывной зависимости корней полинома от его коеффициентов
следует существование последовательности
$\lambda_{j}\in\sigma(\wt{b})\subset\sigma(b)$, сходящейся к 0, причём для всех
$j\gs 1$ $\lambda_{j}\neq 0$.

Докажем (2).

Пусть $H_{1}+H_{2}=H$. Тогда
$H_{1}^{\bot}\cap H_{2}^{\bot}=0$. Спектры
$\sigma(b_{1,1})=f_{1}(1)+f_{2}(1)+f_{3}(1)+f_{4}(1)\neq 0$,
$\sigma(b_{1,0})=f_{1}(0)\neq 0$,
$\sigma(b_{0,1})=f_{2}(0)\neq 0$
(тут равенства написаны при условии, что соответствующий спектр непустой).
Поскольку
$1\notin\sigma(a)$, то
$0\notin\sigma(\wt{b})$. Поэтому
$b$ обратим.

Пусть
$b$ обратим.
Тогда $\tmd{Im}(b)=H$. Из определения
$b$ следует, что $\tmd{Im}(b)\subset H_{1}+H_{2}$. Поэтому
$H_{1}+H_{2}=H$.
\end{proof}

\begin{example}
Пусть
$f_{1}(x)=f_{2}(x)=1,f_{3}(x)=f_{4}(x)=0$.
Тогда оператор
$b=P_{1}+P_{2}$ неотрицателен. Сумма
$H_{1}+H_{2}$ замкнута тогда и только тогда, когда существует
$\ve>0$, такое, что
$\sigma(P_{1}+P_{2})\cap(0,\ve)=\varnothing$. Сумма
$H_{1}+H_{2}=H$ тогда и только тогда, когда существует
$\ve>0$, такое, что $P_{1}+P_{2}\gs\ve I$.
\end{example}

\begin{example}
Пусть
$f_{1}=\tau_{1},f_{2}=\tau_{2},f_{3}=f_{4}=0$, где
$\tau_{1},\tau_{2}$ "--- действительные числа, отличные от 0.
Используя выражение для спектра
$\tau_{1}P_{1}+\tau_{2}P_{2}$, условие
замкнутости $H_{1}+H_{2}$ можно сформулировать точнее, чем в утверждении \ref{ST:functionpair1}:
сумма $H_{1}+H_{2}$ "--- подпространство тогда и только тогда, когда
существует $\ve>0$ такое, что:
\begin{enumerate}
\item если
$\tau_{1}>0,\tau_{2}>0$ то $\sigma(\tau_{1}P_{1}+\tau_{2}P_{2})\cap(0,\ve)=\varnothing$
\item если
$\tau_{1}<0,\tau_{2}<0$ то $\sigma(\tau_{1}P_{1}+\tau_{2}P_{2})\cap(-\ve,0)=\varnothing$
\item если
$\tau_{1},\tau_{2}$ имеют разный знак и $\tau_{1}+\tau_{2}\gs 0$ то
$\sigma(\tau_{1}P_{1}+\tau_{2}P_{2})\cap(-\ve,0)=\varnothing$
\item если
$\tau_{1},\tau_{2}$ имеют разный знак и $\tau_{1}+\tau_{2}\ls 0$ то
$\sigma(\tau_{1}P_{1}+\tau_{2}P_{2})\cap(0,\ve)=\varnothing$.
\end{enumerate}
\end{example}

\begin{statement}\label{ST:functionpair2}
Пусть функция
$F(x)$ не обращается в 0 на $[0,1)$.
$H_{1}+H_{2}$ замкнуто тогда и только тогда, когда
$\tmd{Im}(b)$ замкнут.
\end{statement}
\begin{proof}
$\tmd{Im}(b)$ замкнут тогда и только тогда, когда
$\tmd{Im}(\wt{b})$ замкнут. Из теоремы Дугласа (см. \cite{Douglas},
а также раздел \ref{SS:Douglas} данной работы)
следует, что
$\tmd{Im}(\wt{b})=\tmd{Im}(\wt{b}(\wt{b})^{*})^{1/2}$.
Поэтому
$\tmd{Im}(b)$ является подпространством тогда и только тогда, когда для некоторого
$\ve>0$
$\sigma(\wt{b}(\wt{b})^{*})\cap(0,\ve)=\varnothing$.

\textbf{1.}
Пусть $H_{1}+H_{2}$ "--- подпространство, т.е.
$1\notin\sigma(a)$. Тогда оператор
$\wt{b}$ обратим, а поэтому
$\tmd{Im}(\wt{b})=\wt{H}$.

\textbf{2.}
Пусть теперь
$H_{1}+H_{2}$ не является подпространством, т.е.
$1\in\sigma(a)$. Существует последовательность
$x_{k}\in\sigma(a),\,x_{k}\to 1$, причём
$x_{k}<1,\,k\gs 1$. Операторный определитель блочной $2\times 2$
матрицы, задающей
$\wt{b}(\wt{b})^{*}$ равен
$D_{1}(a)$, где
$D_{1}(x)=(1-x)^{2}|F(x)|^{2}$, а её операторный след равен
$T_{1}(a)$ для некоторой (её явный вид нам не нужен) непрерывной на $[0,1]$ функции $T_{1}(x)$.
Спектр оператора $\wt{b}(\wt{b})^{*}$ есть множество решений
$\lambda$ уравнения $\lambda^{2}-T_{1}(x)\lambda+D_{1}(x)=0$, когда
$x$ пробегает $\sigma(a)$.
Теперь из теоремы про непрерывную зависимость корней полинома от его коеффициентов
следует существование последовательности
$\lambda_{j}\in\sigma(\wt{b}(\wt{b})^{*})$, сходящейся к 0, причём для всех
$j\gs 1$ $\lambda_{j}\neq 0$.
Поэтому
$\tmd{Im}(\wt{b})$ "--- не подпространство.
\end{proof}

\begin{statement}\label{ST:functionpair3}
Предположим, что $\tmd{Im}(\wt{b})=\wt{H}_{1}+\wt{H}_{2}$.
Тогда $H_{1}+H_{2}$ "--- подпространство.
\end{statement}
\begin{proof}
Пусть
$\tmd{Im}(\wt{b})=\wt{H}_{1}+\wt{H}_{2}$. Из теоремы Дугласа (см.\cite{Douglas}, а также раздел
\ref{SS:Douglas} данной работы) следует, что существует
$\ve>0$, такое, что
$\wt{b}(\wt{b})^{*}\gs\ve(\wt{P}_{1}+\wt{P}_{2})$.
Для
$x\in[0,1]$ определим следующие $(2\times 2)$-матрицы:
\begin{equation*}
\wt{b}(x)=
\begin{pmatrix}
f_{1}(x)+xf_{2}(x)+xf_{3}(x)+xf_{4}(x)&\sqrt{x(1-x)}(f_{2}(x)+f_{3}(x))\\
\sqrt{x(1-x)}(f_{2}(x)+f_{4}(x))&(1-x)f_{2}(x).
\end{pmatrix}
\end{equation*}
и
\begin{equation*}
\wt{P}_{1}(x)=\begin{pmatrix}
   1&0 \\
   0&0
   \end{pmatrix},
\wt{P}_{2}(x)=\begin{pmatrix}
   x & \sqrt{x(1-x)}\\
   \sqrt{x(1-x)}& 1-x
   \end{pmatrix}.
\end{equation*}
Тогда для каждого
$x\in\sigma(a)$
$\wt{b}(x)(\wt{b}(x))^{*}\gs\ve(\wt{P}_{1}(x)+\wt{P}_{2}(x))$, а поэтому
$\det(\wt{b}(x)(\wt{b}(x))^{*})\gs\ve^{2}\det(\wt{P}_{1}(x)+\wt{P}_{2}(x))$.
Поэтому
$(1-x)^{2}|F(x)|^{2}\gs\ve^{2}(1-x)$, т.е.
$(1-x)|F(x)|^{2}\gs\ve^{2}$
для всякого
$x\in\sigma(a)\setminus\{1\}$.
Поэтому
$1\notin\sigma(a)$, что означает замкнутость
$H_{1}+H_{2}$.
\end{proof}

\begin{conclusion}
Если
$\tmd{Im}(b)=H_{1}+H_{2}$, то
$H_{1}+H_{2}$ "--- подпространство.
\end{conclusion}

\begin{conclusion}
Пусть
$F(x)$ не обращается в 0 на $[0,1)$ и
$H_{1}\cap H_{2}=0$. Сумма
$H_{1}+H_{2}$ "--- подпространство тогда и только тогда, когда $\tmd{Im}(b)=H_{1}+H_{2}$.
\end{conclusion}

\begin{conclusion}
Пусть
$F(x)$ не обращается в 0 на $[0,1)$ и
$f_{1}(1)+f_{2}(1)+f_{3}(1)+f_{4}(1)\neq 0$. Сумма
$H_{1}+H_{2}$ "--- подпространство тогда и только тогда, когда $\tmd{Im}(b)=H_{1}+H_{2}$.
\end{conclusion}

\begin{conclusion}
Пусть
$F(x)$ не обращается в 0 на $[0,1)$. Сумма
$H_{1}+H_{2}$ "--- подпространство тогда и только тогда, когда
$\tmd{Im}(b)\supset (H_{1}+H_{2})\bigcap(H_{1}\cap H_{2})^{\bot}$.
\end{conclusion}

В дальнейшем нам понадобится свойство <<почти>> симметричности
$\sigma(b)$.
Из полученых формул для $\sigma(b)$ следует следующее
утверждение.

\begin{statement}\label{ST:pairsymmetry}
Пусть функция
$f_{1}(x)+f_{2}(x)+x(f_{3}(x)+f_{4}(x))=c$, $x\in[0,1]$.
Тогда
$\sigma(b)$ <<почти>> симметричен относительно точки $c/2$:
если
$\lambda\in\sigma(b)$ и
$\lambda\notin\{0,f_{1}(0),f_{2}(0),c\}$, то
$(c-\lambda)\in\sigma(b)$.
\end{statement}

\section{Сведение задачи про замкнутость суммы $n$ подпространств
к задаче про замкнутость суммы пары подпространств}\label{S:nsubspacesto2}

\subsection{}

Пусть
$H_{1},\ldots,H_{n}$ "--- подпространства
гильбертова пространства
$H$,
$P_{1},\ldots,P_{n}$ "--- соответствующие ортопроекторы.

Введём в рассмотрение гильбертово пространство
$X=\underbrace{H\oplus\ldots\oplus H}_{n}$.
Определим
в нём подпространства
$\Delta=\{(x,\ldots,x),\,x\in H\}$,
а также подпространство
$\wt{H}=H_{1}\oplus\ldots\oplus H_{n}$.
Соответствующие ортопроекторы имеют вид
\begin{equation*}
P_{\Delta}=\begin{pmatrix}
\frac{1}{n} I_{H}&\ldots&\frac{1}{n} I_{H}\\
\hdotsfor{3}\\
\frac{1}{n} I_{H}&\ldots&\frac{1}{n} I_{H}
\end{pmatrix},~
P_{\wt{H}}=diag(P_{1},\ldots,P_{n}).
\end{equation*}

Поскольку
\begin{equation*}
\Delta^{\bot}=\{(x_{1},\ldots,x_{n}),\,x_{k}\in H,\,1\ls k\ls n,\,\sum_{k=1}^{n}x_{k}=0\}
\end{equation*}
то
\begin{equation*}
\Delta^{\bot}+\wt{H}=\{(x_{1},\ldots,x_{n}),\,\sum_{k=1}^{n}x_{k}\in H_{1}+\ldots+H_{n}\}.
\end{equation*}
Отсюда вытекает следующее утверждение:

\begin{statement}
$\sum_{k=1}^{n}H_{k}$ замкнуто тогда и только тогда, когда
$\Delta^{\bot}+\wt{H}$ замкнуто.
$\sum_{k=1}^{n}H_{k}=H$ тогда и только тогда, когда
$\Delta^{\bot}+\wt{H}=X$.
\end{statement}

Для того, чтобы воспользоваться критерием замкнутости суммы пары подпространств,
рассмотрим оператор
\begin{equation*}
P_{\Delta}P_{\wt{H}}P_{\Delta}=\dfrac{1}{n^{2}}\begin{pmatrix}
\sum_{k=1}^{n}P_{k}&\ldots&\sum_{k=1}^{n}P_{k}\\
\hdotsfor{3}\\
\sum_{k=1}^{n}P_{k}&\ldots&\sum_{k=1}^{n}P_{k}
\end{pmatrix}.
\end{equation*}
Легко проверить, что
$\sigma(P_{\Delta}P_{\wt{H}}P_{\Delta})=0\cup\sigma\left(\dfrac{P_{1}+\ldots+P_{n}}{n}\right)$. Используя спектральную
теорему для пары ортопроекторов $P_{\Delta},P_{\wt{H}}$, несложно получить
$\sigma(P_{\Delta^{\bot}}P_{\wt{H}})=\{1-\alpha,\,\alpha\in\sigma\left(\dfrac{P_{1}+\ldots+P_{n}}{n}\right)\}$
с точностью до точек 0,1.
Используя утверждение \ref{ST:pair1}, получаем следующий критерий замкнутости суммы
$n$ подпространств.

\begin{statement}\label{ST:statement2}
$H_{1}+\ldots+H_{n}$ замкнуто тогда и только тогда, когда существует
$\ve>0$, такое, что
$\sigma(P_{1}+\ldots+P_{n})\cap(0,\ve)=\varnothing$.
\end{statement}

Из утверждения \ref{ST:statement2}, а также того, что изолированная точка спектра
самосопряженного оператора является его собственным значением, следует, что
$\sum_{k=1}^{n}H_{k}=H$ тогда и только тогда, когда
$\sum_{k=1}^{n}P_{k}$ обратим.

\begin{remark}
Утверждение \ref{ST:statement2} можно получить различными способами, например:\\
$(\textbf{a})$ используя теорему Р. Дугласа (см. раздел \ref{S:Imagesofoperators}),\\
$(\textbf{b})$
достаточно показать, что если
$H_{1}+\ldots+H_{n}=H$, то
$P_{1}+\ldots+P_{n}$ обратим.

Определим оператор
$A:\bi_{k=1}^{n}H_{k}\to H$ равенством
$A(x_{1},\ldots,x_{n})=\sum_{k=1}^{n}x_{k}$. Тогда
$\tmd{Im}(A)=H$, а поэтому
$A^{*}$ "--- изоморфное вложение (т.е. $\|A^{*}x\|\gs\ve\|x\|,\,x\in H$  для некоторого $\ve>0$),
$AA^{*}$ обратим. Поскольку
$A^{*}x=(P_{1}x,\ldots,P_{n}x),\,x\in H$, то оператор
$AA^{*}=P_{1}+\ldots+P_{n}$ обратим.
\end{remark}

Приведём некоторые достаточные условия для того, чтобы
$H_{1}+\ldots+H_{n}$ было замкнуто.

\begin{example}
Пусть
$\mc{A}=\mc{A}(P_{1},\ldots,P_{n})$ "--- алгебра, порожденная
$P_{1},\ldots,P_{n}$. Предположим, что
$\dim \mc{A}<\infty$. Тогда для элемента этой алгебры
$Q=P_{1}+\ldots+P_{n}$ существует ненулевой полином
$R(z)$, для которого
$R(Q)=0$. Поэтому
$\sigma(Q)$ состоит из конечного числа точек, а поэтому
$H_{1}+\ldots+H_{n}$ "--- подпространство.

В качестве примеров приведём такие системы подпространств:
\begin{enumerate}
\item пусть ортопроекторы
$P_{1},\ldots,P_{n}$ попарно коммутируют.
Тогда
$\dim \mc{A}\ls 2^{n}-1$. В этом случае
$\sigma(P_{1}+\ldots+P_{n})\in\{0,1,\ldots,n\}$.
\item
Пусть система
$H_{1},\ldots,H_{n}$ есть простой $n$-кой подпространств, связанной с
деревом
$\mathbb{G}$ (см. \cite{SamStr}). Тогда
$\dim \mc{A}\ls n^{2}$.
\end{enumerate}
\end{example}

Теперь, используя критерии замкнутости суммы пары подпространств, можем получить
критерии замкнутости суммы $n$ подпространств. Начнём с примера.

\begin{example}
Сумма
$H_{1}+\ldots+H_{n}=H$ тогда и только тогда, когда существует
$\ve>0$, такое, что
$P_{\Delta^{\bot}}+P_{\wt{H}}\gs\ve I_{X}$. Возьмём
$x=(x_{1},\ldots,x_{n})\in X$. Тогда условие
$((P_{\Delta^{\bot}}+P_{\wt{H}})x,x)\gs\ve\|x\|^{2}$
принимает вид
\begin{equation*}
\sum_{k=1}^{n}\|P_{k}x_{k}\|^{2}\gs\dfrac{1}{n}\|\sum_{k=1}^{n}x_{k}\|^{2}-(1-\ve)
\left(\sum_{k=1}^{n}\|x_{k}\|^{2}\right).
\end{equation*}
\end{example}

\subsection{Критерий замкнутости суммы подпространств в терминах их
ортогональных дополнений}

Сумма
$H_{1}+\ldots+H_{n}=H$ тогда и только тогда, когда
$\Delta^{\bot}+\wt{H}=X$. Поскольку для подпространств
$M_{1},M_{2}$
$M_{1}+M_{2}$ замкнуто тогда и только тогда, когда
$M_{1}^{\bot}+M_{2}^{\bot}$ замкнуто, то последнее условие равносильно следующему:
$\Delta\cap\wt{H}^{\bot}=0$ и
$\Delta+\wt{H}^{\bot}$ замкнуто. В силу утверждения
\ref{ST:pair3} последнее равносильно
существованию
$\ve_{1}>0$, такого, что $\|(I-P_{\Delta})x\|^{2}\gs\ve_{1}\|x\|^{2}$ для произвольного
$x\in\wt{H}^{\bot}$.
Возьмём
$x=(x_{1},\ldots,x_{n}),\,x_{i}\in H_{i}^{\bot},\,1\ls i\ls n$.
Тогда неравенство $\|(I-P_{\Delta})x\|^{2}\gs\ve_{1}\|x\|^{2}$ можно записать в виде
\begin{equation*}
\sum_{k=1}^{n}\|x_{k}\|^{2}-\dfrac{1}{n}\|\sum_{k=1}^{n}x_{k}\|^{2}\gs\ve_{1}\sum_{k=1}^{n}\|x_{k}\|^{2},
\end{equation*}
что равносильно
\begin{equation*}
\sum_{i<j}\|x_{i}-x_{j}\|^{2}\gs n\ve_{1}\sum_{k=1}^{n}\|x_{k}\|^{2}.
\end{equation*}
Таким образом, $\sum_{k=1}^{n}H_{k}=H$ тогда и только тогда, когда существует
$\ve>0$, такое, что для произвольных
$x_{k}\in H_{k}^{\bot},\,1\ls k\ls n$
\begin{equation}\label{E:orth1}
\sum_{i<j}\|x_{i}-x_{j}\|^{2}\gs\ve\sum_{k=1}^{n}\|x_{k}\|^{2}.
\end{equation}

Далее
$\Gamma$ обозначает неориентированный граф с множеством вершин
$V(\Gamma)=\{1,2,\ldots,n\}$. Обозначим через
$E(\Gamma)$ множество рёбер
$\Gamma$. Будем писать
$i\sim j$, если
$i$ соединено с
$j$.

\begin{statement}\label{ST:orth1}
Пусть $\Gamma$ "--- связный граф. Утверджения равносильны:
\begin{enumerate}
\item
$\sum_{k=1}^{n}H_{k}=H$,
\item
существует
$\ve>0$, такое, что для произвольных
$x_{k}\in H_{k}^{\bot},\,1\ls k\ls n$
\begin{equation*}
\sum_{\{i,j\}\in E(\Gamma)}\|x_{i}-x_{j}\|^{2}\gs\ve\sum_{k=1}^{n}\|x_{k}\|^{2}.
\end{equation*}
\end{enumerate}
\end{statement}
\begin{proof}
\textbf{$(2)\Rightarrow(1)$} очевидно.

\textbf{$(1)\Rightarrow (2)$}
Пусть
$i\neq j$. Поскольку
$\Gamma$ связен, то существует путь
$i=i(0)\sim i(1)\sim\ldots\sim i(m)=j$. Тогда
\begin{equation*}
\|x_{i}-x_{j}\|^{2}\ls(\sum_{k=0}^{m-1}\|x_{i(k+1)}-x_{i(k)}\|)^{2}\ls m\sum_{k=0}^{m-1}\|x_{i(k+1)}-x_{i(k)}\|^{2}.
\end{equation*}
Теперь из неравенства
(\ref{E:orth1}) получим требуемое утверждение.
\end{proof}

Пусть, например,
$E(\Gamma)=\{\{1,2\},\{2,3\},\ldots,\{n-1,n\}\}$, т.е.
$\Gamma$ "--- цепь.
Из утверждения \ref{ST:orth1} следует, что сумма
$H_{1}+\ldots+H_{n}=H$ тогда и только тогда, когда существует
$\ve>0$ (уменьшенное $\ve$ из утверждения \ref{ST:orth1}), такое, что
для произвольных
$x_{k}\in H_{k}^{\bot},\,\|x_{k}\|=1,\,1\ls k\ls n$, и для произвольных
$t_{k}\in\mathbb{C},\,1\ls k\ls n$, не равных одновременно 0, выполнено
\begin{equation*}
\sum_{k=1}^{n-1}\|t_{k}x_{k}-t_{k+1}x_{k+1}\|^{2}>\ve\sum_{k=1}^{n}|t_{k}|^{2}.
\end{equation*}
Это условие равносильно положительной определённости эрмитово матрицы
\begin{equation*}
A_{\mathbb{C}}=\begin{pmatrix}
1-\ve&-(x_{1},x_{2})&0&\ldots&0\\
-(x_{2},x_{1})&2-\ve&-(x_{2},x_{3})&\ddots&\vdots\\
0&-(x_{3},x_{2})&\ddots&\ddots&0\\
\vdots&\ddots&\ddots&2-\ve&-(x_{n-1},x_{n})\\
0&\ldots&0&-(x_{n},x_{n-1})&1-\ve
\end{pmatrix}.
\end{equation*}
Используя критерий Сильвестра, можно получить необходимые и достаточные условия для положительной определенности
$A_{\mathbb{C}}$.

\begin{example} Пусть
$n=3$. Сумма
$H_{1}+H_{2}+H_{3}=H$ тогда и только тогда, когда
\begin{equation*}
\sup\{|(x_{1},x_{2})|^{2}+|(x_{2},x_{3})|^{2},\,x_{k}\in H_{k}^{\bot},\,\|x_{k}\|=1,\,k=1,2,3\}<2.
\end{equation*}
\end{example}

Далее мы будем рассматривать графы
$\Gamma$ с положительными весами на рёбрах.
Это значит, что каждому ребру
$e=\{i,j\}\in E(\Gamma)$ сопоставлено число
$\gamma_{e}>0$, которое мы будем обозначать
$\gamma_{i,j}=\gamma_{j,i}$. Для вершины
$i$ определим
$\rho_{i}=\sum_{j\sim i}\gamma_{i,j}$.

\begin{statement}\label{ST:orth2}
Пусть $\Gamma$ "--- связный граф с положительными весами на рёбрах. Утверждения равносильны:
\begin{enumerate}
\item
$\sum_{k=1}^{n}H_{k}=H$,
\item
существует
$\ve>0$, такое, что для произвольных
$x_{k}\in H_{k}^{\bot},\,1\ls k\ls n$
\begin{equation}\label{E:orth2}
2\sum_{\{i,j\}\in E(\Gamma)}\gamma_{i,j}|(x_{i},x_{j})|\ls\sum_{i=1}^{n}(\rho_{i}-\ve)\|x_{i}\|^{2}.
\end{equation}
\end{enumerate}
\end{statement}
\begin{proof}
\textbf{$(2)\Rightarrow (1)$}
Для произвольных
$x_{k}\in H_{k}^{\bot},\,1\ls k\ls n$ имеем:
\begin{align*}
&\sum_{\{i,j\}\in E(\Gamma)}\gamma_{i,j}\|x_{i}-x_{j}\|^{2}=
\sum_{i=1}^{n}\rho_{i}\|x_{i}\|^{2}-2\sum_{\{i,j\}\in E(\Gamma)}\gamma_{i,j}\tmd{Re}(x_{i},x_{j})\gs\\
&\gs\sum_{i=1}^{n}\rho_{i}\|x_{i}\|^{2}-2\sum_{\{i,j\}\in E(\Gamma)}\gamma_{i,j}|(x_{i},x_{j})|\gs
\ve\sum_{i=1}^{n}\|x_{i}\|^{2}.
\end{align*}
Из утверждения
\ref{ST:orth1} следует, что
$\sum_{k=1}^{n}H_{k}=H$.

\textbf{$(1)\Rightarrow (2)$}
Докажем требуемое утверждение индукцией по
$|E(\Gamma)|$. Наименьшее возможное значение
$|E(\Gamma)|$ равно
$n-1$, и достигается тогда и только тогда, когда
$\Gamma$ "--- дерево. Из утверждения
\ref{ST:orth1} следует, что существует
$\ve>0$, такое, что для произвольных
$y_{k}\in H_{k}^{\bot}$ выполнено:
\begin{equation*}
\sum_{\{i,j\}\in E(\Gamma)}\gamma_{i,j}\|y_{i}-y_{j}\|^{2}\gs\ve\sum_{i=1}^{n}\|y_{i}\|^{2},
\end{equation*}
т.е.
\begin{equation}\label{E:orth3}
2\sum_{\{i,j\}\in E(\Gamma)}\gamma_{i,j}\tmd{Re}(y_{i},y_{j})\ls\sum_{i=1}^{n}(\rho_{i}-\ve)\|y_{i}\|^{2}.
\end{equation}
Зафиксируем произвольные
$x_{k}\in H_{k}^{\bot},\,1\ls k\ls n$. В неравенство
(\ref{E:orth3}) подставим
$y_{k}=e^{i\varphi_{k}}x_{k}$, где
$\varphi_{k}\in\mathbb{R}$. Поскольку
$\Gamma$ "--- дерево, то
$\varphi_{k}$ можно выбрать так, что
$(y_{k},y_{l})=|(x_{k},x_{l})|$ если
$k\sim l$. Тогда из неравенства
(\ref{E:orth3}) следует требуемое.

Выполним индукционный переход. Пусть
$|E(\Gamma)|\gs n$. Тогда в
$\Gamma$ есть цикл
$i_{1},i_{2},\ldots,i_{m}$,
$i_{k}\sim i_{k+1},\,1\ls k\ls m$
(тут $i_{m+1}=i_{1}$). Для
$k=1,2,\ldots,m$ обозначим
$\Gamma_{k}$ граф, полученый из
$\Gamma$ удалением ребра
$\{i_{k},i_{k+1}\}$. Ясно, что
$\Gamma_{k}$ связен. Веса
$\gamma_{i,j}^{(k)}$ на рёбрах
$\Gamma_{k}$ определим следующим образом:
если ребро
$\{i,j\}$ является ребром цикла
$\{i_{p},i_{p+1}\}$, то
$\gamma_{i,j}^{(k)}=\gamma_{i,j}$, иначе
$\gamma_{i,j}^{(k)}=\frac{m-1}{m}\gamma_{i,j}$.
Из предположения индукции следует, что существует
$\ve_{k}>0$ такое, что для
$\Gamma_{k}$ выполнено неравенство
(\ref{E:orth2}). Уменьшив
$\ve_{k}$, можно считать, что
$\ve_{1}=\ldots=\ve_{m}=\ve$. Для произвольных
$x_{i}\in H_{i}^{\bot}$ имеем:
\begin{equation*}
2\sum_{\{i,j\}\in E(\Gamma_{k})}\gamma_{i,j}^{(k)}|(x_{i},x_{j})|\ls\sum_{i=1}^{n}(\rho_{i}^{(k)}-\ve)\|x_{i}\|^{2}.
\end{equation*}
Прибавив эти неравенства для
$k=1,2,\ldots,m$, получим неравенство
\begin{equation*}
2(m-1)\sum_{\{i,j\}\in E(\Gamma)}\gamma_{i,j}|(x_{i},x_{j})|\ls\sum_{i=1}^{n}((m-1)\rho_{i}-m\ve)\|x_{i}\|^{2}.
\end{equation*}
Разделив его на
$m-1$, получим требуемое утверждение.
\end{proof}

\section{Линейно независимые системы подпространств}\label{S:linind}

\begin{definition}
Подпространства
$X_{1},\ldots,X_{n}$ банахова пространства
$X$ линейно независимы, если из
\begin{equation*}
\sum_{j=1}^{n}x_{j}=0,~x_{j}\in X_{j},~1\ls j\ls n
\end{equation*}
следует
$x_{1}=\ldots=x_{n}=0$.
\end{definition}

В этом разделе мы будем изучать $n$-ки подпространств с замкнутой суммой
с дополнительным условием линейной независимости $n$-ки. Свойство линейной
независимости <<хорошо сочетается>> со свойством замкнутости суммы. Как мы
увидим в этом разделе, некоторые свойства $n$-ок, неверные только при
условии замкнутости суммы, становятся верными при дополнительном условии
линейной независимости.

Отметим, что при $n\gs 3$ задача об описание неприводимых
$n$-ок ортопроекторов с точностью до унитарной эквивалентности
чрезвычайно сложна (см. например
\cite{OstSam}). Поэтому
<<хорошего>> представления
(типа представления П. Халмоша для пары подпространств) $n$-ки подпространств при $n\gs 3$ нет.
Однако некоторые представления для линейно независимых $n$-ок подпространств с суммой $H$
существуют
(см., например, \cite{Sunder}).
В указанной работе используется, но не доказано, что если
$H_{1},\ldots,H_{n}$ "--- линейно независимые подпространства $H$
и $H_{1}+\ldots+H_{n}=H$, то для всех
$1\ls k\ls n$ сумма $H_{1}+\ldots+H_{k}$ замкнута. Это просто восполнить
с помощью следствия \ref{C:sumlinind}).

Приведём несколько примеров условий, при выполнении которых подпространства
$H_{1},\ldots,H_{n}$ линейно независимы, а их сумма замкнута.

\begin{example}
Пусть
$H_{1},\ldots,H_{n}$ "--- ненулевые подпространства
гильбертова пространства
$H$,
$P_{1},\ldots,P_{n}$ "--- соответствующие ортопроекторы,
$\tau_{1},\ldots,\tau_{n}$ "--- положительные числа.
Предположим, для некоторого
$\gamma>0$ выполнено
$\tau_{1}P_{1}+\ldots+\tau_{n}P_{n}\ls\gamma I$.
Покажем, что если
\begin{equation}\label{E:inequalitygamma}
\gamma<\dfrac{\sum_{j=1}^{n}\tau_{j}}{n-1},
\end{equation}
то
$H_{1},\ldots,H_{n}$ линейно независимы, а их сумма замкнута.

Поскольку
$H_{j}\neq 0,\,1\ls j\ls n$, то
$\tau_{j}\ls\gamma,\,1\ls j\ls n$.
Определим подпространства
$M_{k}=\ol{H_{1}+\ldots+H_{k}},\,1\ls k\ls n$
Проводя рассуждения,
аналогичные доказательству леммы 2 в ~\cite{SumProject.}, получим:
для всех
$1\ls k\ls n$ выполнено:
$\tau_{1}P_{1}+\ldots+\tau_{k}P_{k}\gs (\tau_{1}+\ldots+\tau_{k}-(k-1)\gamma)P_{M_{k}}.$
Подставив в это неравенство
$k=n$ и используя неравенство (\ref{E:inequalitygamma}), получим, что
$H_{1}+\ldots+H_{n}$ "--- подпространство. Покажем, что
$H_{1},\ldots,H_{n}$ линейно независимы.
Предположим, что существует
$x\in H_{n}\cap(H_{1}+\ldots+H_{n-1}),\,x\neq 0$.
Тогда из неравенств
\begin{equation*}
\gamma||x||^{2}\gs((\tau_{1}P_{1}+\ldots+\tau_{n-1}P_{n-1})x,x)+\tau_{n}(P_{n}x,x)\gs
(\sum_{k=1}^{n-1}\tau_{k}-(n-2)\gamma)||x||^{2}+\tau_{n}||x||^{2}
\end{equation*}
получаем:
$\gamma\gs\dfrac{\sum_{j=1}^{n}\tau_{j}}{n-1}$, противоречие.
Из соображений симметрии имеем:
для всех
$1\ls i\ls n$ $H_{i}\cap(\sum_{j\neq i}H_{j})=0$, что
и означает линейную независимость $H_{1},\ldots,H_{n}$.

Оценка (\ref{E:inequalitygamma}) для $\gamma$, при выполнении которой
$H_{1},\ldots,H_{n}$ линейно независимы и их сумма замкнута, вообще говоря, неулучшаема.
В работе \cite{SumProject.} показано, что в гильбертовом пространстве
$H=\mathbb{C}^{n-1}$
существуют одномерные подпространства
$H_{1},\ldots,H_{n}$, для которых
$P_{1}+\ldots+P_{n}=\frac{n}{n-1}I$.
Ясно, что
$H_{1},\ldots,H_{n}$ линейно зависимы.
\end{example}

\begin{example}\cite{Gua}
Пусть $X,Y,Z$ "--- банаховы пространства, $T:X\rightarrow Y, S:Y\rightarrow Z$
"--- линейные непрерывные операторы. Предположим, что $ST:X\rightarrow Z$ "--- изоморфизм.
Тогда $\tmd{Im}(T)$ "--- подпространство; подпространства $\tmd{Im}(T)$ и $\ker(S)$
линейно независимы и их сумма равна $Y$.
\end{example}

\subsection{Критерий замкнутости суммы линейно независимых подпространств. Примеры
его использования.}

\begin{theorem}\label{T:sumlinind}
Пусть $X_{1},\ldots,X_{n}$ "--- подпространства
банахова пространства $X$. Справедливы утверждения:
\begin{enumerate}
\item если для некоторого
$\ve>0$ для всех
$x_{j}\in X_{j},~1\ls j\ls n$ выполнено:
\begin{equation*}
\|x_{1}+\ldots+x_{n}\|\gs \ve(\|x_{1}\|+\ldots+\|x_{n-1}\|)
\end{equation*}
то
$X_{1},\ldots,X_{n}$ линейно независимы и их сумма "---
подпространство.
\item если
$X_{1},\ldots,X_{n}$ "--- линейно независимые подпространства,
$X_{1}+\ldots+X_{n}$ "--- подпространство, то
существует
$\ve>0$, такое, что для всех
$x_{j}\in X_{j},~1\ls j\ls n$ выполнено:
\begin{equation*}
\|x_{1}+\ldots+x_{n}\|\gs \ve(\|x_{1}\|+\ldots+\|x_{n}\|).
\end{equation*}
\end{enumerate}
\end{theorem}
\begin{proof}
Докажем
$(1)$.
Ясно, что $X_{1},\ldots,X_{n}$
линейно независимы.
Покажем, что $X_{1}+\ldots+X_{n}$ "--- подпространство. Пусть
$x_{k,1}+\ldots+x_{k,n}\to z$,
де $x_{k,i}\in X_{i},\,k\gs 1,\,1\ls i\ls n$.
Поскольку
\begin{equation*}
\|(x_{k,1}+\ldots+x_{k,n})-(x_{l,1}+\ldots+x_{l,n})\|\gs \ve\sum_{i=1}^{n-1}\|x_{k,i}-x_{l,i}\|,
\end{equation*}
то
при $1\ls i\ls n-1$ последовательность $\{x_{k,i},\,k\gs 1\}$  фундаментальна, а потому
$x_{k,i}\to x_{i}\in X_{i}$.
Поэтому последовательность $x_{k,n}\to x_{n}\in X_{n}$.
Тогда
\begin{equation*}
z=x_{1}+\ldots+x_{n},
\end{equation*}
откуда следует замкнутость $X_{1}+\ldots+X_{n}$.

Докажем $(2)$.
На пространстве
$X_{1}\oplus\ldots\oplus X_{n}$
определим норму
$\|(x_{1},\ldots,x_{n})\|=\sum_{i=1}^{n}\|x_{i}\|$, относительно
которой пространство банахово.
Рассмотрим оператор
$A:X_{1}\oplus\ldots\oplus X_{n}\rightarrow X_{1}+\ldots+X_{n}$
определённый равенством
$A(x_{1},\ldots,x_{n})=x_{1}+\ldots+x_{n}$.

Очевидно, $A$ является линейным непрерывным оператором между банаховыми пространствами, причём $A$ "--- биекция.
По теореме Банаха $A$ обратим, откуда непосредственно следует нужное неравенство.
\end{proof}

\begin{remark}
Вместо $\sum_{i=1}^{n}\|x_{i}\|$ иногда удобнее рассматривать эквивалентную величину
\begin{equation*}
(\sum_{i=1}^{n}\|x_{i}\|^{p})^{\frac{1}{p}}, 1\leq p<\infty.
\end{equation*}
Тогда условие замкнутости $X_{1}+\ldots+X_{n}$ и линейной независимости
$X_{1},\ldots,X_{n}$ имеет вид:
\begin{equation*}
\|\sum_{i=1}^{n}x_{i}\|^{p}\gs \ve\left(\sum_{i=1}^{n}\|x_{i}\|^{p}\right)
\end{equation*}
для некоторого $\ve>0$ и всех $x_{1}\in X_{1},\ldots,x_{n}\in X_{n}$.
\end{remark}

\begin{conclusion}\label{C:sumlinind}
Пусть $X_{1},\ldots,X_{n}$ "--- линейно независимые подпространства банахова пространства $X$ и
сумма $X_{1}+\ldots+X_{n}$ "--- подпространство. Тогда для произвольного набора индексов $i(1),\ldots,i(k)$ сумма
$X_{i(1)}+\ldots+X_{i(k)}$ "--- подпространство.
\end{conclusion}

\begin{conclusion}\label{C:sumsublinind}
Пусть $X_{1},\ldots,X_{n}$ "--- линейно независимые подпространства
банахова пространства
$X$ и сумма
$X_{1}+\ldots+X_{n}$ "--- подпространство. Пусть для каждого
$1\ls k \ls n$ подпространство $Y_{k}\subset X_{k}$.
Тогда
$Y_{1}+\ldots+Y_{n}$ "--- подпространство.
\end{conclusion}

Следствие ~\ref{C:sumlinind} показывает, что при условии линейной независимости
из замкнутости суммы всех подпространств следует замкнутость суммы любого поднабора.
Если не накладывать условия линейной независимости, то это утверждение неверно.
Более того, справедливо следующее утверждение.

\begin{statement}
Обозначим
$\mathbb{N}_{n}=\{1,2,\ldots,n\}$. Пусть множество
$\{I\subset\mathbb{N}_{n},|I|\gs 2\}$
разбито на две части
$I_{c}$ и $I_{nc}$. Тогда существуют гильбертово
пространство
$H$ и $n$-ка подпространств
$H_{1},\ldots,H_{n}$ в нем такие, что
если
$I\in I_{c}$ то $\sum_{j\in I}H_{j}$ замкнута, а
если
$I\in I_{nc}$ то $\sum_{j\in I}H_{j}$ не замкнута.
\end{statement}
\begin{proof}
\textbf{1.}
Сначала построим пример гильбертова пространства
$H$ и его подпространств
$H_{1},\ldots,H_{n}$, таких, что для произвольного
$|I|\ls n-1$ сумма $\sum_{j\in I}H_{j}$ замкнута, а
$H_{1}+\ldots+H_{n}$ не замкнута. Для
$k\gs 1$ в пространстве
$\mathbb{C}^{n}$ выберем ортонормированный базис
$e_{1},\ldots,e_{n}$, и определим набор одномерных подпространств
\begin{equation*}
H_{1,k}=\langle e_{1}\rangle,\ldots,H_{n-1,k}=\langle e_{n-1}\rangle,
H_{n,k}=\langle e_{1}+\ldots+e_{n-1}+\dfrac{1}{k}e_{n}\rangle.
\end{equation*}
Легко видеть, что
гильбертово пространство
$H=\bi_{k=1}^{\infty}\mathbb{C}^{n}$
и набор подпространств $H_{j}=\bi_{k=1}^{\infty}H_{j,k},~1\ls j\ls n$
обладают нужными свойствами.

\textbf{2.}
Докажем утверждение индукцией по $n$.
При
$n=2$ утверждение очевидно.
Выполним индукционный переход. Предположим, что
$\{1,2,\ldots,n\}\in I_{c}$. Для
$n$ гильбертовых пространств $L_{1},\ldots,L_{n}$
определим
$H=L_{1}\oplus\ldots\oplus L_{n}$, определим
\begin{equation*}
H_{j}=R_{j,1}\oplus\ldots\oplus R_{j,j-1}\oplus L_{j}\oplus R_{j,j+1}\oplus\ldots\oplus R_{j,n}
\end{equation*}
для некоторых подпространств
$R_{j,i}\subset L_{i}$. Очевидно,
$H_{1}+\ldots+H_{n}=H$. Осталось выбрать подпространства
$\{R_{j,i}\}$, чтобы выполнялись все нужные условия.
Фиксируем произвольное
$1\ls m\ls n$. На подпространства
$R_{j,m},~j\neq m$ наложим такие условия:
если множество
$I\subset\mathbb{N}_{n}\setminus\{m\}$ есть элементом
$I_{c}$, то $\sum_{j\in I}R_{j,m}$ замкнута, если
$I$ является элементом $I_{nc}$, то сумма
$\sum_{j\in I}R_{j,m}$ не замкнута. Из индукционного предположения
следует, что существует гильбертово пространство
$L_{m}$ и набор подпространств $R_{j,m},~j\neq m$
с нужными свойствами. Очевидно, построенная $n$-ка
$H_{1},\ldots,H_{n}$ гильбертова  пространства $H$
обладает нужными свойствами.

Пусть теперь
$\{1,2,\ldots,n\}\in I_{nc}$. Возьмём гильбертово пространство
$K$ и набор подпространств
$K_{1},\ldots,K_{n}$ в нем такие, что для всех
множеств $I\subset\mathbb{N}_{n},~|I|\ls n-1$ сумма $\sum_{j\in I}K_{j}$
замкнута, а сумма
$K_{1}+\ldots+K_{n}$ не замкнута. Для гильбертовых пространств
$L_{1},\ldots,L_{n}$ определим
$H=L_{1}\oplus\ldots\oplus L_{n}\oplus K$ и подпространства
\begin{equation*}
H_{j}=R_{j,1}\oplus\ldots\oplus R_{j,j-1}\oplus L_{j}\oplus R_{j,j+1}\oplus\ldots\oplus R_{j,n}\oplus K_{j}
\end{equation*}
для некоторых подпространств
$R_{i,j}\subset L_{i}$. Очевидно, сумма
$H_{1}+\ldots+H_{n}$ не замкнута.
Для фиксированного
$1\ls m\ls n$ набор подпространств
$R_{j,m},~j\neq m$ пространства $L_{m}$
(и само $L_{m}$) выбирается аналогично предыдущему случаю.
\end{proof}

Аналогом предыдущего утверждения для линейно независимых подпространств, с
учётом следствия \ref{C:sumlinind}, есть следующее утверждение.

\begin{statement}
Пусть
$n\gs 2$. Пусть множество
$\{I\subset\mathbb{N}_{n},|I|\gs 2\}$
разбито на две части
$I_{c}$ и $I_{nc}$, причём выполнено следующее условие: если
$A,B\subset \mathbb{N}_{n},\,|A|\gs 2,\,|B|\gs 2,\,A\subset B$ и
$B\in I_{c}$, то $A\in I_{c}$. Тогда существует гильбертово
пространство
$H$, и линейно независимые подпространства
$H_{1},\ldots,H_{n}$ в нём, такие, что
если
$I\in I_{c}$ то $\sum_{j\in I}H_{j}$ замкнута, а
если
$I\in I_{nc}$ то $\sum_{j\in I}H_{j}$ не замкнута.
\end{statement}
\begin{proof}
Положим
$N=2^{n}-1-n$. Занумеруем элементы множества
$\{I\subset \mathbb{N}_{n},\,|I|\gs 2\}$ числами от 1 до
$N$, т.е. $\{I\subset \mathbb{N}_{n},\,|I|\gs 2\}=\{I_{1},\ldots,I_{N}\}$.
Гильбертово пространство
$H$ и линейно независимые подпространства $H_{1},\ldots,H_{n}$ в нём
будем искать в следующем виде:
\begin{equation*}
H_{i}=R_{i,1}\oplus\ldots\oplus R_{i,N},\,1\ls i\ls n,\,H=R_{1}\oplus\ldots\oplus R_{N},
\end{equation*}
где для каждого
$1\ls j\ls N$
$R_{i,j},\,1\ls i\ls n$ есть подпространства гильбертова пространства $R_{j}$.
Покажем, как их выбрать, чтобы выполнялись требуемые условия.
Фиксируем $1\ls j\ls N$ и пусть
$I_{j}=\{i(1),\ldots,i(s)\}$.
Предположим, что $I_{j}\in I_{nc}$. При доказательстве предыдущего утверждения мы
показали, что существуют гильбертово пространство
$K$ и линейно независимые подпространства $K_{1},\ldots,K_{s}$ в нем такие, что
$\sum_{l\in J}K_{l}$ замкнута для всех
$J\subset\{1,2,\ldots,s\},\,|J|<s$, а сумма
$K_{1}+\ldots+K_{s}$ не замкнута. Положим
$R_{j}=K$,
$R_{i(1),j}=K_{1},\ldots,R_{i(s),j}=K_{s}$, остальные
$R_{i,j}=0$.
Предположим, что $I_{j}\in I_{c}$. Тогда положим
$R_{j}=\mathbb{C}^{1},\,R_{1,j}=\ldots=R_{n,j}=0$.
Легко видеть, что построенные таким образом подпространства
$H_{1},\ldots,H_{n}$ гильбертова пространства $H$ удовлетворяют всем нужным условиям.
\end{proof}

Приведём несколько примеров использования теоремы \ref{T:sumlinind}.

\begin{example}
Пусть
$X,Y$ "--- банаховы пространства,
$a_{1},\ldots,a_{n}:X\to Y$ таковы, что оператор
$a=\sum_{k=1}^{n}a_{k}$
обратим.
Тогда подпространства
$X_{i}=\bigcap_{j\neq i}\ker(a_{j}),~ 1\ls i\ls n$
линейно независимы и их сумма замкнута.
\begin{proof}
Пусть $x_{i}\in X_{i},~1\ls i\ls n$.
Рассмотрим произвольное
$1\ls k\ls n$. Тогда
\begin{equation*}
\|a_{k}\|\|\sum_{i=1}^{n}x_{i}\|\gs\|a_{k}x_{k}\|=\|ax_{k}\|\gs\|a^{-1}\|^{-1}\|x_{k}\|.
\end{equation*}
Прибавив полученые неравенства для $k=1,2,\ldots,n$, из теоремы
\ref{T:sumlinind} получим нужное утверждение.
\end{proof}

В частности, если
$\sum_{k=1}^{n}H_{k}=H$, то
$\sum_{k=1}^{n}P_{k}$ обратим, а поэтому подпространства
$M_{k}=\bigcap_{i\neq k}H_{i}^{\bot},\,1\ls k\ls n$ линейно независимы и их сумма замкнута.
\end{example}

\begin{example}
Покажем связь между операторами с конечным спектром и системами
линейно независимых подпространств с суммой $H$. Пусть
$A:H\to H$ "--- линейный непрерывный оператор с конечным спектром
$\sigma(A)=\{\lambda_{1},\ldots,\lambda_{n}\}$.
Пусть
$\Gamma_{k},\, 1\ls k\ls n$ "--- окружность достаточно малого радиуса с центром в
$\lambda_{k}$. Напомним, что проектором Рисса, соответствующим изолированной точке спектра
$\lambda_{k}$, называется оператор
\begin{equation*}
R_{k}=\dfrac{1}{2 \pi i}\int_{\Gamma_{k}}R(z,A)\,dz=\dfrac{1}{2 \pi i}\int_{\Gamma_{k}}(zI-A)^{-1}\,dz.
\end{equation*}
Хорошо известно (см. например \cite{Sadovnychy}), что
$R_{k}\neq 0$ при всех $k$,
$R_{k}^{2}=R_{k}$, а также
$R_{i}R_{j}=0$ при
$i\neq j$, сумма
$R_{1}+\ldots+R_{n}=I$.

Определим
$H_{k}=\tmd{Im}(R_{k}),\,1\ls k\ls n$.
Очевидно,
$H_{1},\ldots,H_{n}$ "--- ненулевые линейно независимые подпространства,
сумма которых равна
$H$. Справедливо и обратное утверждение.

\begin{statement}
Пусть
$\{\lambda_{1},\ldots,\lambda_{n}\}$ "--- множество $n$
комплексных чисел,
$H_{1},\ldots,H_{n}$ "--- ненулевые линейно независимые подпространства $H$,
сумма которых равна
$H$. Тогда найдётся линейный непрерывный оператор
$A:H\to H$ со спектром $\sigma(A)=\{\lambda_{1},\ldots,\lambda_{n}\}$, для которого
$H_{k}$ есть подпространством Рисса, соответствующим $\lambda_{k}$ для всех $1\ls k\ls n$.
\end{statement}

\begin{proof}
Каждый
$x\in H$ однозначно представляется в виде
$x=\sum_{k=1}^{n}x_{k}$, где
$x_{k}\in H_{k},\,1\ls k\ls n$.
Определим
$Q_{k}x=x_{k}$, тогда из теоремы
\ref{T:sumlinind} следует, что
$Q_{k}$ ограничен.

Определим линейный непрерывный оператор
$A:H\to H$ равенством
$Ax=\lambda_{1}Q_{1}+\ldots+\lambda_{n}Q_{n}$.
Иначе говоря, для $x=x_{1}+\ldots+x_{n},\,x_{i}\in H_{i},\,1\ls i\ls n$,
$Ax=\lambda_{1}x_{1}+\ldots+\lambda_{n}x_{n}$.
Очевидно,
$\sigma(A)=\{\lambda_{1},\ldots,\lambda_{n}\}$.
Для точки $\lambda_{k},1\ls k\ls n$ соответствующий
проектор Рисса равен
\begin{equation*}
R_{k}=\dfrac{1}{2\pi i}\int_{\Gamma_{k}}(zI-A)^{-1}\,dz=
\dfrac{1}{2\pi i}\int_{\Gamma_{k}}\sum_{j=1}^{n}\dfrac{1}{z-\lambda_{j}}Q_{j}\,dz=Q_{k},
\end{equation*}
а поэтому соответствующее подпространство Рисса равно $H_{k}$.
\end{proof}
\end{example}

Покажем несколько применений теоремы \ref{T:sumlinind}
для гильбертова пространства $H$.

\begin{example}
Пусть
$H_{1},\ldots,H_{n}$ "--- подпространства
гильбертова пространства
$H$. Сформулируем условие их линейной независимости и замкнутости их
суммы. Это равносильно существованию
$\ve>0$, такого, что для произвольных действительных чисел
$t_{1},\ldots,t_{n}\in\mathbb{R}$, не равных одновременно 0, и для
произвольных
$x_{k}\in H_{k},\,\|x_{k}\|=1,\,1\ls k\ls n$ справедливо неравенство
$\|t_{1}x_{1}+...+t_{n}x_{n}\|^{2}> \ve(t_{1}^{2}+...+t_{n}^{2})$.
Поставленое условие равносильно положительной определённости действительной
симметричной матрицы
\begin{equation*}
A_{\mathbb{R}}=\begin{pmatrix}
1-\ve&\tmd{Re}(x_{1},x_{2})&\ldots&\tmd{Re}(x_{1},x_{n})\\
\tmd{Re}(x_{2},x_{1})&1-\ve&\ddots&\vdots\\
\vdots&\ddots&\ddots&\tmd{Re}(x_{n-1},x_{n})\\
\tmd{Re}(x_{n},x_{1})&\ldots&\tmd{Re}(x_{n},x_{n-1})&1-\ve
\end{pmatrix}.
\end{equation*}
Используя критерий Сильвестра, можно получить необходимые и достаточные условия для
положительной определённости $A_{\mathbb{R}}$.
В частности, при
$n=2$ имеем условие
$|\tmd{Re}(x_{1},x_{2})|<1-\ve$, которое
равносильно условию:
$|(x_{1},x_{2})|<1-\ve$
для всех
$x_{j}\in H_{j},\,\|x_{j}\|=1,\,j=1,2$.

При
$n=3$ имеем условия:
\begin{enumerate}
\item
$|\tmd{Re}(x_{1},x_{2})|<1-\ve$
\item
$(\tmd{Re}(x_{1},x_{2}))^{2}+(\tmd{Re}(x_{2},x_{3}))^{2}+(\tmd{Re}(x_{3},x_{1}))^{2}<
(1-\ve)^{2}+\dfrac{2}{1-\ve}\tmd{Re}(x_{1},x_{2})\tmd{Re}(x_{2},x_{3})\tmd{Re}(x_{3},x_{1})$.
\end{enumerate}

Аналогично можно сформулировать условие линейной независимости
и замкнутости суммы подпространств
$H_{1},\ldots,H_{n}$, считая числа
$t_{1},\ldots,t_{n}$ комплексными. Тогда приходим к условию
положительной определённости эрмитовой матрицы
\begin{equation*}
A_{\mathbb{C}}=\begin{pmatrix}
1-\ve&(x_{1},x_{2})&\ldots&(x_{1},x_{n})\\
(x_{2},x_{1})&1-\ve&\ddots&\vdots\\
\vdots&\ddots&\ddots&(x_{n-1},x_{n})\\
(x_{n},x_{1})&\ldots&(x_{n},x_{n-1})&1-\ve
\end{pmatrix}.
\end{equation*}
Используя критерий Сильвестра, можно
получить необходимые и достаточные условия для положительной определённости $A_{\mathbb{C}}$.
В частности, при
$n=3$ имеем условия:
\begin{enumerate}
\item
$|(x_{1},x_{2})|<1-\ve$
\item
$|(x_{1},x_{2})|^{2}+|(x_{2},x_{3})|^{2}+|(x_{3},x_{1})|^{2}<
(1-\ve)^{2}+\dfrac{2}{1-\ve}\tmd{Re}((x_{1},x_{2})(x_{2},x_{3})(x_{3},x_{1}))$.
\end{enumerate}
\end{example}

\begin{example}
Пусть $H$ "--- гильбертово пространство, $n\geq 2$.
Пусть для каждого
$k=1,\ldots,n$ $\{e_{k,s},\,s\in \mathbb{Z}\}$ "--- ортонормированная система в
$H$. Пусть
$H_{k}$ "--- подпространство, порожденное системой
$\{e_{k,s},\,s\in \mathbb{Z}\}$.

Используя предыдущий пример, найдём достаточные условия, при которых
подпространства $H_{1},\ldots,H_{n}$ линейно независимы, а их сумма
$H_{1}+\ldots+H_{n}$ замкнута.

Для каждого целого
$p$ и индексов $i\neq j,\,1\ls i,j\ls n$
определим
\begin{equation*}
\alpha_{i,j,p}=\sup_{l-k=p}|(e_{i,k},e_{j,l})|.
\end{equation*}
Очевидно, что
$\alpha_{i,j,p}=\alpha_{j,i,-p}$. Для каждой пары
$i\neq j,~1\ls i,j\ls n$ определим
$\beta_{i,j}=\sum_{p\in \mathbb{Z}}\alpha_{i,j,p}.$
Предположим, что $\beta_{i,j}<\infty$ для всех $i\neq j$.

Пусть для
$k=1,\ldots,n$ имеем
$x_{k}=\sum_{s\in \mathbb{Z}}x_{k,s}e_{k,s} \in H_{k}$, причём
$\|x_{k}\|=1$. Это означает, что $\sum_{s}|x_{k,s}|^{2}=1$.
Используя неравенство Коши-Буняковского, получим, что при $i\neq j$
\begin{equation*}
|(x_{i},x_{j})|\leqslant\sum_{v,w\in\mathbb{Z}}|x_{i,v}||x_{j,w}|\alpha_{i,j,w-v}\leqslant\beta_{i,j}.
\end{equation*}
Отсюда
$|\tmd{Re}(x_{i},x_{j})|\leqslant\beta_{i,j}$ при $i\neq j$.

Для $t_{1},\ldots,t_{n}\in \mathbb{R}$ рассмотрим квадратичную форму
\begin{equation*}
\sum_{k\neq l}t_{k}t_{l}\tmd{Re}(x_{k},x_{l})\gs
-\sum_{k\neq l}\beta_{k,l}|t_{k}||t_{l}|.
\end{equation*}

Отсюда следует, что если существует $\ve>0$, для которого
матрица
\begin{equation*}
B=\begin{pmatrix}
1-\ve&-\beta_{1,2}&\ldots&-\beta_{1,n}\\
-\beta_{2,1}&1-\ve&\ldots&\vdots\\
\vdots&\ddots&\ddots&-\beta_{n-1,n}\\
-\beta_{n,1}&\ldots&-\beta_{n,n-1}&1-\ve
\end{pmatrix}
\end{equation*}
положительно определена, то подпространства
$H_{1},\ldots,H_{n}$ линейно независимы и их сумма является
подпространством.
В частности, если
$\max_{k}(\sum_{l\neq k}\beta_{k,l})<1$,
то это условие выполнено (при
$\ve<1-\max_{k}(\sum_{l\neq k}\beta_{k,l}))$.
\end{example}

Следующий пример мотивирован результатами работ \cite{Grobler} и \cite{probability1}.
К сожалению, в работе
\cite{Grobler} есть ошибки.
На с.184 определение $\ol{\mathbb{P}}(A\cap B)=\mathbb{P}_{1}(A)\mathbb{P}_{2}(B)$ для
$A\in\mc{F}_{1},\,B\in\mc{F}_{2}$ вообще говоря, некорректно, поскольку множество
может допускать различные представления в виде $A\cap B,\,A\in\mc{F}_{1},\,B\in\mc{F}_{2}$.
Кроме того, наложив дополнительные условия на $\sigma$-алгебры
$\mc{F}_{1},\mc{F}_{2}$, чтобы определение
$\ol{\mathbb{P}}$ стало корректным (такое условие мы приводим в следующем примере),
для продолжения по Каратеодори
надо показать, что $\ol{\mathbb{P}}$ есть мерой на полуалгебре
$\Gamma=\{A\cap B,\,A\in\mc{F}_{1},\,B\in\mc{F}_{2}\}$.

\begin{example}(Замкнутость суммы маргинальных подпространств)
Пусть
$(\Omega,\mc{F},\textbf{P})$ "--- вероятностное пространство,
$\mc{F}_{1},\ldots,\mc{F}_{n}$ "---
$\sigma$-алгебры, причём $\mc{F}_{j}\subset\mc{F},\,1\ls j\ls n$.
Будем говорить, что $\mc{F}$-измеримые комплекснозначные случайные величины $\xi$ и $\eta$
эквивалентны, если
$\textbf{P}\{\xi\neq\eta\}=0$. Определим пространство
$H$ как множество классов эквивалентности всех $\mc{F}$-измеримых комплекснозначных
случайных величин $\xi$, для которых $\int_{\Omega}|\xi|^{2}\,d\tbf{P}<\infty$ и
$\int_{\Omega}\xi\,d\tbf{P}=0$. Тогда $H$ "--- гильбертово пространство
относительно скалярного произведения
$(\xi,\eta)=\int_{\Omega}\xi\ol{\eta}\,d\tbf{P}$. Для каждого
$1\ls i\ls n$ пусть
$H_{i}$ "--- множество классов эквивалентности из $H$, в которых есть хотя бы одна $\mc{F}_{i}$-измеримая
случайная величина. Легко видеть, что
$H_{i}$ "--- подпространство $H$.
$H_{i}$ называют маргинальным подпространством $H$. Наша цель: показать, что если существует
$\ve>0$, такое, что для произвольных $A_{1}\in\mc{F}_{1},\ldots,A_{n}\in\mc{F}_{n}$
\begin{equation}\label{E:equation1}
\tbf{P}(A_{1}\cap\ldots\cap A_{n})\gs\ve\tbf{P}(A_{1})\ldots\tbf{P}(A_{n}),
\end{equation}
то
$H_{1},\ldots,H_{n}$ линейно независимы и их сумма "--- подпространство.
Итак, далее считаем, что выполнено неравенство (\ref{E:equation1}).
Через $\tbf{P}_{n}$ обозначим продакт-меру
$\underbrace{\tbf{P}\times\ldots\times\tbf{P}}_{n}$
на измеримом пространстве
$\Omega^{n}=\underbrace{\Omega\times\ldots\times\Omega}_{n}$
с $\sigma$-алгеброй
$\underbrace{\mc{F}\otimes\ldots\otimes\mc{F}}_{n}$ (порожденной измеримыми брусами
$A_{1}\times\ldots\times A_{n}$, $A_{i}\in\mc{F},\,1\ls i\ls n$).

\begin{lemma}\label{L:lemma1}
Пусть
$X_{i},Y_{i}\in\mc{F}_{i},\,1\ls i\ls n$ и
$X_{1}\cap\ldots \cap X_{n}\subset Y_{1}\cap\ldots\cap Y_{n}$. Тогда
\begin{equation*}
\tbf{P}_{n}((X_{1}\times\ldots\times X_{n})\setminus (Y_{1}\times\ldots\times Y_{n}))=0.
\end{equation*}
\end{lemma}
\begin{proof}
Поскольку
$X_{1}\cap\ldots \cap X_{n}\subset Y_{1}\cap\ldots\cap Y_{n}$, то для каждого
$1\ls i\ls n$ $X_{1}\cap\ldots\cap X_{i-1}\cap(X_{i}\setminus Y_{i})\cap X_{i+1}\ldots\cap X_{n}=\varnothing$.
Из неравенства
\eqref{E:equation1} следует, что
$\tbf{P}_{n}(X_{1}\times\ldots\times X_{i-1}\times(X_{i}\setminus Y_{i})\times X_{i+1}\ldots\times X_{n})=0$.
Отсюда следует нужное.
\end{proof}

Определим класс множеств
$\mc{P}=\{X_{1}\cap\ldots\cap X_{n},\,X_{i}\in\mc{F}_{i},\,1\ls i\ls n\}$. Очевидно,
$\mc{P}$ "--- полуалгебра. Определим функцию множеств $\tbf{Q}$ на $\mc{P}$ равенством
\begin{equation*}
\tbf{Q}(X_{1}\cap\ldots\cap X_{n})=\tbf{P}(X_{1})\ldots\tbf{P}(X_{n})=
\tbf{P}_{n}(X_{1}\times\ldots\times X_{n}).
\end{equation*}
В силу леммы \ref{L:lemma1} $\tbf{Q}$ определена корректно.
Мы покажем, что
$\tbf{Q}$ "--- мера на $\mc{P}$. Для этого нам нужна следующая лемма.

\begin{lemma} Пусть множества
$X_{k,j}\in\mc{F}_{j},\,1\ls k\ls t,\,1\ls j\ls n$. Тогда
\begin{equation*}
\tbf{P}_{n}(\bigcup_{k=1}^{t}X_{k,1}\times\ldots\times X_{k,n})\ls\dfrac{1}{\ve}
\tbf{P}(\bigcup_{k=1}^{t}X_{k,1}\cap\ldots\cap X_{k,n}).
\end{equation*}
\end{lemma}
\begin{proof}
Представим множество
$\bigcup_{k=1}^{t}X_{k,1}\times\ldots\times X_{k,n}$ в виде
$\bigcup_{l=1}^{s}Y_{l,1}\times\ldots\times Y_{l,n}$, где
множества
$Y_{l,j}\in\mc{F}_{j},\,1\ls l\ls s,\,1\ls j\ls n$, и множества
$Y_{l,1}\times\ldots\times Y_{l,n},\,1\leqslant l\leqslant s$ попарно не пересекаются.
Имеем:
\begin{align*}
&\tbf{P}_{n}(\bigcup_{k=1}^{t}X_{k,1}\times\ldots\times X_{k,n})=
\tbf{P}_{n}(\bigcup_{l=1}^{s}Y_{l,1}\times\ldots\times Y_{l,n})=\\
&=\sum_{l=1}^{s}\tbf{P}_{n}(Y_{l,1}\times\ldots\times Y_{l,n})\ls
\dfrac{1}{\ve}\sum_{l=1}^{s}\tbf{P}(Y_{l,1}\cap\ldots\cap Y_{l,n})=\\
&=\dfrac{1}{\ve}\tbf{P}(\bigcup_{l=1}^{s}Y_{l,1}\cap\ldots\cap Y_{l,n})=
\dfrac{1}{\ve}\tbf{P}(\bigcup_{k=1}^{t}X_{k,1}\cap\ldots\cap X_{k,n}),
\end{align*}
что и требовалось доказать.
\end{proof}

Теперь покажем, что
$\tbf{Q}$ является мерой на $\mc{P}$. Предположим, что
\begin{equation*}
A_{1}\cap\ldots\cap A_{n}=\bigcup_{k=1}^{\infty}(A_{k,1}\cap\ldots\cap A_{k,n}),
\end{equation*}
где
$A_{j},A_{k,j}\in\mc{F}_{j},\,1\ls j\ls n,\,k\gs 1$ и множества
$A_{k,1}\cap\ldots\cap A_{k,n},\,k\gs 1$ попарно не пересекаются. Определим на $\Omega^{n}$ функции
\begin{equation*}
F_{1}(x_{1},\ldots,x_{n})=\mathbb{I}_{A_{1}}(x_{1})\ldots\mathbb{I}_{A_{n}}(x_{n}),\,
F_{2}(x_{1},\ldots,x_{n})=\sum_{k=1}^{\infty}\mathbb{I}_{A_{k,1}}(x_{1})\ldots\mathbb{I}_{A_{k,n}}(x_{n}).
\end{equation*}
Покажем, что
$F_{1}=F_{2}$ почти всюду относительно
$\tbf{P}_{n}$. Для этого предположим, что для элемента
$(x_{1},\ldots,x_{n})$ $F_{1}\neq F_{2}$. Возможны следующие варианты:
\begin{enumerate}
\item
$F_{1}=0,F_{2}\gs 1$,
\item
$F_{1}=1,F_{2}\gs 2$,
\item
$F_{1}=1,F_{2}=0$.
\end{enumerate}
Множество
$(x_{1},\ldots,x_{n})$, для которых выполнено $(1)$ имеет вид
\begin{equation*}
\bigcup_{k=1}^{\infty}\bigcup_{j=1}^{n}
(A_{k,1}\times\ldots\times A_{k,j-1}\times (A_{k,j}\setminus A_{j})\times A_{k,j+1}\times\ldots\times A_{k,n}),
\end{equation*}
которое, в силу леммы
\ref{L:lemma1} имеет $\tbf{P}_{n}$-меру 0.

Множество элементов $(x_{1},\ldots,x_{n})$, для которых имеет место второй случай,
вложено в множество
\begin{equation*}
\bigcup_{k>l}((A_{k,1}\cap A_{l,1})\times\ldots\times(A_{k,n}\cap A_{l,n})),
\end{equation*}
которое, как следует из
(\ref{E:equation1}), имеет $\tbf{P}_{n}$-меру 0.

Осталось показать, что множество элементов, для которых имеет место третий вариант, имеет
$\tbf{P}_{n}$-меру 0. Это множество равно
$A_{1}\times\ldots\times A_{n}\setminus\bigcup_{k=1}^{\infty}A_{k,1}\times\ldots\times A_{k,n}$.
Имеем (напомним, что $\mathbb{N}_{t}=\{1,2,\ldots,t\}$):
\begin{align*}
&\tbf{P}_{n}\left(A_{1}\times\ldots\times A_{n}\setminus\bigcup_{k=1}^{\infty}A_{k,1}\times\ldots\times A_{k,n}\right)=\\
&=\lim_{t\to\infty}
\tbf{P}_{n}\left(A_{1}\times\ldots\times A_{n}\setminus\bigcup_{k=1}^{t}A_{k,1}\times\ldots\times A_{k,n}\right)=\\
&=\lim_{t\to\infty}\tbf{P}_{n}\left(\bigcup_{I_{1}\bigcup\ldots\bigcup I_{n}=\mathbb{N}_{t}}
(A_{1}\setminus\bigcup_{k\in I_{1}}A_{k,1})\times\ldots\times(A_{n}\setminus\bigcup_{k\in I_{n}}A_{k,n})\right)\ls\\
&\ls\lim_{t\to\infty}\dfrac{1}{\ve}\tbf{P}\left(\bigcup_{I_{1}\bigcup\ldots\bigcup I_{n}=\mathbb{N}_{t}}
(A_{1}\setminus\bigcup_{k\in I_{1}}A_{k,1})\cap\ldots\cap(A_{n}\setminus\bigcup_{k\in I_{n}}A_{k,n})\right)=\\
&=\lim_{t\to\infty}\frac{1}{\varepsilon}\tbf{P}\left(A_{1}\cap\ldots\cap A_{n}\setminus
\bigcup_{k=1}^{t}A_{k,1}\cap\ldots\cap A_{k,n}\right)=\\
&=\lim_{t\to\infty}\dfrac{1}{\ve}\tbf{P}\left(\bigcup_{k\gs t+1}A_{k,1}\cap\ldots\cap A_{k,n}\right)=0.
\end{align*}
Интегрируя равенство
$F_{1}=F_{2}$, выполненое почти всюду относительно
$\tbf{P}_{n}$ по множеству $\Omega^{n}$ и мере $\tbf{P}_{n}$, получим
\begin{equation*}
\tbf{Q}(A_{1}\cap\ldots\cap A_{n})=\sum_{k=1}^{\infty}\tbf{Q}(A_{k,1}\cap\ldots\cap A_{k,n}).
\end{equation*}
Таким образом, мы показали, что
$\tbf{Q}$ "--- мера на $\mc{P}$.
$\tbf{Q}$ имеет единственное продолжение на
$\sigma$-алгебру $\mc{G}=\sigma a(\mc{F}_{1},\ldots,\mc{F}_{n})\subset\mc{F}$.
При этом
$\tbf{P}(A)\gs\ve\tbf{Q}(A),\,A\in\mc{G}$ (поскольку это неравенство выполнено для множеств
$A=A_{1}\cap\ldots\cap A_{n},\,A_{j}\in\mc{F}_{j},\,1\ls j\ls n$). Также отметим, что если
$\xi_{j}$ есть $\mc{F}_{j}$-измеримая случайная величина, интегрируемая по $\tbf{P}$ ($1\ls j\ls n$), то
величина
$\xi_{1}\ldots\xi_{n}$ $\mc{G}$-измерима и интегрируема по мере $\tbf{Q}$, причём
\begin{equation*}
\int_{\Omega}\xi_{1}\ldots\xi_{n}\,d\tbf{Q}=\int_{\Omega}\xi_{1}\,d\tbf{P}\ldots\int_{\Omega}\xi_{n}\,d\tbf{P}.
\end{equation*}

Теперь уже легко показать, что $H_{1},\ldots,H_{n}$
линейно независимы и их сумма замкнута. Пусть
$\xi_{j}\in H_{j}$ есть $\mc{F}_{j}$-измеримая случайная величина ($1\ls j\ls n$). Неравенство
\begin{equation*}
\|\xi_{1}+\ldots+\xi_{n}\|^{2}\gs\ve(\|\xi_{1}\|^{2}+\ldots+\|\xi_{n}\|^{2})
\end{equation*}
равносильно неравенству
\begin{equation*}
\int_{\Omega}|\xi_{1}+\ldots+\xi_{n}|^{2}\,d\tbf{P}\gs\ve\int_{\Omega}|\xi_{1}+\ldots+\xi_{n}|^{2}\,d\tbf{Q},
\end{equation*}
которое, очевидно, выполнено.

\begin{remark}
Полученый результат можно применить в следующем случае:
пусть
$(\Omega_{j},\mc{G}_{j})$ "--- измеримое пространство ($1\ls j\ls n$).
Пусть
$\Omega=\Omega_{1}\times\ldots\times\Omega_{n}$, и $\sigma$-алгебра
$\mc{F}=\mc{G}_{1}\otimes\ldots\otimes\mc{G}_{n}$.
Определим $\sigma$-алгебру
\begin{equation*}
\mc{F}_{j}=\{\Omega_{1}\times\ldots\times\Omega_{j-1}\times A_{j}\times\Omega_{j+1}\times\ldots\times\Omega_{n},\,
A_{j}\in\mc{G}_{j}\}
\end{equation*}
для всех
$1\ls j\ls n$. Определим маргинальную вероятность
$\tbf{P}_{j}$ на $(\Omega_{j},\mc{G}_{j})$ ($1\ls j\ls n$) равенством
\begin{equation*}
\tbf{P}_{j}(A_{j})=\tbf{P}(\Omega_{1}\times\ldots\times\Omega_{j-1}\times A_{j}\times\Omega_{j+1}\times\ldots\times\Omega_{n}),\,A_{j}\in\mc{G}_{j}.
\end{equation*}
Если существует $\ve>0$, такое, что для произвольных
$A_{1}\in\mc{G}_{1},\ldots,A_{n}\in\mc{G}_{n}$
\begin{equation*}
\tbf{P}(A_{1}\times\ldots\times A_{n})\gs\ve\tbf{P}_{1}(A_{1})\ldots\tbf{P}_{n}(A_{n}),
\end{equation*}
то маргинальные подпространства $H_{1},\ldots,H_{n}$ (в нашем случае $H_{j}$ есть множество (классов эквивалентности)
случайных величин $\xi(x_{1},\ldots,x_{n})=\eta(x_{j})$, где
$\eta$ $\mc{G}_{j}$-измерима) линейно независимы, а их сумма "--- подпространство.
\end{remark}
\end{example}

\subsection{Свойство обратного наилучшего приближения системы
подпространств гильбертова пространства}\label{SS:IBAP}

Пусть
$H$ "--- гильбертово пространство,
$H_{i},\,1\ls i\ls n$ "--- система его подпространств с соответствующими ортопроекторами
$P_{i}$.

\begin{definition}\label{DEF:IBAP}
Будем говорить, что $n$-ка подпространств
$H_{1},\ldots,H_{n}$ имеет свойство обратного наилучшего приближения
относительно набора линейных непрерывных операторов $A_{i},\,1\ls i\ls n$, если
\begin{enumerate}
\item
$\tmd{Im}(A_{i})\subset H_{i},\,1\ls i\ls n$,
\item
для
произвольных $u_{1}\in H_{1},\ldots,u_{n}\in H_{n}$ найдётся
$x\in H$ такой, что
$A_{k}x=u_{k},\,1\ls k\ls n$.
\end{enumerate}
\end{definition}

Системы подпространств, обладающие свойством обратного наилучшего приближения относительно
$A_{i}=P_{i}$, изучаются в
\cite{IBAP}. Далее мы предполагаем условие
$(1)$ определения
\ref{DEF:IBAP} выполненым.

Определим оператор
$A:H\to\bi_{k=1}^{n}H_{k}$ равенством
$Ax=(A_{1}x,\ldots,A_{n}x),\,x\in H$, тогда
$A^{*}(y_{1},\ldots,y_{n})=\sum_{k=1}^{n}A_{k}^{*}y_{k}$.
Система подпространств
$H_{i},\,1\ls i\ls n$ обладает свойством обратного наилучшего приближения относительно
$A_{i},\,1\ls i\ls n$ тогда и только тогда, когда
$\tmd{Im}(A)=\bi_{k=1}^{n}H_{k}$, что равносильно тому, что
$A^{*}$ является изоморфным вложением (т.е.
$\|A^{*}y\|\gs\ve\|y\|,\,y\in\bi_{k=1}^{n}H_{k}$ для некоторого $\ve>0$). Последнее равносильно существованию
$\ve>0$ такого, что для произвольных
$y_{i}\in H_{i},\,1\ls i\ls n$,
\begin{equation}\label{E:IBAP1}
\|\sum_{k=1}^{n}A_{k}^{*}y_{k}\|\gs\ve\sum_{k=1}^{n}\|y_{k}\|
\end{equation}
(мы перешли к эквивалентной норме).

\begin{statement}\label{ST:IBAP1}
Система подпространств
$H_{i},\,1\ls i\ls n$ обладает свойством обратного наилучшего приближения относительно набора
операторов
$A_{i},\,1\ls i\ls n$ тогда и только тогда, когда выполнены следующие два условия:
\begin{enumerate}
\item
$A_{k}^{*}\upr_{H_{k}}$ "--- изоморфное вложение,
\item
подпространства
$A_{k}^{*}(H_{k}),\,1\ls k\ls n$ линейно независимы и их сумма замкнута.
\end{enumerate}
\end{statement}
\begin{proof}
$\Rightarrow$
Из неравенства
\ref{E:IBAP1} следует, что
$\|A_{k}^{*}y_{k}\|\gs\ve\|y_{k}\|,\,y_{k}\in H_{k}$, т.е.
$A_{k}^{*}\upr_{H_{k}}$ является изоморфным вложением.
Поэтому
$A_{k}^{*}(H_{k})$ "--- подпространство. С помощью неравенства
\ref{E:IBAP1} легко получить
$(2)$.

$\Leftarrow$
Из
$(1)$ следует, что существует
$\ve_{1}>0$ такое, что для
$k=1,2,\ldots,n$,
$y_{k}\in H_{k}$ выполнено
$\|A_{k}^{*}y_{k}\|\gs\ve_{1}\|y_{k}\|$.
Из
$(2)$ следует, что существует
$\ve_{2}>0$, такое, что для произвольных
$z_{i}\in A_{i}^{*}(H_{i}),\,1\ls i\ls n$ выполнено
$\|\sum_{i=1}^{n}z_{i}\|\gs\ve_{2}\sum_{i=1}^{n}\|z_{i}\|$.
Тогда для произвольных
$y_{i}\in H_{i},\,1\ls i\ls n$, имеем
\begin{equation*}
\|\sum_{i=1}^{n}A_{i}^{*}y_{i}\|\gs\ve_{2}\sum_{i=1}^{n}\|A_{i}^{*}y_{i}\|\gs\ve_{1}\ve_{2}\sum_{i=1}^{n}\|y_{i}\|,
\end{equation*}
откуда следует требуемое утверждение.
\end{proof}

\begin{statement}\label{ST:IBAP2}
Пусть
$\ker(A_{k})=\ker(A_{k}^{*}),\,1\ls k\ls n$.
Система подпространств
$H_{k},\,1\ls k\ls n$ обладает свойством обратного наилучшего приближения относительно набора операторов
$A_{k},\,1\ls k\ls n$ тогда и только тогда, когда выполнены следующие два условия:
\begin{enumerate}
\item
$a_{k}=A_{k}\upr_{H_{k}}:H_{k}\to H_{k}$ обратим, $k=1,2,\ldots,n$,
\item
$H_{1},\ldots,H_{n}$ линейно независимы и их сумма замкнута.
\end{enumerate}
\end{statement}
\begin{proof}
Относительно ортогонального разложения
$H=H_{k}\oplus H_{k}^{\bot}$ оператор
$A_{k}=a_{k}\oplus 0_{H_{k}^{\bot}}$.

$\Rightarrow$
Из предыдущего утверждения следует, что
$a_{k}^{*}$ является изоморфным вложением. Поэтому
$\tmd{Im}(a_{k})=H_{k}$. Поскольку $\ker(A_{k})=\ker(A_{k}^{*})$, то
$\ker(a_{k})=\ker(a_{k}^{*})=0$.
Поэтому
$a_{k}$ обратим.
Следовательно, $a_{k}^{*}$ также обратим,
$\tmd{Im}(a_{k}^{*})=H_{k}$. Из утверждения
\ref{ST:IBAP1} следует, что
$H_{1},\ldots,H_{n}$ линейно независимы и их сумма замкнута.

$\Leftarrow$
Следует из утверждения
\ref{ST:IBAP1}.
\end{proof}

\begin{example}
Пусть
$A_{k}=P_{k},\,1\ls k\ls n$.
Система подпространств
$H_{1},\ldots,H_{n}$ обладает свойством обратного наилучшего приближения относительно
$P_{k},\,1\ls k\ls n$ тогда и только тогда, когда
$H_{1},\ldots,H_{n}$ линейно независимы и их сумма замкнута.

Другие критерии того, что
$H_{1},\ldots,H_{n}$ обладает свойством
обратного наилучшего приближения относительно $P_{k},\,1\ls k\ls n$
см. в
\cite{IBAP}, теорема 2.8.
Отметим, что эти критерии следуют непосредственно
из приведенного критерия и критериев замкнутости суммы пары подпространств.
\end{example}

\subsection{Спектральные свойства линейной комбинации ортопроекторов
на линейно независимые подпространства $H_{1},\ldots,H_{n}$ с суммой
$H_{1}+\ldots+H_{n}=H$}\label{SS:linearcombination}

При изучении спектральных свойств линейной комбинации ортопроекторов
нам будет нужно следующее утверждение.

\begin{statement}
Пусть
$H_{1},\ldots,H_{n}$ "---
подпространства $H$,
$\alpha_{1},\ldots,\alpha_{n}$ "--- положительные числа.
Следующие утверждения эквивалентны:
\begin{enumerate}
\item для произвольных $x_{i}\in H_{i},\,1\ls i\ls n$, справедливо:
\begin{equation*}
\|x_{1}+\ldots+x_{n}\|^{2}\gs\dfrac{1}{\alpha_{1}}\|x_{1}\|^{2}+\ldots+\dfrac{1}{\alpha_{n}}\|x_{n}\|^{2}
\end{equation*}

\item $H_{1},\ldots,H_{n}$ линейно независимы и
$\sigma(\alpha_{1}P_{1}+\ldots+\alpha_{n}P_{n})\cap(0,1)=\varnothing$.
\end{enumerate}
\end{statement}
\begin{proof}
$(1)\Rightarrow (2)$
Очевидно,
$H_{1},\ldots,H_{n}$ линейно независимы. Для $x\in H$ положим
$x_{i}=\alpha_{i}P_{i}x,\,1\ls i\ls n$. Тогда имеем:
\begin{equation*}
((\sum_{i=1}^{n}\alpha_{i}P_{i})^{2}x,x)\gs((\sum_{i=1}^{n}\alpha_{i}P_{i})x,x).
\end{equation*}
Поэтому
$(\sum_{i=1}^{n}\alpha_{i}P_{i})^{2}\gs\sum_{i=1}^{n}\alpha_{i}P_{i}$,
т.е.
$\sigma(\alpha_{1}P_{1}+\ldots+\alpha_{n}P_{n})\cap(0,1)=\varnothing$.

$(2)\Rightarrow (1)$
Ясно, что
$\sum_{k=1}^{n}H_{k}$ замкнуто и
$\sum_{k=1}^{n}\alpha_{k}P_{k}\gs P_{H_{1}+\ldots+H_{n}}$.
Поэтому
$H_{1}+\ldots+H_{n}=\tmd{Im}(\alpha_{1}P_{1}+\ldots+\alpha_{n}P_{n})$.
Поэтому
для произвольных
$x_{i}\in H_{i},\,1\ls i\ls n$ существует
$x\in H$, для которого
$(\sum_{k=1}^{n}\alpha_{k}P_{k})x=\sum_{k=1}^{n}x_{k}$. Из линейной независимости
$H_{1},\ldots,H_{n}$ следует
$x_{k}=\alpha_{k}P_{k}x,\,1\ls k\ls n$.
Повторяя рассуждения $(1)\Rightarrow (2)$,
получим нужное.
\end{proof}

\begin{conclusion}
Пусть $H_{1},\ldots,H_{n}$ "--- линейно независимые подпространства $H$,
$\alpha_{1},\ldots,\alpha_{n},\varepsilon$ "--- положительные числа. Если
$\alpha_{1}P_{1}+\ldots+\alpha_{n}P_{n}\gs\ve I$,
то для произвольных (попарно различных) индексов
$i(1),\ldots,i(s)$
\begin{equation*}
\alpha_{i(1)}P_{i(1)}+\ldots+\alpha_{i(s)}P_{i(s)}\gs\ve P_{H_{i(1)}+\ldots+H_{i(s)}}
\end{equation*}
(из следствия \ref{C:sumlinind} следует замкнутость $H_{i(1)}+\ldots+H_{i(s)}$).
\end{conclusion}

Частный случай следующего утверждения
(для $\alpha_{1}=\ldots=\alpha_{n}=1$)
сформулирован в
~\cite{Raban}.

\begin{statement}\label{ST:linearcombination}
Пусть
$H_{1},\ldots,H_{n}$ "---  ненулевые линейно независимые подпространства
$H$, $\alpha_{1},\ldots,\alpha_{n},\ve$ "--- положительные числа. Если
$\alpha_{1}P_{1}+\ldots+\alpha_{n}P_{n}\gs\ve I$, то
\begin{equation*}
\alpha_{1}P_{1}+\ldots+\alpha_{n}P_{n}\ls(\sum_{i=1}^{n}\alpha_{i}-(n-1)\ve)I.
\end{equation*}
\end{statement}

\begin{proof}
Сразу отметим, что для всех
$1\ls i\ls n$ выполнено:
$\alpha_{i}\gs\ve$.
Докажем требуемое утверждение индукцией по
$n$.
При
$n=2$ оно следует из свойств спектра линейной комбинации двух ортопроекторов
(см. утверждение \ref{ST:pairsymmetry}).
Сделаем шаг индукции. Пусть
$\alpha_{1}P_{1}+\ldots+\alpha_{n+1}P_{n+1}\gs\ve I$.
Тогда
\begin{equation*}
\alpha_{1}P_{1}+\ldots+\alpha_{n}P_{n}\gs\ve P_{H_{1}+\ldots+H_{n}},
\end{equation*}
а поэтому из предположения индукции:
\begin{equation*}
\alpha_{1}P_{1}+\ldots+\alpha_{n}P_{n}\ls(\sum_{i=1}^{n}\alpha_{i}-(n-1)\ve) P_{H_{1}+\ldots+H_{n}}.
\end{equation*}
Поэтому
\begin{equation*}
(\sum_{i=1}^{n}\alpha_{i}-(n-1)\ve) P_{H_{1}+\ldots+H_{n}}+\alpha_{n+1}P_{n+1}\gs\ve I.
\end{equation*}
Из свойств спектра линейной комбинации двух ортопроекторов:
\begin{equation*}
(\sum_{i=1}^{n}\alpha_{i}-(n-1)\ve) P_{H_{1}+\ldots+H_{n}}+\alpha_{n+1}P_{n+1}
\ls(\sum_{i=1}^{n+1}\alpha_{i}-n\ve)I,
\end{equation*}
а поэтому
\begin{equation*}
\alpha_{1}P_{1}+\ldots+\alpha_{n+1}P_{n+1}\ls(\sum_{i=1}^{n+1}\alpha_{i}-n\ve)I,
\end{equation*}
что и требовалось доказать.
\end{proof}

\begin{conclusion}
Пусть
$H_{1},\ldots,H_{n}$ "---
ненулевые подпространства
$H$,
$\alpha_{1},\ldots,\alpha_{n},\ve$ "--- положительные числа.
Предположим, что
$\alpha_{1}P_{1}+\ldots+\alpha_{n}P_{n}\gs\ve I$, и для некоторых
$t\ls n-1$ индексов $i~H_{i}\bigcap\ol{\sum_{j\neq i}H_{j}}=0$.
Тогда
\begin{equation*}
\alpha_{1}P_{1}+\ldots+\alpha_{n}P_{n}\ls (\sum_{i=1}^{n}\alpha_{i}-t\ve)I.
\end{equation*}
\end{conclusion}

\begin{proof}Можно считать, что
индексами из условия являются
$1,\ldots,t$.

Подпространства
$H_{1},\ldots,H_{t},\ol{H_{t+1}+\ldots+H_{n}}$ линейно независимы и
\begin{equation*}
\alpha_{1}P_{1}+\ldots+\alpha_{t}P_{t}+(\sum_{i=t+1}^{n}\alpha_{i})P_{\ol{H_{t+1}+\ldots+H_{n}}}
\gs\ve I.
\end{equation*}
Поэтому
\begin{equation*}
\alpha_{1}P_{1}+\ldots+\alpha_{t}P_{t}+(\sum_{i=t+1}^{n}\alpha_{i})P_{\ol{H_{t+1}+\ldots+H_{n}}}
\ls(\sum_{i=1}^{n}\alpha_{i}-t\ve)I,
\end{equation*}
откуда следует нужное.
\end{proof}

\subsection{<<Сведение>> системы подпространств $H_{1},\ldots,H_{n}$
к линейно независимой системе подпространств с сохранением суммы подпространств}

Пусть $H_{1},H_{2}$ "--- подпространства гильбертова пространства $H$.
Определим
$M_{1}=H_{1}$ и $M_{2}=H_{2}\ominus(H_{1}\cap H_{2})$.
Тогда подпространства $M_{1},M_{2}$
линейно независимы и сумма
$M_{1}+M_{2}=H_{1}+H_{2}$.
Таким образом, имея пару подпространств, мы
можем их <<уменьшить>> до линейно независимых, не меняя
суммы. В этом разделе мы докажем аналогичное утверждение
для произвольной $n$-ки подпространств.

\begin{theorem}\label{T:sumclotolinind}
Для каждого $n\gs 2$ существует постоянная
$c_{n}~(0<c_{n}<1)$ с таким свойством:
если $H_{1},\ldots,H_{n}$ "--- подпространства
$H$ и $0<\ve<1$ такое, что $P_{H_{1}}+\ldots+P_{H_{n}}\gs\ve I$,
то существуют подпространства
$M_{2},\ldots,M_{n}$ такие, что:
\begin{enumerate}
\item
$M_{k}\subset H_{k}$ для всех $2\ls k\ls n$,
\item подпространства
$H_{1},M_{2},\ldots,M_{n}$ линейно независимы,
\item
$P_{H_{1}}+P_{M_{2}}+\ve P_{M_{3}}+\ldots+\ve^{n-2}P_{M_{n}}\geq c_{n}\ve^{n-1}I$.
\end{enumerate}
\end{theorem}

Для доказательства теоремы \ref{T:sumclotolinind}
нам нужна следующая лемма.

\begin{lemma}\label{L:sumtwotoclosed}
Пусть
$H_{1},H_{2}$ "--- подпространства
$H$. Для произвольного
$0<\ve<1$ существует подпространство
$M_{2}\subset H_{2}$, такое, что:
\begin{enumerate}
\item
$H_{1}+M_{2}$ "--- подпространство
\item
$3(P_{H_{1}}+P_{M_{2}})+\ve I\gs P_{H_{1}}+P_{H_{2}}$
\item
$P_{H_{1}}+P_{M_{2}}\gs\dfrac{\ve}{4}P_{H_{1}+M_{2}}$.
\end{enumerate}
\end{lemma}

\begin{proof}
Достаточно доказать утверждение
леммы для случая, когда
$H=K\oplus K$ для некоторого гильбертова пространства
$K$, а ортопроекторы
\begin{equation*}
P_{H_{1}}=\begin{pmatrix}
a&\sqrt{a(I-a)}\\
\sqrt{a(I-a)}&I-a
\end{pmatrix},\,
P_{H_{2}}=\begin{pmatrix}
I&0\\
0&0
\end{pmatrix}
\end{equation*}
для некоторого самосопряженного оператора
$a:K\to K$,
$0\ls a=a^{*}\ls I,\,\ker(a)=\ker(I-a)=0$
(см. подраздел \ref{SS:spectheorem}).

Тогда
$H_{1}=\{(\sqrt{a}x,\sqrt{I-a}x),\,x\in K\}$,
$H_{2}=\{(x,0),\,x\in K\}$.
Рассмотрим спектральное представление
$a=\int_{[0,1)}x\,dE(x)$. Для числа
$\delta<1$, которое мы определим позже,
определим подпространство
$K_{1}=E([0,\delta))K$. Пусть
$q=E([0,\delta))$. Определим подпространство
$M_{2}=\{(x,0),\,x\in K_{1}\}$.
Поскольку
$|(\sqrt{a}x,x_{1})|\ls\sqrt{\delta}\|x\|\|x_{1}\|,\,x\in K,\,x_{1}\in K_{1},$
то
$H_{1}+M_{2}$ "--- подпространство, причём
$P_{H_{1}}+P_{M_{2}}\gs(1-\sqrt{\delta})P_{H_{1}+M_{2}}\gs\dfrac{1-\delta}{2}P_{H_{1}+M_{2}}$.

Поскольку
$P_{M_{2}}=\begin{pmatrix}
q&0\\
0&0
\end{pmatrix}$, то неравенство
$3(P_{H_{1}}+P_{M_{2}})+\ve I\gs P_{H_{1}}+P_{H_{2}}$ равносильно неотрицательности оператора
\begin{equation*}
A=\begin{pmatrix}
3q+2a+(\ve-1)I&2\sqrt{a(I-a)}\\
2\sqrt{a(I-a)}&2(I-a)+\ve I
\end{pmatrix}.
\end{equation*}
Точка
$\lambda\in \sigma(A)$ тогда и только тогда, когда
операторный определитель
$\det(\lambda I-A)$ не обратим. Операторный определитель
имеет вид:
\begin{align*}
&\lambda^{2}I-(3q+(2\ve+1)I)\lambda+(3q+2a+(\ve-1)I)(2(I-a)+\ve I)-4a(I-a)=\\
&\int_{[0,1)}\left(\lambda^{2}-(3\cdot I_{[0,\delta)}+2\ve+1)\lambda+
(3\cdot I_{[0,\delta)}+2x+\ve-1)(2(1-x)+\ve)-4x(1-x)\right)\,dE(x).
\end{align*}

Для доказательства того, что при
$\lambda<0$~ $\det(\lambda I-A)$ обратим
достаточно доказать, что при
$x\in[0,1)$ справедливо неравенство
\begin{equation*}
(3\cdot I_{[0,\delta)}+2x+\ve-1)(2(1-x)+\ve)\gs 4x(1-x).
\end{equation*}

При
$x\in [0,\delta)$ это неравенство имеет вид
$(2x+\ve+2)(2(1-x)+\ve)\gs 4x(1-x)$ и есть очевидным.

При
$x\in [\delta,1)$ рассматриваемое неравенство имеет вид
$(2x-(1-\ve))(2(1-x)+\ve)\gs 4x(1-x)$, т.е.
$x\gs 1-\dfrac{\ve}{2}-\dfrac{\ve^{2}}{2}.$
Поэтому выбрав
$\delta=1-\dfrac{\ve}{2}$ будем иметь нужное.
Осталось заметить, что при таком $\delta$ выполнено:
$P_{H_{1}}+P_{M_{2}}\gs\dfrac{\ve}{4}P_{H_{1}+M_{2}}$.
\end{proof}

\begin{proof}[Доказательство теоремы ~\ref{T:sumclotolinind}]

Доказываем нужное утверждение индукцией по $n$.
При $n=2$ подойдёт $c_{2}=\frac{1}{2}$. Пусть имеем
$n\geq 3$ подпространств $H_{1},\ldots,H_{n}$ таких, что
$P_{H_{1}}+\ldots+P_{H_{n}}\gs\ve I$.
Пусть
$H_{2}^{\prime}=H_{2}\ominus (H_{1}\cap H_{2})$.
Тогда $P_{H_{1}}+P_{H_{2}^{\prime}}+\ldots+P_{H_{n}}\gs\dfrac{\ve}{2}I$.

Из леммы~\ref{L:sumtwotoclosed}
следует, что существует подпространство
$M_{2}\subset H_{2}^{\prime}$, такое, что
\begin{enumerate}
\item
$H_{1}+M_{2}$ "--- подпространство,
\item
$3(P_{H_{1}}+P_{M_{2}})+\dfrac{\ve}{4}I\gs P_{H_{1}}+P_{H_{2}^{\prime}}$,
\item
$P_{H_{1}}+P_{M_{2}}\gs\dfrac{\ve}{16}P_{H_{1}+M_{2}}$.
\end{enumerate}
Тогда
\begin{equation*}
P_{H_{1}}+P_{M_{2}}+P_{H_{3}}+\ldots+P_{H_{n}}\gs\dfrac{\ve}{12}I,
\end{equation*}
а потому
\begin{equation*}
P_{H_{1}+M_{2}}+P_{H_{3}}+\ldots+P_{H_{n}}\gs\dfrac{\ve}{24}I.
\end{equation*}
Используем предположение индукции для набора подпространств
$(H_{1}+M_{2}),H_{3},\ldots,H_{n}$. Получим:
существуют подпространства
$M_{k}\subset H_{k},\,3\ls k\ls n$ такие, что
$(H_{1}+M_{2}),M_{3},\ldots,M_{n}$ линейно независимы и
\begin{equation*}
P_{H_{1}+M_{2}}+\sum_{k=3}^{n}(\dfrac{\ve}{24})^{k-3}P_{M_{k}}\gs c_{n-1}\left(\dfrac{\ve}{24}\right)^{n-2}I.
\end{equation*}
Домножая обе части на $\ve/16$, получим
\begin{equation*}
P_{H_{1}}+P_{M_{2}}+\sum_{k=3}^{n}\dfrac{\ve^{k-2}}{16\cdot24^{k-3}}P_{M_{k}}\gs c_{n-1}\left(\dfrac{\ve}{24}\right)^{n-2}\cdot\frac{\ve}{16}I.
\end{equation*}
Поэтому можно положить
$c_{n}=\dfrac{c_{n-1}}{16\cdot24^{n-2}}.$
\end{proof}

\begin{conclusion}\label{C:sumclotolinind}
Пусть
$H_{1},\ldots,H_{n}$ "--- подпространства
$H$, сумма которых
$H_{1}+\ldots+H_{n}=H$. Тогда существуют подпространства
$M_{k},\,2\ls k\ls n$ такие, что:
\begin{enumerate}
\item
$M_{k}\subset H_{k}$ для всех $2\ls k\ls n$,
\item
$H_{1},M_{2},\ldots,M_{n}$ линейно независимы,
\item
$H_{1}+M_{2}+\ldots+M_{n}=H.$
\end{enumerate}
\end{conclusion}

\begin{conclusion}
Пусть
$H_{1},\ldots,H_{n}$ "--- подпространства
$H$, причём
$H_{1}+\ldots+H_{n}=H$. Пусть
$\Delta$ "--- подпространство в
$H$. Существуют линейно независимые подпространства
$M_{j},\,1\ls j\ls n$, такие, что:
\begin{enumerate}
\item
$M_{j}\subset H_{j}$ для всех $1\ls j\ls n$
\item
$M_{1}+\ldots+M_{n}$ "--- подпространство
\item
$M_{1}+\ldots+M_{n}+\Delta=H$
\item
$(M_{1}+\ldots+M_{n})\bigcap\Delta=0$
\end{enumerate}
\end{conclusion}

\begin{proof}
Из следствия \ref{C:sumclotolinind} для набора подпространств
$\Delta,H_{1},\ldots,H_{n}$ получим нужное.
\end{proof}

Следующее утверждение показывает, что в следствии \ref{C:sumclotolinind}
подпространства можно заменить на образы линейных непрерывных операторов.

\begin{conclusion}
Пусть $K_{1},\ldots,K_{n}$ "--- гильбертовы пространства,
$a_{i}:K_{i}\to H,\,1\ls i\ls n$ "--- линейные непрерывные операторы, причём
$\tmd{Im}(a_{1})+\ldots+\tmd{Im}(a_{n})=H$. Тогда существуют
линейно независимые подпространства
$M_{j},\,1\ls j\ls n$, такие, что
$M_{j}\subset \tmd{Im}(a_{j}),\,1\ls j\ls n$
и сумма
$M_{1}+\ldots+M_{n}=H$.
\end{conclusion}
\begin{proof}
$\tmd{Im}(a_{1})+(\sum_{i\neq 1}\tmd{Im}(a_{i}))=H$.
Из теоремы 2.4 работы
\cite{Fillmore} (см. также раздел \ref{S:Imagesofoperators})
следует, что существует подпространство
$H_{1}\subset\tmd{Im}(a_{1})$, такое, что
$H_{1}+\tmd{Im}(a_{2})+\ldots+\tmd{Im}(a_{n})=H$.
Аналогично рассуждая, найдём набор подпространств
$H_{2}\subset\tmd{Im}(a_{2}),\ldots,H_{n}\subset\tmd{Im}(a_{n})$, таких, что
$H_{1}+\ldots+H_{n}=H$. Осталось воспользоваться следствием
\ref{C:sumclotolinind}.
\end{proof}

Теперь мы докажем, что произвольную $n$-ку подпространств
можно <<уменьшить>> до линейно независимой $n$-ки с сохранением суммы подпространств. Введём необходимые обозначения.
Для подпространств
$H_{1},\ldots,H_{n}$ определим гильбертово пространство
$\wt{H}=H_{1}\oplus\ldots\oplus H_{n}$, и в нём подпространства
$\Delta_{0}=\{(x_{1},\ldots,x_{n}),\,\sum_{k=1}^{n}x_{k}=0\}$
и
$\wt{H}_{k}=\{(0,\ldots,0,\underbrace{x}_{k},0\ldots,0),\,x\in H_{k}\},\,1\ls k\ls n$.

\begin{theorem}\label{T:umenshenie1}
Пусть
$H_{1},\ldots,H_{n}$ "--- подпространства
$H$. Тогда существуют подпространства
$M_{j}\subset H_{j},\,2\ls j\ls n$ такие, что:
\begin{enumerate}
\item
$H_{1},M_{2},\ldots,M_{n}$ линейно независимы
\item
$H_{1}+M_{2}+\ldots+M_{n}=H_{1}+\ldots+H_{n}.$
\end{enumerate}
\end{theorem}

\begin{remark} Эта теорема означает следующее:
можно не <<уменьшать>> $H_{1}$, <<уменьшить>> подпространства
$H_{2},\ldots,H_{n}$ так, что сумма не изменится, но
<<уменьшенные>> подпространства уже линейно независимы.
\end{remark}

\begin{proof}
Заметим, что
$\Delta_{0}\cap\wt{H}_{1}=0$ и
$\Delta_{0}+\wt{H}_{1}=\{(x_{1},\ldots,x_{n}),\sum_{i=1}^{n}x_{i}\in H_{1}\}$
замкнуто в $\wt{H}$.
Используя следствие \ref{C:sumclotolinind} для набора подпространств
$\Delta_{0}+\wt{H}_{1},\wt{H}_{2},\ldots,\wt{H}_{n}$
получим, что существуют подпространства
$M_{k}\subset H_{k},\,2\ls k\ls n$, такие, что
\begin{equation*}
(H_{1}\oplus M_{2}\oplus\ldots\oplus M_{n})\cap\Delta_{0}=0,\,
(H_{1}\oplus M_{2}\oplus\ldots\oplus M_{n})+\Delta_{0}=\wt{H}.
\end{equation*}
Тогда
$H_{1},M_{2},\ldots,M_{n}$ "--- линейно независимы и
$H_{1}+M_{2}+\ldots+M_{n}=H_{1}+\ldots+H_{n}.$
\end{proof}

Теперь естественно возникает следующий вопрос:
пусть для некоторого
$m<n$ подпространства
$H_{1},\ldots,H_{m}$ линейно независимы. Можно ли
<<уменьшить>> только
$H_{m+1},\ldots,H_{n}$, чтобы
сумма не изменилась, но
<<уменьшенные>> пространства были линейно независимы.
В связи с этим вопросом дадим следующее определение.
(Систему подпространств $H_{1},\ldots,H_{n}$ пространства $H$ будем обозначать
$S=(H;H_{1},\ldots,H_{n})$.)

\begin{definition}\label{DEF:opredelenie1}
Для
$m\ls n$ определим
$\textbf{RPS}(H,m,n)$ (reduction preserving sum) как множество $n$-ок
подпространств $S=(H;H_{1},\ldots,H_{n})$, для которых существуют
подпространства
$M_{j}\subset H_{j},\,m+1\ls j\ls n$, такие, что
\begin{enumerate}
\item
$H_{1},\ldots,H_{m},M_{m+1},\ldots,M_{n}$ линейно независимы
\item
$H_{1}+\ldots+H_{m}+M_{m+1}+\ldots+M_{n}=H_{1}+\ldots+H_{n}$
\end{enumerate}
\end{definition}

\begin{remark}
Если $S\in\textbf{RPS}(H,m,n)$, то $H_{1},\ldots,H_{m}$
линейно независимы.
\end{remark}

Мы изучим некоторые свойства классов
$\textbf{RPS}(H,m,n)$. Из теоремы \ref{T:umenshenie1}
следует, что
$\textbf{RPS}(H,1,n)$ совпадает с множеством
$n$-ок подпространств $H$,
кроме того, очевидно
$\textbf{RPS}(H,n,n)$
совпадает с множеством $n$-ок подпространств, для которых
$H_{1},\ldots,H_{n}$ линейно независимы. Далее
$m<n$.

Используя следствие
\ref{C:sumlinind} и теорему \ref{T:umenshenie1}, легко получить
следующее утверждение.

\begin{statement}
Пусть
$H_{1}+\ldots+H_{n}$ "--- подпространство.
$S\in\textbf{RPS}(H,m,n)$ тогда и только тогда, когда
$H_{1},\ldots,H_{m}$ линейно независимы и
$H_{1}+\ldots+H_{m}$ "--- подпространство.
\end{statement}

\begin{statement}\label{ST:RPS}
Утверждения эквивалентны:
\begin{enumerate}
\item
$S\in\textbf{RPS}(H,m,n)$,
\item
$H_{1},\ldots,H_{m}$ линейно независимы и
$\Delta_{0}+\wt{H}_{1}+\ldots+\wt{H}_{m}$ "--- подпространство в
$\wt{H}$,
\item
существует
$\ve>0$, такое, что для произвольных
$y_{j}\in H_{j},\,1\ls j\ls n,\,\sum_{j=1}^{n}y_{j}=0$
справедливо:
\begin{equation*}
\|y_{m+1}\|+\ldots+\|y_{n}\|\gs\ve(\|y_{1}\|+\ldots+\|y_{m}\|).
\end{equation*}
\end{enumerate}
\end{statement}
\begin{proof}
$(1)\Rightarrow (2)$.
Пусть
$S\in\textbf{RPS}(H,m,n)$.
Для соответствующих подпространств
$M_{j},\,m+1\ls j\ls n$ (см. определение \ref{DEF:opredelenie1})
обозначим
\begin{equation*}
\wt{M}_{j}=\{(0,\ldots,0,\underbrace{x}_{j},0,\ldots,0),\,x\in M_{j}\}.
\end{equation*}
Тогда подпространства
$\Delta_{0},\wt{H}_{1},\ldots,\wt{H}_{m},\wt{M}_{m+1},\ldots,\wt{M}_{n}$
линейно независимы и их сумма равна
$\wt{H}$. Из следствия
\ref{C:sumlinind} следует, что
$\Delta_{0}+\wt{H}_{1}+\ldots+\wt{H}_{m}$ "--- подпространство в
$\wt{H}$.

$(2)\Rightarrow (1)$.
Пусть
$H_{1},\ldots,H_{m}$ линейно независимы и
$\Delta_{0}+\wt{H}_{1}+\ldots+\wt{H}_{m}$ замкнуто в $\wt{H}$.
Применив теорему \ref{T:umenshenie1} к набору подпространств
$(\Delta_{0}+\wt{H}_{1}+\ldots+\wt{H}_{m}),\wt{H}_{m+1},\ldots,\wt{H}_{n}$, получим:
существуют подпространства
$M_{j}\subset H_{j},\,m+1\ls j\ls n$, для которых
\begin{equation*}
(H_{1}\oplus\ldots\oplus H_{m}\oplus M_{m+1}\oplus\ldots\oplus M_{n})\cap\Delta_{0}=0
\end{equation*}
и
\begin{equation*}
(H_{1}\oplus\ldots\oplus H_{m}\oplus M_{m+1}\oplus\ldots\oplus M_{n})+\Delta_{0}=\wt{H}.
\end{equation*}
Эти
$M_{j},\,m+1\ls j\ls n$ "--- требуемые.

$(2)\Leftrightarrow(3)$.
$H_{1},\ldots,H_{m}$ линейно независимы и сумма
$\Delta_{0}+\wt{H}_{1}+\ldots+\wt{H}_{m}$ "--- подпространство
тогда и только тогда, когда
$\Delta_{0}\cap(\sum_{i=1}^{m}\wt{H}_{i})=0$ и
$\Delta_{0}+(\sum_{i=1}^{m}\wt{H}_{i})$ замкнуто в
$\wt{H}$. В силу утверждения
\ref{ST:pair3}, пункт (5), последнее равносильно существованию такого
$\ve_{1}>0$, что для произвольного
$(y_{1},\ldots,y_{n})\in\Delta_{0}$
выполнено (нам удобнее перейти к эквивалентной норме):
\begin{equation*}
\sum_{j=m+1}^{n}\|y_{j}\|\gs\ve_{1}\sum_{i=1}^{n}\|y_{i}\|.
\end{equation*}
Отсюда следует, что
$(2)\Leftrightarrow(3)$.
\end{proof}

Отметим, что
\begin{equation*}
\Delta_{0}+\wt{H}_{1}+\ldots+\wt{H}_{m}=\{(x_{1},\ldots,x_{n})\in\wt{H},\,\sum_{j=m+1}^{n}x_{j}
\in H_{1}+\ldots+H_{m}\}.
\end{equation*}
Теперь из предыдущего утверждения получаем такие следствия.

\begin{conclusion}
Если
$H_{1},\ldots,H_{m}$ линейно независимы и
$(H_{1}+\ldots+H_{m})\cap(H_{m+1}+\ldots+H_{n})$ замкнуто в
$H_{m+1}+\ldots+H_{n}$, то
$S\in\textbf{RPS}(H,m,n)$.
\end{conclusion}

\begin{conclusion}
$S\in\textbf{RPS}(H,n-1,n)$ тогда и только тогда, когда
$H_{1},\ldots,H_{n-1}$ линейно независимы и
$H_{n}\cap(H_{1}+\ldots+H_{n-1})$ "--- подпространство.
\end{conclusion}

\begin{conclusion}
Пусть для двух систем подпространств
$S=(H;H_{1},\ldots,H_{n})$,
$S'=(H;H_{1}',\ldots,H_{n}')$
$H_{k}\subset H_{k}',\,1\ls k\ls n$.
Если
$S'\in\textbf{RPS}(H,m,n)$, то и
$S\in\textbf{RPS}(H,m,n)$.
\end{conclusion}

\begin{conclusion}
Если
$S\in\textbf{RPS}(H,m,n)$, то для всякого
$l\ls m$ и $1\ls r\ls n-m$
$S'=(H;H_{1},\ldots,H_{l},H_{m+1},\ldots,H_{m+r})\in\textbf{RPS}(H,l,l+r)$
\end{conclusion}

\begin{conclusion}
Если
$S\in\textbf{RPS}(H,m,n)$, то для всякого
$m+1\ls j\ls n$
$H_{j}\cap(H_{1}+\ldots+H_{m})$ "--- подпространство.
\end{conclusion}

\begin{statement}
Пусть множество
$\{m+1,\ldots,n\}$ разбито на подмножества
$I_{k},\,1\ls k\ls r$. Тогда:
\begin{enumerate}
\item
Если
$S'=(H;H_{1},\ldots,H_{m},\ol{\sum_{j\in I_{1}}H_{j}},\ldots,\ol{\sum_{j\in I_{r}}H_{j}})\in\textbf{RPS}(H,m,m+r)$, то
$S\in\textbf{RPS}(H,m,n)$,
\item
Пусть для всех
$1\ls k\ls r$
$\sum_{j\in I_{k}}H_{j}$ "--- подпространство.
Если
$S\in\textbf{RPS}(H,m,n)$ то
$S'=(H;H_{1},\ldots,H_{m},\sum_{j\in I_{1}}H_{j},\ldots,\sum_{j\in I_{r}}H_{j})\in\textbf{RPS}(H,m,m+r)$.
\end{enumerate}
\end{statement}
\begin{proof}
Докажем (1).
Для доказательства нужного утверждения используем утверждение \ref{ST:RPS}.
Пусть $y_{j}\in H_{j},\,1\ls j\ls n,\,\sum_{j=1}^{n}y_{j}=0$.
Для всех
$1\ls k\ls r$
$\sum_{j\in I_{k}}\|y_{j}\|\gs\|\sum_{j\in I_{k}}y_{j}\|$.
Отсюда легко получим, что
$S\in\textbf{RPS}(H,m,n)$.

Докажем (2). Пусть для всех $1\ls k\ls r$
сумма $\sum_{j\in I_{k}}H_{j}$ "--- подпространство,
$S\in\textbf{RPS}(H,m,n)$.
Для всех
$1\ls k\ls r$ существуют подпространства
$K_{j}\subset H_{j},\,j\in I_{k}$, такие, что
$K_{j},\,j\in I_{k}$ линейно независимы и сумма
$\sum_{j\in I_{k}}K_{j}=\sum_{j\in I_{k}}H_{j}$.
Тогда $n$-ка
\begin{equation*}
S''=(H;H_{1},\ldots,H_{m},K_{m+1},\ldots,K_{n})\in\textbf{RPS}(H,m,n).
\end{equation*}
Для всякого
$1\ls k\ls r$ существует
$\ve_{k}>0$, такое, что для произвольных
$y_{j}\in K_{j},\,j\in I_{k}$ справедливо:
$\|\sum_{j\in I_{k}}y_{j}\|\gs\ve_{k}(\sum_{j\in I_{k}}\|y_{j}\|)$.
Используя утверждение \ref{ST:RPS}, легко
получить, что
$(H;H_{1},\ldots,H_{m},\sum_{j\in I_{1}}K_{j},\ldots,\sum_{j\in I_{r}}K_{j})\in\textbf{RPS}(H,m,m+r)$,
что и требовалось доказать.
\end{proof}

\begin{conclusion}
Пусть
$H_{m+1}+\ldots+H_{n}$ "--- подпространство.
$S\in\textbf{RPS}(H,m,n)$ тогда и только тогда, когда
$H_{1},\ldots,H_{m}$ линейно независимы и
$(H_{1}+\ldots+H_{m})\cap(H_{m+1}+\ldots+H_{n})$ "--- подпространство.
\end{conclusion}

\section{Замкнутость суммы образов операторов}\label{S:Imagesofoperators}

В этом разделе мы рассмотрим более общий объект, чем систему подпространств "---
систему образов линейных непрерывных операторов. Отметим, что изучение таких систем не сводится
к изучению системы линеалов в $H$, поскольку в бесконечномерном гильбертовом пространстве
есть линеалы, отличные от образов линейных непрерывных операторов (см. \cite{Fillmore}).

Рассмотрение системы подпространств
$H_{i},\,1\ls i\ls n$ как системы
$H_{i}=\tmd{Im}(P_{H_{i}}),\,1\ls i\ls n$ позволит нам получить критерии замкнутости суммы подпространств,
а также изучить некоторые свойства сумм подпространств
$H$.

\subsection{Теорема Р. Дугласа и её следствия}\label{SS:Douglas}

\begin{theorem}(\cite{Douglas})
Пусть
$H,H_{1},H_{2}$ "--- гильбертовы пространства,
$A:H_{1}\to H$,
$B:H_{2}\to H$ "--- линейные непрерывные операторы.
Следующие условия эквивалентны:
\begin{enumerate}
\item $\tmd{Im}(A)\subset \tmd{Im}(B)$
\item Существует $\lambda>0$ такое, что $AA^{*}\ls\lambda BB^{*}$
\item Существует линейный непрерывный оператор $C:H_{1}\to H_{2}$ такой, что $A=BC$.
\end{enumerate}
Кроме того, при выполнении (1) оператор $C$ может быть выбран так, что
$\ker C=\ker A$, $\tmd{Im}(C)\subset(\ker B)^{\bot}$.
\end{theorem}

\begin{conclusion}
Пусть
$A:H_{1}\ra H$ "--- линейный непрерывный оператор.
Тогда
$\tmd{Im}(A)=\tmd{Im}(\sqrt{AA^{*}})$.
\end{conclusion}

\begin{conclusion}
Пусть
$H,H_{1},\ldots,H_{n}$ "--- гильбертовы пространства,
$a_{k}:H_{k}\ra H$ "--- линейные непрерывные операторы. Тогда
$\sum_{k=1}^{n}\tmd{Im}(a_{k})=\tmd{Im}(\sqrt{\sum_{k=1}^{n}a_{k}a_{k}^{*}})$.
\begin{proof}
Определим оператор
$A:H_{1}\oplus\ldots\oplus H_{n}\ra H$
равенством
\begin{equation*}
A(x_{1},\ldots,x_{n})=a_{1}x_{1}+\ldots+a_{n}x_{n}.
\end{equation*}
Блочная запись
$A$ имеет вид $(a_{1},\ldots,a_{n})$, поэтому
$A^{*}=(a_{1}^{*},\ldots,a_{n}^{*})^{\top}$.
Отсюда
$AA^{*}=\sum_{k=1}^{n}a_{k}a_{k}^{*}$.
Осталось заметить, что
$\sum_{k=1}^{n}\tmd{Im}(a_{k})=\tmd{Im}(A)=\tmd{Im}(\sqrt{AA^{*}})$.
\end{proof}
\end{conclusion}

\begin{conclusion}
Пусть
$H_{1},\ldots,H_{n}$ "--- подпространства
$H$,
$P_{1},\ldots,P_{n}$ "--- соответствующие ортопроекторы.
Тогда сумма
$H_{1}+\ldots+H_{n}=\tmd{Im}(\sqrt{P_{1}+\ldots+P_{n}})$.
Отсюда следует, что сумма $H_{1}+\ldots+H_{n}$ замкнута тогда и только тогда,
когда
$\sigma(P_{1}+\ldots+P_{n})\cap(0,\ve)=\varnothing$ для некоторого
$\varepsilon>0$.
\end{conclusion}

\begin{statement}
Пусть
$a_{k}:H\ra H,\,1\ls k\ls n$ "--- неотрицательные
самосопряженные операторы в $H$.
Следующие условия равносильны:
\begin{enumerate}
\item
$\tmd{Im}(a_{1})+\ldots+\tmd{Im}(a_{n})$ замкнуто,
\item
для некоторого $\ve>0$ $\sigma(a_{1}+\ldots+a_{n})\cap(0,\ve)=\varnothing.$
\end{enumerate}
При выполнении этих условий (одного из этих условий)
$\sum_{k=1}^{n}\tmd{Im}(a_{k})=\tmd{Im}(\sum_{k=1}^{n}a_{k})$.
\end{statement}
\begin{proof}
Определим подпространства
$H_{1}=\ol{\sum_{k=1}^{n}\tmd{Im}(a_{k})}$,
$H_{2}=\bigcap_{k=1}^{n}\ker(a_{k})$. Тогда
$H=H_{1}\oplus H_{2}$.
Относительно этого ортогонального разложения пространства $H$ оператор $a_{i}=b_{i}\oplus 0$,
где $b_{i}:H_{1}\to H_{1}$ "--- неотрицательный самосопряженный оператор. Ясно, что
$\bigcap_{k=1}^{n}\ker b_{k}=0$.

\textbf{1.}
Пусть
$\sum_{k=1}^{n}\tmd{Im}(a_{k})$ замкнуто.
Тогда
$\sum_{k=1}^{n}\tmd{Im}(b_{k})$ замкнуто. Поскольку
$\bigcap_{k=1}^{n}\ker(b_{k})=0$, то
$\sum_{k=1}^{n}\tmd{Im}(b_{k})$ плотно в
$H_{1}$. Поэтому
$\sum_{k=1}^{n}\tmd{Im}(b_{k})=H_{1}$, т.е.
$\tmd{Im}\sqrt{b_{1}^{2}+\ldots+b_n^2}=H_{1}$.
Тогда $b_{1}^{2}+\ldots+b_{n}^{2}\gs\ve_{1}I_{H_{1}}$ для некоторого $\ve_{1}>0$.
Поскольку
$\|b_{i}\|b_{i}\gs b_{i}^{2}$, то
$b_{1}+\ldots+b_{n}\gs\ve_{2}I_{H_{1}}$ для некоторого $\varepsilon_{2}>0$.
Поэтому
$\sigma(\sum_{k=1}^{n}a_{k})\cap(0,\ve_{2})=\varnothing$ и
$\tmd{Im}(a_{1}+\ldots+a_{n})=H_{1}$, что и требовалось доказать.

\textbf{2.}
Пусть
$\sigma(a_{1}+\ldots+a_{n})\cap(0,\ve)=\varnothing$ для некоторого $\ve>0$.
Тогда
$\sigma(b_{1}+\ldots+b_{n})\cap(0,\ve)=\varnothing$. Поскольку
$\ker(b_{1}+\ldots+b_{n})=0$, то
$b_{1}+\ldots+b_{n}\gs\ve I_{H_{1}}$. Следовательно,
$\tmd{Im}(a_{1}+\ldots+a_{n})=H_{1}=\sum_{k=1}^{n}\tmd{Im}(a_{k})$.
\end{proof}

\begin{example}
Пусть
$a_{1},\ldots,a_{n}$ "--- неотрицательные
операторы в
$H$, причём для произвольного
$1\ls k\ls n$
$\tmd{Im}(a_{k})$ "--- подпространство. Определим оператор
$C=a_{1}+\ldots+a_{n}$. Предположим,
$Ca_{k}=a_{k}C,\, 1\ls k\ls n$.
Тогда
$\tmd{Im}(a_{1})+\ldots+\tmd{Im}(a_{n})$ "---
подпространство.
\begin{proof}
Существует
$\ve>0$ такое, что
$\sigma(a_{k})\bigcap(0,\ve)=\varnothing$
для всех $1\leqslant k\leqslant n$.
Поскольку
$C\gs a_{k}$ и $Ca_{k}=a_{k}C$ то
$Ca_{k}\gs a_{k}^{2}\gs\ve a_{k},\, 1\ls k\ls n$.
Прибавив полученые неравенства, получим
$C^{2}\gs\ve C$, откуда
$\sigma(C)\bigcap(0,\ve)=\varnothing$.
\end{proof}
\end{example}

\subsection{Различные критерии замкнутости $\sum_{k=1}^{n}\tmd{Im}(T_{i})$}

В этом разделе
$T_{1},\ldots,T_{n}$ "--- неотрицательные операторы в $H$, причём
$\|T_{k}\|<2,\,1\ls k\ls n$. Определим $E=(I-T_{n})\ldots(I-T_{1})$.
Лемма
\ref{L:lemma4} следует из лемм 3.3, 3.10 работы \cite{Holst}.

\begin{lemma}\label{L:lemma4}
Пусть $\|T_{k}\|\ls\omega<2,\,1\ls k\ls n$. Тогда для всякого
$x\in H$
\begin{equation*}
\dfrac{2+\omega^{2}n(n-1)}{2-\omega}(\|x\|^{2}-\|Ex\|^{2})\gs\sum_{k=1}^{n}(T_{k}x,x).
\end{equation*}
\end{lemma}

Следующая лемма следует из леммы \ref{L:lemma4}, но мы приведём непосредственное доказательство.

\begin{lemma}\label{L:lemma5}
Если для некоторого
$x\in H$
$\|Ex\|=\|x\|$, то $T_{k}x=0,\,1\ls k\ls n$.
\end{lemma}
\begin{proof}
Отметим, что для неотрицательного оператора
$T$ с нормой $\|T\|<2$ и для произвольного $x\in H$
$\|(I-T)x\|\ls \|x\|$. Если
$\|(I-T)x\|=\|x\|$,
то
$Tx=0$. Действительно, пусть $\|(I-T)x\|=\|x\|$. Тогда
$\|x-Tx\|^{2}=\|x\|^{2}$, откуда
$\|Tx\|^{2}=2(Tx,x)$. Поскольку
$\|Tx\|^{2}\ls\|T\|(Tx,x)$, то
$(Tx,x)=0,\,Tx=0$.

Предположим, что
$\|Ex\|=\|x\|$. Поскольку
$\|Ex\|\ls\|(I-T_{1})x\|\ls\|x\|$, то
$\|(I-T_{1})x\|=\|x\|$, откуда
$T_{1}x=0$. Тогда
$Ex=(I-T_{n})\ldots(I-T_{2})x$. Продолжив аналогичные рассуждения, получим
$T_{2}x=\ldots=T_{n}x=0$.
\end{proof}

Обобщая определение работы
\cite{Protasov} на бесконечномерный случай, дадим следующее определение.

\begin{definition}
Пусть
$A_{1},\ldots,A_{n}$ "--- линейные непрерывные операторы в $H$,
$1\ls p<\infty$.
$p$-радиусом набора операторов
$A_{1},\ldots,A_{n}$ назовём число
\begin{equation*}
\hat{\rho}_{p}(A_{1},\ldots,A_{n})=\lim_{k\to\infty}\left(\dfrac{1}{n^{k}}\sum\|A_{i(1)}\ldots A_{i(k)}\|^{p}\right)^
{\frac{1}{pk}},
\end{equation*}
где сумма берётся по всем наборам индексов
$1\ls i(1),\ldots,i(k)\ls n$.
\end{definition}
Определение корректно.
Действительно, определим
\begin{equation*}
a_{k,p}=\left(\dfrac{1}{n^{k}}\sum\|A_{i(1)}\ldots A_{i(k)}\|^{p}\right)^{\frac{1}{p}}.
\end{equation*}
Ясно, что
$a_{k+l,p}\ls a_{k,p}a_{l,p},\,k,l\gs 1$.
Поэтому существует
$\lim_{k\to\infty}a_{k,p}^{\frac{1}{k}}$.
Различные свойства $p$-радиуса для операторов в конечномерном гильбертовом пространстве,
а также его применения содержатся в работе \cite{Protasov}.
Мы сформулируем условие замкнутости суммы образов операторов $T_{1},\ldots,T_{n}$
в терминах $p$-радиуса набора операторов $I-T_{1},\ldots,I-T_{n}$.

\begin{statement}
Пусть
$1\ls p<\infty$. Сумма
$\tmd{Im}(T_{1})+\ldots+\tmd{Im}(T_{n})=H$ тогда и только тогда, когда
$\hat{\rho}_{p}(I-T_{1},\ldots,I-T_{n})<1$.
\end{statement}
\begin{proof}
\textbf{1.}
Пусть
$\tmd{Im}(T_{1})+\ldots+\tmd{Im}(T_{n})\neq H$. Тогда для произвольного набора индексов
$i(1),\ldots,i(k)$ норма
$\|(I-T_{i(1)})\ldots(I-T_{i(k)})\|=1$. Действительно, если
$\|(I-T_{i(1)})\ldots(I-T_{i(k)})\|<1$, то оператор
$I-(I-T_{i(1)})\ldots(I-T_{i(k)})$ обратим, а потому
$\tmd{Im}(T_{1})+\ldots+\tmd{Im}(T_{n})=H$, противоречие. Отсюда
$\hat{\rho}_{p}(I-T_{1},\ldots,I-T_{n})=1$.

\textbf{2.}
Пусть теперь
$\tmd{Im}(T_{1})+\ldots+\tmd{Im}(T_{n})=H$.
Тогда $T_{1}+\ldots+T_{n}\gs\ve I$ для некоторого $\ve>0$.
Из леммы \ref{L:lemma4} следует, что
$\|(I-T_{n})\ldots(I-T_{1})\|<1$. Поэтому
$a_{n,p}(I-T_{1},\ldots,I-T_{n})<1$. Для всех $k\gs 1$
$a_{kn,p}\ls a_{n,p}^{k}$,
$a_{kn,p}^{\frac{1}{kn}}\ls a_{n,p}^{\frac{1}{n}}$.
Перейдя к границе при
$k\to\infty$, получаем
$\hat{\rho}_{p}(I-T_{1},\ldots,I-T_{n})\ls a_{n,p}^{\frac{1}{n}}<1$.
\end{proof}

Сформулируем условие замкнутости суммы образов операторов в терминах
порожденной ими $C^{*}$-алгебры. Обозначим $\mathcal{B}(H)$ множество всех линейных непрерывных операторов в $H$.
Пусть
$\mc{A}(T_{1},\ldots,T_{n})$ "--- алгебра, порожденная
$T_{1},\ldots,T_{n}$, и
$C^{*}(T_{1},\ldots,T_{n})$ "--- замыкание
$\mc{A}(T_{1},\ldots,T_{n})$ в
$\mc{B}(H)$.

\begin{statement}\label{ST:Cstaralgebras}
Ортопроектор
$P_{\ol{\tmd{Im}(T_{1})+\ldots+\tmd{Im}(T_{n})}}\in C^{*}(T_{1},\ldots,T_{n})$
тогда и только тогда, когда
$\tmd{Im}(T_{1})+\ldots+\tmd{Im}(T_{n})$ "--- подпространство.
\end{statement}
\begin{proof}
Без ограничения общности можно считать, что
$\sum_{k=1}^{n}\tmd{Im}(T_{k})$ плотна в
$H$.

\textbf{1.} Пусть $\tmd{Im}(T_{1})+\ldots+\tmd{Im}(T_{n})=H$. Тогда
$\|(I-T_{n})\ldots(I-T_{1})\|<1$, а поэтому
$I-((I-T_{n})\ldots(I-T_{1}))^{k}\to I,\,k \to\infty$.
Отсюда
$I\in C^{*}(T_{1},\ldots,T_{n})$.

\textbf{2.} Пусть теперь $I\in C^{*}(T_{1},\ldots,T_{n})$.
Существует элемент $a\in\mathcal{A}(T_{1},\ldots,T_{n})$ такой, что $\|I-a\|<1$.
Тогда $a$ обратим. Следовательно, $\tmd{Im}(a)=H$, а значит, и
$\tmd{Im}(T_{1})+\ldots+\tmd{Im}(T_{n})=H$.
\end{proof}

\begin{remark}
Из приведенного доказательства следует, что
$I\in C^{*}(T_{1},\ldots,T_{n})$ тогда и только тогда, когда
$\sum_{k=1}^{n}\tmd{Im}(T_{k})=H$. Это верно для произвольных
неотрицательных $T_{k}$ (не обязательно $\|T_{k}\|<2$).
Действительно, вместо операторов
$T_{k}$ можно рассмотреть операторы $\lambda T_{k}$, где
$\lambda>0$ таково, что $\lambda\|T_{k}\|<2,\,1\ls k\ls n$.
\end{remark}

\begin{conclusion}
Пусть алгебра
$\mc{A}=\mc{A}(T_{1},\ldots,T_{n})$ конечномерна.
Определим $T=\sum_{k=1}^{n}T_{k}$. Поскольку
$\mc{A}$ конечномерна, то для некоторого ненулевого полинома
$F$
$~F(T)=0$. Поэтому
$\sigma(T)$ конечен.
Тогда
$\sum_{k=1}^{n}\tmd{Im}(T_{k})$ замкнуто, кроме того,
$C^{*}(T_{1},\ldots,T_{n})=\mc{A}(T_{1},\ldots,T_{n})$.
Поэтому
$P_{\tmd{Im}(T_{1})+\ldots+\tmd{Im}(T_{n})}\in\mc{A}(T_{1},\ldots,T_{n})$.
\end{conclusion}

Из утверждения \ref{ST:Cstaralgebras} и утверждения
\ref{ST:compactproduct} следует следующий критерий замкнутости
суммы подпространств.

\begin{statement}
Пусть
$H_{1},\ldots,H_{n}$ "--- подпространства $H$, натуральное
$~m\ls n-1$. Предположим, что
оператор
$P_{i}P_{j}$ компактный для всех $\,1\ls i\ls m,\,m+1\ls j\ls n$.
Если
$H_{1}+\ldots+H_{m},H_{m+1}+\ldots+H_{n}$ "--- подпространства, то
$H_{1}+\ldots+H_{n}$ "--- подпространство, причём
\begin{equation*}
P_{H_{1}+\ldots+H_{n}}=P_{H_{1}+\ldots+H_{m}}+P_{H_{m+1}+\ldots+H_{n}}\pmod{S_{\infty}(H)}.
\end{equation*}
\end{statement}

Обратное утверждение неверно, как показывает следующий пример

\begin{example}
В гильбертовом пространстве
$H=l_{2}\bi l_{2}$ рассмотрим 3 подпространства
$H_{1},H_{2},H_{3}$, ортопроекторы на которые
\begin{equation*}
P_{1}=\begin{pmatrix}
0&0\\
0&I
\end{pmatrix},\,
P_{2}=\begin{pmatrix}
I&0\\
0&0
\end{pmatrix},\,
P_{3}=\begin{pmatrix}
I-a^{2}&a\sqrt{(I-a^{2})}\\
a\sqrt{(I-a^{2})}&a^{2}
\end{pmatrix},
\end{equation*}
где самосопряженный
$a:l_{2}\to l_{2}$,
$0\ls a\ls I$,
$\ker(a)=0$ и
$a$ компактный.
Тогда операторы
$P_{1}P_{2},\,P_{1}P_{3}$ компактны,
$H_{1}+H_{2}+H_{3}=H$, но
$H_{2}+H_{3}$ "--- не подпространство.
\end{example}

\subsection{Образы элементов алгебры, порожденной $T_{1},\ldots,T_{n}$}

Для последовательности
$\mc{I}=(i(1),\ldots,i(k)),\,1\ls i(r)\ls n,\,1\ls r\ls k$
определим оператор
\begin{equation*}
E_{\mc{I}}=(I-T_{i(1)})\ldots(I-T_{i(k)}).
\end{equation*}
Определим последовательность
$\mc{I}^{*}=(i(k),\ldots,i(1))$, тогда $E_{\mc{I}}^{*}=E_{\mc{I}^{*}}$.
Далее мы будем рассматривать оператор
$T=\sum_{\mc{I}\in\mc{U}}\alpha_{\mc{I}}E_{\mc{I}}$, где все
$\alpha_{\mc{I}}>0,\,\mc{I}\in\mc{U}$
(тут $\mc{U}$ "--- некоторое множество последовательностей).
Мы предполагаем, что для всякого $1\ls l\ls n$
найдётся последовательность
$\mc{I}\in\mc{U}$, в которой встречается $l$.

\begin{statement}\label{ST:statement3}
Следующие условия равносильны:
\begin{enumerate}
\item
$\sum_{k=1}^{n}\tmd{Im}(T_{k})$ замкнуто,
\item
$\tmd{Im}((\sum_{\mc{I}\in\mc{U}}\alpha_{\mc{I}})I-T)$ замкнуто.
\end{enumerate}
При выполнении этих условий (одного из этих условий)
\begin{equation*}
\sum_{k=1}^{n}\tmd{Im}(T_{k})=\tmd{Im}((\sum_{\mc{I}\in\mc{U}}\alpha_{\mc{I}})I-T).
\end{equation*}
\end{statement}
\begin{proof}
Без ограничения общности можно считать, что
$\sum_{k=1}^{n}\tmd{Im}(T_{k})$ плотно в
$H$.

\textbf{1.}
Пусть
$\tmd{Im}(T_{1})+\ldots+\tmd{Im}(T_{n})=H$. Тогда
$T_{1}+\ldots+T_{n}\gs\ve I$ для некоторого $\ve>0$.
Возьмём произвольный
$x\in H$. Из леммы
\ref{L:lemma4} следует, что для всякого
$\mc{I}\in\mc{U}$ существует $m_{\mc{I}}>0$, не зависящее от $x$, такое, что
$m_{\mc{I}}(\|x\|^{2}-\|E_{\mc{I}}x\|^{2})\gs\sum_{i\in \mc{I}}(T_{i}x,x)$. Прибавим эти неравенства.
Тогда
$\sum_{\mc{I}\in\mc{U}}m_{\mc{I}}(\|x\|^{2}-\|E_{\mc{I}}x\|^{2})\gs\ve\|x\|^{2}$.
Поэтому существует
$\mc{J}\in\mc{U}$ (зависящее от $x$), для которого $m_{\mc{J}}(\|x\|^{2}-\|E_{\mc{J}}x\|^{2})\gs\dfrac{\ve}{|\mc{U}|}\|x\|^{2}$
т.е.
$\|E_{\mc{J}}x\|\ls\sqrt{1-\dfrac{\ve}{m_{\mc{J}}|\mc{U}|}}\|x\|$. Поэтому
\begin{equation*}
\|Tx\|\ls(\sum_{\mc{I}\neq \mc{J}}\alpha_{\mc{I}}+\alpha_{\mc{J}}\sqrt{1-\dfrac{\ve}{m_{\mc{J}}|\mc{U}|}})\|x\|.
\end{equation*}
Поскольку $\mc{J}=\mc{J}(x)\in\mc{U}$, то
$\|T\|<\sum_{\mc{I}\in\mc{U}}\alpha_{\mc{I}}$, а поэтому оператор
$(\sum_{\mc{I}\in\mc{U}}\alpha_{\mc{I}})I-T$ обратим, откуда следует нужное.

\textbf{2.}
Пусть теперь
$\tmd{Im}((\sum_{\mc{I}\in\mc{U}}\alpha_{\mc{I}})I-T)$ "--- подпространство.
Обозначим
$A=(\sum_{\mc{I}\in\mc{U}}\alpha_{\mc{I}})I-T$.
Покажем, что $\ker(A^{*})=0$. Действительно, пусть
$x\in\ker A^{*}$. Поскольку все
$\alpha_{\mc{I}}>0,\,\mc{I}\in\mc{U}$, то для всякого
$\mc{I}\in\mc{U}$ $\|E_{\mc{I}^{*}} x\|=\|x\|$. Из леммы
\ref{L:lemma5} получим $x\in\bigcap_{i\in \mc{I}}\ker T_{i}$.
Поэтому
$x\in\bigcap_{i=1}^{n}\ker T_{i}$, т.е. $x=0$. Итак,
$\ker A^{*}=0$, поэтому
$\tmd{Im}(A)=H$. Поскольку $A$ принадлежит алгебре, порожденной $T_{1},\ldots,T_{n}$, то
$\tmd{Im}(T_{1})+\ldots+\tmd{Im}(T_{n})=H$.
\end{proof}

После этого утверждения естественно возникает вопрос:
следует ли из равенства
$\tmd{Im}(T_{1})+\ldots+\tmd{Im}(T_{n})=
\tmd{Im}((\sum_{\mc{I}\in\mc{U}}\alpha_{\mc{I}})I-T)$
замкнутость
$\tmd{Im}(T_{1})+\ldots+\tmd{Im}(T_{n})$?
Всегда
$\tmd{Im}(T_{1})=\tmd{Im}(I-(I-T_{1}))$, но
$\tmd{Im}(T_{1})$ не обязательно замкнут. Поэтому естественно наложить на
операторы $T_{i},\,1\ls i\ls n$ дополнительное условие, состоящее в том, что
$\tmd{Im}(T_{i})$ замкнут для всех
$1\ls i\ls n$.

Для изучения этого вопроса нам понадобятся вспомогательные результаты.

\begin{lemma}\label{L:Images0}
Пусть
$M$ "--- гильбертово пространство,
$a,b\in\mc{B}(M)$,
$c=aba^{*}$. Если
$\tmd{Im}(c)=\tmd{Im}(a)$, то
$\tmd{Im}(c)\supset(\ker(c))^{\bot}$.
\end{lemma}
\begin{proof}
Из теоремы Р. Дугласа следует, что
$a=cd$, где
$\tmd{Im}(d)\subset(\ker(c))^{\bot}$. Тогда
$c=cdba^{*}$,
$c(I-dba^{*})=0$. Для произвольного
$x\in M$
$~x=(I-dba^{*})x+dba^{*}x$, при этом
$(I-dba^{*})x\in\ker(c)$,
$dba^{*}x\in(\ker(c))^{\bot}$. Поэтому
$dba^{*}=P_{(\ker(c))^{\bot}}$ "--- ортопроектор на
$(\ker(c))^{\bot}$. Поскольку ортопроектор самосопряжен, то
$ab^{*}d^{*}=P_{(\ker(c))^{\bot}}$, откуда
$(\ker(c))^{\bot}\subset\tmd{Im}(a)=\tmd{Im}(c)$.
\end{proof}

Для гильбертовых пространств
$H,H'$ обозначим
$\mc{B}(H,H')$ множество всех линейных непрерывных операторов из
$H$ в $H'$.
Пусть
$M_{1},\ldots,M_{n},M$ "--- гильбертовы пространства,
$a_{i}\in\mc{B}(M_{i},M)$ "--- набор линейных непрерывных операторов.
Определим
\begin{equation*}
\mc{M}(a_{1},\ldots,a_{n})=\left\{\sum_{1\ls i,j\ls n}a_{i}a_{i,j}a_{j}^{*},\,a_{i,j}\in\mc{B}(M_{j},M_{i}),\,
1\ls i,j\ls n\right\}\subset\mc{B}(M)
\end{equation*}
Отметим, что
$\mc{M}(a_{1},\ldots,a_{n})$ является $\ast$-алгеброй.
Из теоремы Р.Дугласа следует, что если набор операторов
$\wt{a}_{k}\in\mc{B}(\wt{M}_{k},M),\,1\ls k\ls n$ такой, что
$\tmd{Im}(\wt{a}_{k})=\tmd{Im}(a_{k}),\,1\ls k\ls n$, то
$\mc{M}(a_{1},\ldots,a_{n})=
\mc{M}(\wt{a}_{1},\ldots,\wt{a}_{n})$.

\begin{lemma}\label{L:Images}
Пусть оператор
$a\in\mc{M}(a_{1},\ldots,a_{n})$ и
$\tmd{Im}(a)=\sum_{k=1}^{n}\tmd{Im}(a_{k})$.
Тогда
$\tmd{Im}(a)\supset(\ker a)^{\bot}$.
\end{lemma}
\begin{proof}
Определим
$b=(\sum_{k=1}^{n}a_{k}a_{k}^{*})^{1/2}$, тогда
$\tmd{Im}(b)=\sum_{k=1}^{n}\tmd{Im}(a_{k})\supset\tmd{Im}(a_{i})$.
Из теоремы Р. Дугласа следует, что
$a_{i}=bb_{i},\,1\ls i\ls n$, где
$b_{i}\in\mc{B}(M_{i},M)$. Тогда
$a_{i}^{*}=b_{i}^{*}b$. Поэтому
$a=b(\sum_{i,j}b_{i}a_{i,j}b_{j}^{*})b$. Из леммы
\ref{L:Images0} следует
$\tmd{Im}(a)\supset(\ker(a))^{\bot}$.
\end{proof}

\begin{conclusion}
Пусть в условиях леммы \ref{L:Images}
$\ker a=\ker a^{*}$ (это выполнено, например, если
$a$ нормальный). Тогда
$\tmd{Im}(a)=\sum_{k=1}^{n}\tmd{Im}(a_{k})$ замкнут.
\end{conclusion}

\begin{statement}
Пусть
$\sum_{k=1}^{n}\tmd{Im}(T_{k})=\tmd{Im}((\sum_{\mc{I}\in\mc{U}}\alpha_{\mc{I}})I-T)$
и для всех
$1\ls k\ls n$ образ $\tmd{Im}(T_{k})$ замкнут.
Тогда
$\sum_{k=1}^{n}\tmd{Im}(T_{k})$ замкнуто.
\end{statement}
\begin{proof}
Пусть
$A=(\sum_{\mc{I}\in\mc{U}}\alpha_{\mc{I}})I-T$. Поскольку
$\tmd{Im}(T_{k})$ замкнут, то $\tmd{Im}(T_{k})=\tmd{Im}(T_{k}^{\frac{1}{2}})$.
Ясно, что
$A\in\mc{M}(T_{1}^{\frac{1}{2}},\ldots,T_{n}^{\frac{1}{2}})$.
Из леммы \ref{L:Images} следует, что $\tmd{Im}(A)\supset(\ker(A))^{\bot}$.
Аналогично доказательству второй части утверждения \ref{ST:statement3}
можно показать, что
$\ker A=\bigcap_{k=1}^{n}\ker T_{k}=\tmd{Im}(A)^{\bot}$. Поэтому
$\tmd{Im}(A)\supset\ol{\tmd{Im}(A)}$, что
означает замкнутость $\tmd{Im}(A)$.
\end{proof}

\begin{example}
Пусть
$H_{1},\ldots,H_{n}$ "--- подпространства $H$,
$P_{1},\ldots,P_{n}$ "--- соответствующие ортопроекторы.
Следующие утверждения равносильны:
\begin{enumerate}
\item
$\tmd{Im}(I-(I-P_{n})\ldots(I-P_{1}))=\sum_{k=1}^{n}H_{k}$,
\item
$\sum_{k=1}^{n}H_{k}$ замкнуто.
\end{enumerate}
\end{example}

Используя лемму \ref{L:Images}, можем сформулировать
следующий критерий замкнутости суммы образов операторов.

\begin{statement}
Пусть
$a_{k}\in\mc{B}(M_{k},M),\,1\ls k\ls n$. Сумма
$\sum_{k=1}^{n}\tmd{Im}(a_{k})$ замкнута тогда и только тогда, когда
$(a_{1}a_{1}^{*}+\ldots+a_{n}a_{n}^{*})^{1/2}\in\mc{M}(a_{1},\ldots,a_{n})$.
\end{statement}
\begin{proof}
\textbf{1.}
Пусть сначала
$a=(a_{1}a_{1}^{*}+\ldots+a_{n}a_{n}^{*})^{1/2}\in\mc{M}(a_{1},\ldots,a_{n})$.
Поскольку
$\tmd{Im}(a)=\sum_{k=1}^{n}\tmd{Im}(a_{k})$ и
$\ker a=(\tmd{Im}(a))^{\bot}$ (поскольку $a$ самосопряжен), то из леммы
\ref{L:Images} получим $\tmd{Im}(a)\supset\ol{\tmd{Im}(a)}$, т.е.
$\tmd{Im}(a)$ "--- подпространство.

\textbf{2.}
Пусть теперь сумма
$\sum_{k=1}^{n}\tmd{Im}(a_{k})$ "--- подпространство.
Рассмотрев операторы $b_{k}=(a_{k}a_{k}^{*})^{1/2}$, можем сразу считать, что
операторы
$a_{k}\in\mc{B}(M)$, самосопряжены и неотрицательны.
Кроме того, можно считать, что
$\sum_{k=1}^{n}\tmd{Im}(a_{k})=H$. Тогда оператор
$a_{1}^{2}+\ldots+a_{n}^{2}$ обратим. Поэтому
\begin{align*}
&(a_{1}^{2}+\ldots+a_{n}^{2})^{\frac{1}{2}}=(a_{1}^{2}+\ldots+a_{n}^{2})
(a_{1}^{2}+\ldots+a_{n}^{2})^{-\frac{3}{2}}(a_{1}^{2}+\ldots+a_{n}^{2})=\\
&=\sum_{1\ls i,j\ls n}a_{i}^{2}(a_{1}^{2}+\ldots+a_{n}^{2})^{-\frac{3}{2}}a_{j}^{2}\in\mc{M}(a_{1},\ldots,a_{n}).
\end{align*}
\end{proof}

\begin{conclusion}
Пусть
$H_{1},\ldots,H_{n}$ "--- подпространства
$H$, $P_{1},\ldots,P_{n}$ "--- соответствующие ортопроекторы.
$\sum_{k=1}^{n}H_{k}$ замкнуто тогда и только тогда, когда
$(\sum_{k=1}^{n}P_{k})^{1/2}\in\mc{M}(P_{1},\ldots,P_{n})$.
\end{conclusion}

\subsection{Критерии замкнутости суммы подпространств
в терминах $\tmd{Im}(A)$, где
$A=\sum_{i,j}\alpha_{i,j}P_{i}P_{j},\,\alpha_{i,j}\in\mathbb{C}$}

Пусть
$H_{1},\ldots,H_{n}$ "--- подпространства $H$ с соответствующими ортопроекторами
$P_{1},\ldots,P_{n}$. Пусть сначала
$n=2$. Из утверждения
\ref{ST:functionpair3} следует, что если
$A\in\mc{A}(P_{1},P_{2})$ (алгебра, порожденная ортопроекторами $P_{1},P_{2}$) и
$\tmd{Im}(A)=H_{1}+H_{2}$, то
$H_{1}+H_{2}$ замкнуто.
При
$n\gs 3$ аналогичное утверждение, вообще говоря, неверно.

\begin{example}
Пусть
$A=P_{1}P_{2}+P_{3}$. Мы построим подпространства
$H_{1},H_{2},H_{3}$, такие, что\\
$\textbf{(Cond1)}$
$\tmd{Im}(A)=H_{1}+H_{2}+H_{3}$\\
$\textbf{(Cond2)}$
$H_{1}+H_{2}+H_{3}$ не замкнуто.

$\textbf{(Cond1)}$
равносильно
$H_{2}\subset H_{1}+H_{3}$ и
$\tmd{Im}(A)=H_{1}+H_{3}$. Для выполнения
$\tmd{Im}(A)=H_{1}+H_{3}$ достаточно, чтобы пара подпространств
$H_{1},H_{3}$ обладала свойством обратного наилучшего приближения относительно пары операторов
$P_{1}P_{2},P_{3}$
(см. подраздел \ref{SS:IBAP}).
Используя утверждение
\ref{ST:IBAP1}, получим, что для выполнения
$\textbf{(Cond1),(Cond2)}$ достаточно, чтобы:
\begin{enumerate}
\item
$H_{2}\subset H_{1}+H_{3}$,
\item
существует
$\ve>0$, такое, что
$\|P_{2}x\|\gs\ve\|x\|,\,x\in H_{1}$,
\item
подпространства
$P_{2}(H_{1}),H_{3}$ линейно независимы и их сумма замкнута,
\item
$H_{1}+H_{3}$ не замкнуто.
\end{enumerate}

Пусть операторы
$a,b:l_{2}\to l_{2}$ удовлетворяют следующим условиям:
\begin{enumerate}
\item
$a$ самосопряжен и $\frac{1}{4}I\ls a\ls I$,
\item
$1\in\sigma(a)$ и $\ker(I-a)=0$,
\item
$E([1/4,3/4))l_{2}$ бесконечномерно, тут
$E(\cdot)$ обозначена спектральная проекторнозначная мера
$a$,
\item
$(I-a^{-1/2}b)$ "--- изометрия и
$\tmd{Im}(I-a^{-1/2}b)=E([1/4,3/4))l_{2}$.
\end{enumerate}

Определим
$H=l_{2}\oplus l_{2}$,
$H_{1}=\{(x,0),\,x\in l_{2}\}$,
$H_{2}=\{(\sqrt{a}x,\sqrt{I-a}x),\,x\in l_{2}\}$,
$H_{3}=\{(bx,\sqrt{I-a}x),\,x\in l_{2}\}$.

Проверим, что
$H_{3}$ "--- подпространство.
Для этого достаточно показать, что существует
$\ve_{1}>0$, такое, что для всякого
$x\in l_{2}$
$\|bx\|+\|\sqrt{I-a}x\|\gs\ve_{1}\|x\|$.
Для этого достаточно доказать, что если
$x_{n}\in l_{2}$,
$bx_{n}\to 0$,
$\sqrt{I-a}x_{n}\to 0$ то
$x_{n}\to 0$. Докажем это. Имеем:
$-a^{-1/2}bx_{n}\to 0$,
$(I-a^{-1/2}b)x_{n}-x_{n}\to 0$,
$\sqrt{I-a}(I-a^{-1/2}b)x_{n}-\sqrt{I-a}x_{n}\to 0$,
$\sqrt{I-a}(I-a^{-1/2}b)x_{n}\to 0$.
Поскольку
$\|\sqrt{I-a}z\|\gs\frac{1}{2}\|z\|$ для
$z\in E([1/4,3/4))l_{2}$, то
$(I-a^{-1/2}b)x_{n}\to 0$,
$x_{n}\to 0$.
Итак,
$H_{3}$ "--- подпространство.

Поскольку $H_{1}+H_{3}=\{(x,\sqrt{I-a}y,\,x,y\in l_{2})\}$, то
$H_{2}\subset H_{1}+H_{3}$ и
$H_{1}+H_{3}$ не замкнуто.

Пусть
$x=(y,0)\in H_{1}$. Поскольку
$P_{2}=\begin{pmatrix}
a&\sqrt{a(I-a)}\\
\sqrt{a(I-a)}&I-a
\end{pmatrix}$, то
$\|P_{2}x\|^{2}=(ay,y)\gs\frac{1}{4}\|y\|^{2}=\frac{1}{4}\|x\|^{2}$,
$\|P_{2}x\|\gs(1/2)\|x\|$. Легко видеть, что
$P_{2}(H_{1})=H_{2}$.

Покажем, что
$H_{2},H_{3}$ линейно независимы и их сумма замкнута.
Для этого достаточно доказать, что если
$X_{n}\in H_{2},Y_{n}\in H_{3}$ и
$X_{n}+Y_{n}\to 0$, то
$X_{n}\to 0$.
Докажем это. Пусть
$X_{n}=(\sqrt{a}x_{n},\sqrt{I-a}x_{n})$,
$Y_{n}=(by_{n},\sqrt{I-a}y_{n})$.
Тогда
$\sqrt{a}x_{n}+by_{n}\to 0$ и
$\sqrt{I-a}(x_{n}+y_{n})\to 0$.
Поэтому
$x_{n}+a^{-1/2}by_{n}\to 0$,
$\sqrt{I-a}(I-a^{-1/2}b)y_{n}\to 0$,
$(I-a^{-1/2}b)y_{n}\to 0$. Поскольку
$I-a^{-1/2}b$ изометрия, то
$y_{n}\to 0$,
$x_{n}\to 0$.

Итак,
$H_{1},H_{2},H_{3}$
удовлетворяют условиям
$\textbf{(Cond1),(Cond2)}$.
\end{example}

Далее мы будем рассматривать операторы вида
$A=\sum_{i,j=1}^{n}\alpha_{i,j}P_{i}P_{j}$, где
$\alpha_{i,j}\in\mathbb{C}$. Как показывает предыдущий пример, из
$\tmd{Im}(A)=\sum_{k=1}^{n}H_{k}$ вообще говоря не следует замкнутость
$\sum_{k=1}^{n}H_{k}$. Мы будем накладывать условия на
$\alpha_{i,j}$, чтобы это стало верным.
Далее
$\alpha_{i,i}>0,\,i=1,2,\ldots,n$. Определим
$\beta_{i,i}=\alpha_{i,i}$ и
\begin{equation*}
\beta_{i,j}=-\frac{1}{2}\sqrt{(\tmd{Re}\alpha_{i,j}+\tmd{Re}\alpha_{j,i})^{2}+(\tmd{Im}\alpha_{i,j}-\tmd{Im}\alpha_{j,i})^{2}},\,i\neq j.
\end{equation*}

Для произвольного
$x\in H$
\begin{align}\label{E:Images1}
&\tmd{Re}(Ax,x)=\sum_{i=1}^{n}\alpha_{i,i}\|P_{i}x\|^{2}+\sum_{i<j}\tmd{Re}(\alpha_{i,j}(P_{j}x,P_{i}x)+\alpha_{j,i}(P_{i}x,P_{j}x))\gs\\
&\gs\sum_{i=1}^{n}\alpha_{i,i}\|P_{i}x\|^{2}+2\sum_{i<j}\beta_{i,j}|(P_{i}x,P_{j}x)|\gs\sum_{i,j}\beta_{i,j}\|P_{i}x\|\|P_{j}x\|.\notag
\end{align}

\begin{statement}\label{ST:Images1}
Пусть матрица $B=(\beta_{i,j})$ положительно определена. Тогда условия равносильны:
\begin{enumerate}
\item
$\sum_{k=1}^{n}H_{k}$ замкнута,
\item
$\tmd{Im}(A)=\sum_{k=1}^{n}H_{k}$,
\item
$\tmd{Im}(A)$ замкнут.
\end{enumerate}
\end{statement}
\begin{proof}
Можно считать, что
$\sum_{k=1}^{n}H_{k}$ плотна в
$H$. Из неравенств
(\ref{E:Images1}) и положительной определённости
$B$ следует, что существует
$\ve_{1}>0$, такое, что для произвольного
$x\in H$
\begin{equation*}
\tmd{Re}(Ax,x)\gs\ve_{1}\sum_{k=1}^{n}\|P_{k}x\|^{2}.
\end{equation*}

\textbf{$(1)\Rightarrow(2),(3)$}
Поскольку
$\sum_{k=1}^{n}H_{k}=H$, то существует
$\ve_{2}>0$, такое, что
$\sum_{k=1}^{n}P_{k}\gs\ve_{2}I$.
Тогда
$\tmd{Re}(Ax,x)\gs\ve_{1}\ve_{2}\|x\|^{2}$, поэтому
$A$ обратим. В частности,
$\tmd{Im}(A)=H$.

\textbf{$(2)\Rightarrow(1)$}
Поскольку
$A\in\mc{M}(P_{1},\ldots,P_{n})$, то из леммы
\ref{L:Images} следует
$\tmd{Im}(A)\supset(\ker(A))^{\bot}$.
Пусть
$x\in\ker(A)$, тогда
$Ax=0$,
$\tmd{Re}(Ax,x)=0$. Поэтому
$P_{k}x=0,\,k=1,2,\ldots,n$, откуда
$x=0$. Итак,
$\ker(A)=0$, поэтому
$\tmd{Im}(A)=H$ и
$\sum_{k=1}^{n}H_{k}=H$.

\textbf{$(3)\Rightarrow(1)$}
В этом случае
$\tmd{Im}(A)=(\ker(A^{*}))^{\bot}$.
Пусть
$x\in\ker(A^{*})$, тогда
$A^{*}x=0$,
$(A^{*}x,x)=0$,
$(Ax,x)=0$,
$\tmd{Re}(Ax,x)=0$. Поэтому
$x=0$. Итак,
$\ker(A^{*})=0$, поэтому
$\tmd{Im}(A)=H$,
$\sum_{k=1}^{n}H_{k}=H$.
\end{proof}

Теперь мы откажемся от положительной определённости
$B$. Далее предполагаем, что выполнены следующие условия:\\
\textbf{(B1)}
$B$ неотрицательно определена,
$0\in\sigma(B)$ и имеет кратность
$1$,\\
\textbf{(B2)}
построим граф
$\Gamma_{\beta}$ с множеством вершин
$1,2,\ldots,n$ следующим образом:
$i$ соединено с
$j$ если
$\beta_{i,j}\neq 0$. Граф
$\Gamma_{\beta}$ связен.

Пусть вектор
$s=(s_{1},\ldots,s_{n})\in\mathbb{R}^{n}$
является базисом
$\ker(B)$. Тогда
$(Bs,s)=0$. Определим
$s'=(|s_{1}|,\ldots,|s_{n}|)$.
Поскольку
$\beta_{i,i}>0$ и
$\beta_{i,j}\ls 0$ при
$i\neq j$, то
$(Bs',s')\ls(Bs,s)=0$. Из неотрицательной определённости
$B$ следует
$(Bs',s')=0$,
$Bs'=0$. Поэтому можем считать, что
$s_{1},\ldots,s_{n}$ неотрицательны.
Поскольку
$Bs=0$, то
$\sum_{j=1}^{n}\beta_{i,j}s_{j}=0$,
$\alpha_{i,i}s_{i}=\sum_{j\neq i}(-\beta_{i,j})s_{j}$.
Поэтому если
$s_{i}=0$ и
$\beta_{i,j}\neq 0$ (т.е. $\beta_{i,j}<0$), то
$s_{j}=0$. Предположив, что для некоторого $i$
$s_{i}=0$, из связности
$\Gamma_{\beta}$ получим
$s_{1}=\ldots=s_{n}=0$, противоречие. Итак,
$s_{1},\ldots,s_{n}>0$.

\begin{remark}
Легко показать, что если выполнено
$\textbf{(B1)}$ и
$s_{1},\ldots,s_{n}>0$, то выполнено
$\textbf{(B2)}$.
\end{remark}

\begin{statement}\label{ST:Images2}
Пусть выполнены
$\textbf{(B1),(B2)}$. Тогда:
\begin{enumerate}
\item
если
$\tmd{Im}(A)=\sum_{k=1}^{n}H_{k}$, то
$\sum_{k=1}^{n}H_{k}$ замкнуто,
\item
если
$\tmd{Im}(A)$ замкнут, то $\sum_{k=1}^{n}H_{k}$ замкнуто,
\item
если
$\sum_{k=1}^{n}H_{k}=\sum_{k=1}^{n}H_{k}^{\bot}=H$, то
$A$ обратим.
\end{enumerate}
\end{statement}

\begin{lemma}\label{L:Images2}
Пусть выполнены условия
$\textbf{(B1),(B2)}$ и
$\bigcap_{k=1}^{n}H_{k}=\bigcap_{k=1}^{n}H_{k}^{\bot}=0$. Тогда
$\ker(A)=\ker(A^{*})=0$.
\end{lemma}
\begin{proof}
Пусть
$x\in\ker(A)$. Тогда
$Ax=0$,
$\tmd{Re}(Ax,x)=0$. Из неравенств
(\ref{E:Images1}) и неотрицательной определённости
$B$ следует, что в неравенствах
(\ref{E:Images1}) везде равенства и
$\sum_{i,j}\beta_{i,j}\|P_{i}x\|\|P_{j}x\|=0$.
Поэтому
$\|P_{k}x\|=\lambda s_{k},\,1\ls k\ls n$ для некоторого
$\lambda\gs 0$. Если
$\lambda=0$, то
$P_{k}x=0,\,1\ls k\ls n$, откуда
$x=0$. Предположим, что
$\lambda>0$. Тогда
$P_{k}x\neq 0,\,1\ls k\ls n$.
Поскольку в неравенствах
(\ref{E:Images1}) достигаются равенства, то если
$\beta_{i,j}\neq 0$, то
$|(P_{i}x,P_{j}x)|=\|P_{i}x\|\|P_{j}x\|$, существует
$\lambda_{i,j}\in\mathbb{C}$ такое, что
$P_{i}x=\lambda_{i,j}P_{j}x$. Поскольку
$\Gamma_{\beta}$ связен, то
$P_{1}x=\lambda_{2}P_{2}x=\ldots=\lambda_{n}P_{n}x$
для некоторых комплексных
$\lambda_{2},\ldots,\lambda_{n}$. Поскольку
$\bigcap_{k=1}^{n}H_{k}=0$, то
$P_{1}x=0$, противоречие с
$\lambda>0$.
Итак,
$\ker(A)=0$.

Докажем, что
$\ker(A^{*})=0$. Пусть
$x\in\ker(A^{*})$, тогда
$\tmd{Re}(Ax,x)=0$ и повторяя предыдущие рассуждения получим
$x=0$.
\end{proof}

\begin{proof}(утверждения \ref{ST:Images2})
Можно считать, что
$\bigcap_{k=1}^{n}H_{k}=\bigcap_{k=1}^{n}H_{k}^{\bot}=0$.

Докажем
$(1)$. Поскольку
$A\in\mc{M}(P_{1},\ldots,P_{n})$, из леммы
\ref{L:Images} следует
$\tmd{Im}(A)\supset(\ker(A))^{\bot}$. Из леммы
\ref{L:Images2} получаем
$\ker(A)=0$. Поэтому
$\tmd{Im}(A)=H$, откуда
$\sum_{k=1}^{n}H_{k}=H$.

Докажем $(2)$.
Имеем:
$\tmd{Im}(A)=(\ker(A^{*}))^{\bot}=H$, откуда
$\sum_{k=1}^{n}H_{k}=H$.

Докажем
$(3)$. Пусть
$\gamma_{i,j}=-\beta_{i,j},\,i\neq j$. Поскольку
$Bs=0$, то
$\sum_{j=1}^{n}\beta_{i,j}s_{j}=0$,
$\alpha_{i,i}=\sum_{j\neq i}\gamma_{i,j}\dfrac{s_{j}}{s_{i}}$.
Из неравенств
(\ref{E:Images1}) для произвольного
$x\in H$ имеем:
\begin{align}\label{E:Images2}
&\tmd{Re}(Ax,x)\gs\sum_{i=1}^{n}\alpha_{i,i}\|P_{i}x\|^{2}-2\sum_{i<j}\gamma_{i,j}|(P_{i}x,P_{j}x)|=\\
&\sum_{i=1}^{n}\left(\sum_{j\neq i}\gamma_{i,j}s_{i}s_{j}\right)\|P_{i}x/s_{i}\|^{2}
-2\sum_{i<j}\gamma_{i,j}s_{i}s_{j}|(P_{i}x/s_{i},P_{j}x/s_{j})|.\notag
\end{align}
У нас
$\sum_{k=1}^{n}H_{k}^{\bot}=H$,
$s_{1},\ldots,s_{n}>0$,
$\Gamma_{\beta}$ связен. Поэтому из утверждения
\ref{ST:orth2} следует, что существует
$\ve_{1}>0$, такое, что для произвольных
$y_{k}\in H_{k},\,1\ls k\ls n$
\begin{equation*}
2\sum_{i<j}\gamma_{i,j}s_{i}s_{j}|(y_{i},y_{j})|\ls\sum_{i=1}^{n}\left(\sum_{j\neq i}\gamma_{i,j}s_{i}s_{j}-\ve_{1}\right)\|y_{i}\|^{2}.
\end{equation*}
Подставив в это неравенство
$y_{i}=P_{i}x/s_{i},\,1\ls i\ls n$, из неравенства
(\ref{E:Images2}) получим
$\tmd{Re}(Ax,x)\gs\ve_{1}\sum_{i=1}^{n}\|P_{i}x/s_{i}\|^{2}\gs\ve_{2}\sum_{i=1}^{n}\|P_{i}x\|^{2}$ для некоторого
$\ve_{2}>0$. Поскольку
$\sum_{k=1}^{n}H_{k}=H$, то
$\tmd{Re}(Ax,x)\gs\ve_{3}\|x\|^{2}$ для некоторого
$\ve_{3}>0$. Отсюда следует, что
$A$ обратим.
\end{proof}

Рассмотрим применения утверждения
\ref{ST:Images2}.
Пусть $\Gamma$ "--- связный граф с множеством вершин
$1,2,\ldots,n$. Будем писать
$i\sim j$ если
$i$ соединено с $j$ в $\Gamma$.
Пусть каждой паре вершин
$i,j$, соединённых ребром в
$\Gamma$, сопоставлены действительные числа
$\xi_{i,j},\xi_{j,i}$, причём
$\xi_{i,j}+\xi_{j,i}>0$. Для каждого
$i=1,2,\ldots,n$ определим
$\xi_{i}=(1/2)\sum_{j\sim i}(\xi_{i,j}+\xi_{j,i})$. Ясно, что
$\xi_{i}>0$ и
$\sum_{i=1}^{n}\xi_{i}=\sum_{i\sim j}\xi_{i,j}$. Определим
$A=\sum_{i=1}^{n}\xi_{i}P_{i}-\sum_{i\sim j}\xi_{i,j}P_{i}P_{j}$.

\begin{example}
Пусть
$E(\Gamma)=\{\{1,2\},\{2,3\},\ldots,\{n,1\}\}$, т.е.
$\Gamma$ "--- цикл. Пусть
$\xi_{1,2}=\xi_{2,3}=\ldots=\xi_{n,1}=1$,
$\xi_{2,1}=\xi_{3,2}=\ldots=\xi_{1,n}=0$.
Тогда
$\xi_{1}=\ldots=\xi_{n}=1$ и
$A=\sum_{i=1}^{n}P_{i}-\sum_{i=1}^{n}P_{i}P_{i+1}$, где обозначено
$P_{n+1}=P_{1}$.
\end{example}

В предыдущих обозначениях оператор
$A$ получается при выборе
$\alpha_{i,j}$ таким образом:
\begin{enumerate}
\item
$\alpha_{i,i}=\xi_{i},\,i=1,2,\ldots,n$,
\item
$\alpha_{i,j}=-\xi_{i,j}$ если $i\sim j$,
\item
$\alpha_{i,j}=0$ если $i\neq j$ и $i$ не соединено с $j$ в $\Gamma$.
\end{enumerate}
Тогда для
$\beta_{i,j}$ имеем:
\begin{enumerate}
\item
$\beta_{i,i}=\xi_{i},\,i=1,2,\ldots,n$,
\item
$\beta_{i,j}=-\frac{1}{2}(\xi_{i,j}+\xi_{j,i})$ если $i\sim j$,
\item
$\beta_{i,j}=0$ если $i\neq j$ и $i$ не соединено с $j$ в $\Gamma$.
\end{enumerate}
На векторе
$t=(t_{1},\ldots,t_{n})\in\mathbb{R}^{n}$ квадратичная форма, порождённая
$B$, равна
\begin{equation*}
(Bt,t)=\sum_{i=1}^{n}\xi_{i}t_{i}^{2}-\sum_{\{i,j\}\in E(\Gamma)}(\xi_{i,j}+\xi_{j,i})t_{i}t_{j}=
\sum_{\{i,j\}\in E(\Gamma)}\frac{1}{2}(\xi_{i,j}+\xi_{j,i})(t_{i}-t_{j})^{2}.
\end{equation*}
Поэтому
$B$ неотрицательно определена,
$0\in\sigma(B)$ имеет кратность $1$, соответствующий собственный вектор:
$(1,1,\ldots,1)$ (это следует из связности $\Gamma$). Также ясно, что
$\Gamma_{\beta}=\Gamma$. Таким образом, условия
$\textbf{(B1),(B2)}$ выполнены.

\begin{statement}
Следующие утверждения равносильны:
\begin{enumerate}
\item
$\sum_{k=1}^{n}H_{k},\sum_{k=1}^{n}H_{k}^{\bot}$ замкнуты,
\item
$\tmd{Im}(A)=(\sum_{k=1}^{n}H_{k})\bigcap(\bigcap_{k=1}^{n}H_{k})^{\bot}$,
\item
$\tmd{Im}(A)$ замкнут.
\end{enumerate}
\end{statement}
\begin{proof}
Разложим
$H=\bigcap_{k=1}^{n}H_{k}^{\bot}\oplus\bigcap_{k=1}^{n}H_{k}\oplus\wt{H}$, тогда
$H_{k}=0\oplus\bigcap_{k=1}^{n}H_{k}\oplus\wt{H}_{k}$, где
$\wt{H}_{k}$ "--- подпространство $\wt{H}$.
Ясно, что
$\bigcap_{k=1}^{n}\wt{H}_{k}^{\bot}=\bigcap_{k=1}^{n}\wt{H}_{k}=0$
(тут $\wt{H}_{k}^{\bot}$ обозначает ортогональное дополнение
$\wt{H}_{k}$ в $\wt{H}$). Обозначим
$\wt{P}_{k}$ ортопроектор на
$\wt{H}_{k}$ в пространстве
$\wt{H}$. Тогда
$P_{k}=0\oplus I\oplus\wt{P}_{k}$. Поскольку
$\sum_{i=1}^{n}\xi_{i}=\sum_{i\sim j}\xi_{i,j}$, то
$A=0\oplus 0\oplus\wt{A}$, где
$\wt{A}=\sum_{i=1}^{n}\xi_{i}\wt{P}_{i}-\sum_{i\sim j}\xi_{i,j}\wt{P}_{i}\wt{P}_{j}$.
Таким образом, достаточно доказать нужное утверждение для подпространств
$\wt{H}_{k}$ пространства $\wt{H}$.
Поэтому мы сразу будем считать, что
$\bigcap_{k=1}^{n}H_{k}=\bigcap_{k=1}^{n}H_{k}^{\bot}=0$.

\textbf{$(1)\Rightarrow(2),(3)$}
Поскольку
$\sum_{k=1}^{n}H_{k}=\sum_{k=1}^{n}H_{k}^{\bot}=H$, то нужное следует из утверждения
\ref{ST:Images2}.

\textbf{$(2)\Rightarrow (1)$}
Имеем:
$\tmd{Im}(A)=\sum_{k=1}^{n}H_{k}$. Из утверждения
\ref{ST:Images2} следует замкнутость
$\sum_{k=1}^{n}H_{k}$, поэтому
$\sum_{k=1}^{n}H_{k}=H$. Поэтому
$\tmd{Im}(A)=H$. Для
$i=1,2,\ldots,n$ обозначим $Q_{i}$ ортопроектор на
$H_{i}^{\bot}$. Тогда
$P_{i}=I-Q_{i}$. Поэтому
\begin{equation*}
A=\sum_{i=1}^{n}\xi_{i}(I-Q_{i})-\sum_{i\sim j}\xi_{i,j}(I-Q_{i})(I-Q_{j}).
\end{equation*}
Поскольку
$\sum_{i=1}^{n}\xi_{i}=\sum_{i\sim j}\xi_{i,j}$, то при раскрытии скобок слагаемые, кратные
$I$, сократятся, а поэтому
$\tmd{Im}(A)\subset\sum_{k=1}^{n}H_{k}^{\bot}$. Отсюда следует, что
$\sum_{k=1}^{n}H_{k}^{\bot}=H$.

\textbf{$(3)\Rightarrow(1)$}
Из леммы
\ref{L:Images2} следует
$\ker(A^{*})=0$, поэтому
$\tmd{Im}(A)=H$. Отсюда
$\sum_{k=1}^{n}H_{k}=H$. Повторив рассуджения при доказательстве
\textbf{$(2)\Rightarrow(1)$}, получим
$\sum_{k=1}^{n}H_{k}^{\bot}=H$.
\end{proof}

\subsection{Сумма $n$ подпространств представима в виде суммы пары подпространств}\label{SS:sumnequalsum2}

Рассмотрим следующий вопрос: пусть $H_{1},\ldots,H_{n}$ "---
подпространства
$H$. Можно ли сумму $H_{1}+\ldots+H_{n}$ представить в виде суммы двух подпространств, т.е.
существуют ли подпространства $M_{1},M_{2}$, для которых
$H_{1}+\ldots+H_{n}=M_{1}+M_{2}$ ?
В случае сепарабельного пространства
$H$ утвердительный ответ следует из теоремы 2.6 работы \cite{Fillmore}.
Однако упомянутая теорема неверна в случае несепарабельного пространства
(см. замечание после её доказательства).
Тем не менее,
мы покажем, что ответ утвердительный в случае произвольного гильбертова пространства
$H$ (отметим, что, например, гильбертово пространство почти периодических функций на $\mathbb{R}$,
имеющее размерность континуум, имеет многочисленные применения в теории дифференциальных
уравнений и математической физике).

\begin{theorem}\label{T:sumnequalsum2}
Для произвольных подпространств
$H_{1},\ldots,H_{n}$ существуют два подпространства
$M_{1},M_{2}$, такие, что
$H_{1}+\ldots+H_{n}=M_{1}+M_{2}$.
\end{theorem}

Для доказательства нам необходимы несколько лемм.
Следующая лемма мотивирована леммой
3.2 в \cite{Fillmore}.

\begin{lemma}\label{L:lemma2}
Пусть
$A,B$ "--- неотрицательные самосопряженные операторы,
$E_{A},E_{B}$ "--- их спектральные меры. Пусть
$\tmd{Im}(A)\subset\tmd{Im}(B)$. Тогда существует
$K>0$ такое, что для всех
$\alpha>0$
\begin{equation*}
\dim E_{A}([\alpha,+\infty))H\ls\dim E_{B}([\alpha/K,+\infty))H.
\end{equation*}
\end{lemma}
\begin{proof}
Достаточно доказать, что существует
$K>0$, для которого
\begin{equation*}
(E_{A}([\alpha,+\infty))H)\cap(E_{B}([\alpha/K,+\infty))H)^{\bot}=0
\end{equation*}
для всех
$\alpha>0$.
Из теоремы Р.Дугласа следует, что существует оператор
$C\in\mc{B}(H)$, для которого $A=BC$, а поэтому $A=C^{*}B$.
Пусть
$x\in E_{A}([\alpha,+\infty))H$. Тогда
$\|Ax\|\gs\alpha\|x\|$,
$\|C^{*}Bx\|\gs\alpha\|x\|$, а поэтому
$\|Bx\|\gs\dfrac{\alpha}{\|C\|}\|x\|$.
Поэтому нам подойдёт
$K=2\|C\|$.
\end{proof}

\begin{lemma}\label{L:lemma3}
Пусть
$A$ "--- неотрицательный самосопряженный оператор в $H$, причём
$\ker A=0$,
$E$ "--- его спектральная мера.
Линеал
$\tmd{Im}(A)$ можно представить в виде суммы пары подпространств тогда и только тогда, когда
\begin{equation*}
\dim E([\ve,+\infty))H=\dim H
\end{equation*}
для некоторого $\ve>0$.
\end{lemma}
\begin{proof}
Пусть сначала
$\dim E([\ve,+\infty))H=\dim H$ для некоторого $\ve>0$. Можем считать, что
$\ve<1$. Разложим
$H=E([0,\ve))H\oplus E([\ve,\infty))H$. Относительно этого ортогонального разложения
$A=A_{1}\oplus A_{2}$, где оператор
$0\ls A_{1}\ls\ve I$, а
$A_{2}$ обратим. Поэтому
$\tmd{Im}(A)=\tmd{Im}(A_{1})\oplus E([\ve,\infty))H$. Ясно, что
$\dim E([0,\ve))H\ls \dim E([\ve,\infty))H$.

Таким образом, нам достаточно доказать следующее:
в гильбертовом пространстве
$M\oplus M\oplus N$ линеал
$\tmd{Im}(B)\oplus M\oplus N$ (тут оператор
$B:M\to M$) можно представить в виде суммы пары подпространств $M_{1},M_{2}$.
Ясно, что
$M_{1}=\{(Bx,(I-B)x,0),\,x\in M\}$,
$M_{2}=0\bi M\bi N$ "--- искомые подпространства.

Пусть теперь
$\tmd{Im}(A)$ можно представить в виде суммы пары подпространств.
Если
$H$ конечномерно, то нужное утверждение очевидно, поэтому далее
$H$ бесконечномерно. Пусть
$\tmd{Im}(A)=H_{1}+H_{2}$,
$P_{1},P_{2}$ "--- соответствующие ортопроекторы.
Поскольку
$\ker A=0$, то $H_{1}^{\bot}\cap H_{2}^{\bot}=0$. Определим подпространства
$H_{1,1}=H_{1}\cap H_{2},\,H_{1,0}=H_{1}\cap H_{2}^{\bot},\,H_{0,1}=H_{1}^{\bot}\cap H_{2}$.
Применим для пары подпространств
$H_{1},H_{2}$ теорему
\ref{T:spectraltheorem}, используем её обозначения.
В ортогональном разложении
(\ref{E:rozkladH2}) компонента
$H_{1}^{\bot}\cap H_{2}^{\bot}$ отсутствует.

Тогда
\begin{equation*}
H_{1}+H_{2}=H_{1,1}\oplus H_{1,0}\oplus H_{0,1}\oplus K\oplus\tmd{Im}(\sqrt{I-a})
\end{equation*}
является образом оператора
$B=I_{H_{1,1}}\oplus I_{H_{1,0}}\oplus I_{H_{0,1}}\oplus I_{K}\oplus \sqrt{I-a}$. Из леммы
\ref{L:lemma2} ($\tmd{Im}(B)\subset\tmd{Im}(A)$) следует существование
$\ve>0$, для которого
$\dim E([\ve,\infty))H\gs\dim H_{1,1}+\dim H_{1,0}+\dim H_{0,1}+\dim K=\dim H$, т.е.
$\dim E([\ve,\infty))H=\dim H$.
\end{proof}

\begin{proof}(теоремы \ref{T:sumnequalsum2})
Можно считать, что
$\ol{H_{1}+\ldots+H_{n}}=H$.
В случае конечномерного
$H$ утверждение теоремы очевидно, поэтому далее $H$
бесконечномерно. Пусть размерность
$\dim H_{1}$ "--- наибольшая из $\dim H_{i},\,1\ls i\ls n$.
Тогда
$\dim H_{1}=\dim H$.
Определим оператор
$A=\sqrt{P_{1}+\ldots+P_{n}}$, пусть
$E$ "--- его спектральная мера. Поскольку
$\tmd{Im}(P_{1})\subset\tmd{Im}(A)$, то из леммы
\ref{L:lemma2} следует, что существует
$\ve>0$, для которого
$\dim E([\ve,+\infty))\gs\dim H_{1}=\dim H$. Поэтому
$\dim E([\ve,+\infty))=\dim H$. Теперь из леммы
\ref{L:lemma3} следует нужное утверждение.
\end{proof}

\subsection{Когда из замкнутости $H_{1}+\ldots+H_{n}$ следует замкнутость $H_{1}+\ldots+H_{m}$
($m<n$ фиксировано)}

\begin{statement}
Пусть
$M_{1},M_{2}$ "--- гильбертовы пространства,
$a:M_{1}\to H,\,b:M_{2}\to H$ "--- линейные непрерывные операторы, причём
$\tmd{Im}(a)\bigcap \tmd{Im}(b),\,\tmd{Im}(a)+\tmd{Im}(b)$ "--- подпространства
в
$H$. Тогда
$\tmd{Im}(a),\,\tmd{Im}(b)$ "--- подпространства.
\end{statement}
\begin{proof}
1. (см. теорему 2.3 в \cite{Fillmore})

Сначала докажем требуемое утверждение при условии $\,\tmd{Im}(a)\bigcap\tmd{Im}(b)=0$.
Без ограничения общности можно считать, что
$\ker(a)=0$ и $\ker(b)=0$. Определим оператор
$c:M_{1}\oplus M_{2}\to \tmd{Im}(a)+\tmd{Im}(b)$ равенством
$c(x,y)=ax+by$. Тогда
$c$ непрерывен и биективен, а потому обратим. Поэтому
существует
$\ve>0$ такое, что для произвольных
$x\in M_{1},\,y\in M_{2}$ выполнено:
$\|ax+by\|^{2}\gs\ve^{2}(\|x\|^{2}+\|y\|^{2}).$
Подставив $y=0$, получим
$\|ax\|\gs\ve\|x\|,\,x\in M_{1}$, а поэтому
$\tmd{Im}(a)$ "--- подпространство. Аналогично
$\tmd{Im}(b)$ "--- подпространство.

2. Теперь рассмотрим общий случай. Обозначим $M=\tmd{Im}(a)\cap\tmd{Im}(b)$.
Разложим
$M_{1}=a^{-1}(M)\oplus M_{1}^{\prime}$, и определим оператор
$a^{\prime}:M_{1}^{\prime}\to H$ равенством
$a^{\prime}x=ax,\,x\in M_{1}^{\prime}$. Тогда
$\tmd{Im}(a^{\prime})\cap\tmd{Im}(b)=0$ и
$\tmd{Im}(a^{\prime})+\tmd{Im}(b)=\tmd{Im}(a)+\tmd{Im}(b)$ "--- подпространство,
а поэтому
$\tmd{Im}(b)$ "--- подпространство. Аналогично
$\tmd{Im}(a)$ "--- подпространство.
\end{proof}

\begin{conclusion}
Пусть
$H_{1},\ldots,H_{n}$ "--- подпространства $H$, и
$H_{1}+\ldots+H_{n}$ "--- подпространство. Если для некоторого
$m<n$
$(H_{1}+\ldots+H_{m})\cap (H_{m+1}+\ldots+H_{n})$ "--- подпространство, то
$H_{1}+\ldots+H_{m},H_{m+1}+\ldots+H_{n}$ "--- подпространства.
\end{conclusion}

\begin{conclusion}
Пусть
$H_{1},\ldots,H_{n}$ "--- подпространства $H$, сумма которых
$H_{1}+\ldots+H_{n}$ замкнута. Предположим, что для любого
$I\subset \{1,2,\ldots,n\},\,|I|<n$ существует
$J\subset \{1,2,\ldots,n\},\,J\nsubseteq I$, такое, что
$(\sum_{i\in I}H_{i})\cap(\sum_{j\in J}H_{j})$ "--- подпространство.
Тогда для всякого
$I\subset\mathbb{N}_{n}$ сумма $\sum_{i\in I}H_{i}$ "--- подпространство.

Доказательство получается применением индукции "сверху вниз"\,  по
$|I|$, база индукции "--- $|I|=n$.
\end{conclusion}

\end{document}